\PassOptionsToPackage{dvipsnames}{xcolor}
\documentclass[12pt]{amsart}

\usepackage{microtype}

\usepackage{times}
\usepackage{pifont}
\usepackage{datetime}

\usepackage{amsmath,amssymb,commath,mathtools,mathabx,bbm}
\usepackage{xspace}

\usepackage[T2A,T1]{fontenc}
\usepackage[utf8]{inputenc}

\usepackage[backgroundcolor=yellow!20]{todonotes}
\usepackage{hyperref}
\usepackage[capitalize, noabbrev]{cleveref}

\usepackage{xcolor}
\usepackage{tikz}
\usetikzlibrary{decorations.pathreplacing, positioning, quotes,  decorations.markings}

\usepackage[normalem]{ulem}
\usepackage{soul}

\usepackage{enumitem}
\usepackage[labelfont={normalsize}]{caption,subfig}

\usepackage[backend=bibtex,url=false,style=alphabetic]{biblatex}
\usepackage{geometry}

\newtheorem{thm}{Theorem}[section]
\newtheorem{lem}[thm]{Lemma}
\newtheorem{propos}[thm]{Proposition}
\newtheorem{lemma}[thm]{Lemma}
\newtheorem{corollary}[thm]{Corollary}
\newtheorem{invariant}[thm]{Invariant}

\theoremstyle{remark}
\newtheorem{example}{Example}
\newtheorem{remark}[thm]{Remark}

\newcommand{\Sym}[1]{\mathfrak{S}_{#1}}

\newcommand{\Jump}{\mathbb{J}}
\newcommand{\Bend}{\mathbb{B}}

\newcommand{\Numberb}{\mathfrak{b}}
\newcommand{\CT}{\mathcal{A}}

\newcommand{\trokowska}{the first named author\xspace}

\newcommand{\sniady}{the second named author\xspace}
\newcommand{\Sniady}{The second named author\xspace}

\newcommand{\Root}{\alpha}

\DeclareMathOperator{\Ch}{Ch}

\begin{document}

\graphicspath{{FIGURES}}

\title[Bijection between trees and factorizations]{Bijection between trees \\
    in Stanley character formula \\ and factorizations of a cycle}

\author{Karolina Trokowska}
\email{karolina5284@wp.pl}

\author[Piotr {\'S}niady]{Piotr {\'S}niady} 
\address{Institute of Mathematics,
     Polish Academy of Sciences,
        {\'S}niadeckich~8, 00-656 Warszawa, Poland}
\address{Max-Planck-Institut für Mathematik - Bonn, Vivatsgasse 7, 53111 Bonn,
    Germany}

\email{psniady@impan.pl}

\subjclass[2020]{%
    Primary: 
        05A19; 	      %
    Secondary: 
    		05C05.     %
        }

\begin{abstract}
Stanley and F\'eray gave a formula for the irreducible character of the
symmetric group related to a \emph{multi-rectangular Young diagram}. This
formula shows that the character is a polynomial in the multi-rectangular
coordinates and gives an explicit combinatorial interpretation for its
coefficients in terms of counting certain decorated maps (i.e.,~graphs drawn on
surfaces). In the current paper we concentrate on the coefficients of the
top-degree monomials in the Stanley character polynomial, which corresponds to
counting certain decorated plane trees. We give an explicit bijection between
such trees and minimal factorizations of a cycle.
\end{abstract}

\maketitle

\footnotetext{An abridged, 12-page version of this paper was published in the
    proceedings of \emph{The 33rd International Conference on Formal Power Series
        and Algebraic Combinatorics} (FPSAC 2021) \cite{Wojtyniak2021}. At the time of
    the FPSAC submission \trokowska of the current paper used a different surname.}

\section{Introduction}

\subsection{Normalized characters and Stanley polynomials}

For a Young diagram $\lambda$ with $N=|\lambda|$ boxes and a partition
$\pi\vdash k$ we denote by
\[ \Ch_\pi(\lambda) = \begin{cases}
    \underbrace{N(N-1)\cdots(N-k+1)}_{k \text{ factors}}\  \frac{\chi^{\lambda}\left(\pi\cup
    1^{N-k}\right)}{\chi^{\lambda} \left(1^N\right)} & \text{for } k\leq N,
\\ 0 & \text{otherwise}
\end{cases}
\]
the \emph{normalized irreducible character of the symmetric group}, where
$\chi^\lambda(\rho)$ denotes the value of the usual irreducible character of
the symmetric group that corresponds to the Young diagram $\lambda$, evaluated
on any permutation with the cycle decomposition given by the partition $\rho$.
This choice of the normalization is very natural; see for example
\cite{Ivanov:1999ur,Biane:2003vd}. One of the goals of the \emph{asymptotic
    representation theory} is to understand the behavior of such normalized
characters in the scaling when the partition $\pi$ is fixed and the number of
the boxes of the Young diagram $\lambda$ tends to infinity.

For a pair of sequences of non-negative integers
$\mathbf{p}=(p_1,\dots,p_\ell)$ and $\mathbf{q}=(q_1,\dots,q_\ell)$ such that
$q_1\geq \cdots \geq q_\ell$ we consider the \emph{multi-rectangular Young
    diagram $\mathbf{p}\times \mathbf{q}$}; see~\cref{fig:multi}. Stanley
\cite{Stanley2003/04,Stanley2006} initiated investigation of the normalized
characters evaluated on such multi-rectangular Young diagrams and proved that
\begin{equation}
    \label{eq:char-multi}
    (\mathbf{p},\mathbf{q}) \mapsto \Ch_\pi\big(\mathbf{p}\times \mathbf{q}  \big) 
\end{equation}
is a polynomial (called now \emph{the Stanley character polynomial}) in the
variables $p_1,\dots,p_\ell,\allowbreak q_1,\dots,q_\ell$.

\begin{figure}
    \begin{tikzpicture}[scale=0.9]
        \draw (0,0) grid (5,2);
        \draw[ultra thick] (0,0) rectangle (5,2);
        \draw (0,2) grid (4,5); 
        \draw[ultra thick] (0,2) rectangle (4,5);
        \draw (0,5) grid (2,6); 
        \draw[ultra thick] (0,5) rectangle (2,6);
        \draw [<->,blue,very thin] (5.4,0) -- (5.4,2) node [midway,fill=white,inner sep=2pt] {\textcolor{blue}{$p_1$}};
        \draw [<->,blue,very thin] (0,0.3) -- (5,0.3) node [midway,fill=white,inner sep=2pt] {\textcolor{blue}{$q_1$}};
        \draw [<->,blue,very thin] (4.4,2) -- (4.4,5) node [midway,fill=white,inner sep=2pt] {\textcolor{blue}{$p_2$}};
        \draw [<->,blue,very thin] (0,2.3) -- (4,2.3) node [midway,fill=white,inner sep=2pt] {\textcolor{blue}{$q_2$}};
        \draw [<->,blue,very thin] (2.4,5) -- (2.4,6) node [midway,fill=white,inner sep=2pt] {\textcolor{blue}{$p_3$}};
        \draw [<->,blue,very thin] (0,5.3) -- (2,5.3) node [midway,fill=white,inner sep=2pt] {\textcolor{blue}{$q_3$}};
        
    \end{tikzpicture}
    \caption{Multi-rectangular Young diagram $\mathbf{p}\times \mathbf{q}=(2,3,1) \times (5,4,2)$.}
    \label{fig:multi}
\end{figure}
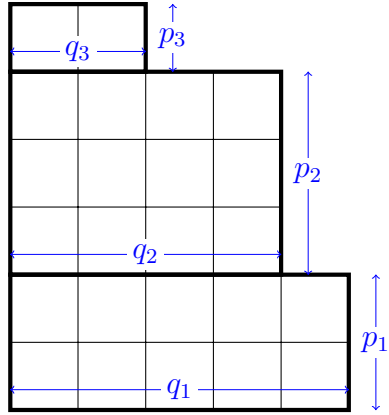

\begin{example}
    \label{example:stanley}
We will concentrate
on the special case when the partition $\pi=(k)$ consists of a single part; in
this case we use a simplified notation $\Ch_k= \Ch_k \big(\mathbf{p}\times \mathbf{q} 
\big) = \Ch_{(k)} \big(\mathbf{p}\times \mathbf{q}  \big)$.
If the number of the rectangles is equal to $\ell=2$ the first Stanley polynomials 
of this flavor are given by
\newcommand{\variables}{}  %
\begin{align*} 
\Ch_1 \variables  & = p_1q_1 + p_2q_2,\\[1ex]
\Ch_2\variables  & = -p_1^2 q_1 + p_1 q_1^2 - 2p_1 p_2 q_2 - p_2^2 q_2 + p_2 q_2^2,\\[1ex]
\Ch_3\variables  & = p_1^3 q_1 - 3 p_1^2 q_1^2 + 3 p_1^2 p_2 q_2 + p_1 q_1^3 - 3 p_1 q_1 p_2 q_2 + 3 p_1 p_2^2 q_2 \\ &   
- 3 p_1 p_2 q_2^2
  + p_2^3 q_2 - 3 p_2^2 q_2^2 + p_2 q_2^3 + p_1 q_1 + p_2 q_2
,\\[1ex] 
\Ch_4\variables  & = -p_1^4 q_1 + 6 p_1^3 q_1^2 - 4 p_1^3 p_2 q_2 - 6 p_1^2 q_1^3 + 12 p_1^2 q_1 p_2 q_2 - 6 p_1^2 p_2^2 q_2\\ 
&   +  6 p_1^2 p_2 q_2^2 + p_1 q_1^4 - 4 p_1 q_1^2 p_2 q_2 + 4 p_1 q_1 p_2^2 q_2 - 4 p_1 q_1 p_2 q_2^2 - 4 p_1 p_2^3 q_2\\ & 
 + 14 p_1 p_2^2 q_2^2 - 4 p_1 p_2 q_2^3 - p_2^4 q_2 + 6 p_2^3 q_2^2 - 6 p_2^2 q_2^3 + p_2 q_2^4 - 5 p_1^2 q_1\\
&   + 5 p_1 q_1^2 - 10 p_1 p_2 q_2 - 5 p_2^2 q_2 + 5 p_2 q_2^2.
\end{align*}
\end{example}

Stanley also gave a conjectural formula (proved later for the top-degree part
by Rattan \cite{Rattan:2008wg} and in the general case by F\'eray
\cite{Feray2010}) that gives a combinatorial interpretation to the coefficients
of the polynomial \eqref{eq:char-multi} in terms of certain \emph{maps} (i.e.,
graphs drawn on surfaces). Stanley also explained how investigation of its
coefficients may shed some light on the \emph{Kerov positivity conjecture};
see \cite{Sniady2016} for more context.

Despite recent progress in this field (for the proof of the Kerov positivity
conjecture see \cite{Feray2009,Dolega2010}) there are several other positivity
conjectures related to the normalized characters $\Ch_\pi$ that remain open
(see \cite[Conjecture 2.4]{Goulden2007} and \cite{Lassalle2008}) and suggest
the existence of some additional hidden combinatorial structures behind such
characters. We expect that such positivity problems are more amenable to
bijective methods (such as the ones from \cite{Chapuy2009}) and the current
article is the first step in this direction.

\medskip

As we already mentioned, we will concentrate on the special case when the
partition $\pi=(k)$ consists of a single part. In this case the degree of the
Stanley polynomial \eqref{eq:char-multi} turns out to be equal to $k+1$. We
will also concentrate on the combinatorial interpretation of the coefficients
of the Stanley polynomial \eqref{eq:char-multi} standing at monomials of this
maximal degree $k+1$; they turn out to be related to maps of genus zero, i.e.,
\emph{plane trees}. Nevertheless, the methods that we present in the current
paper for this special case are applicable in much bigger generality and in a
forthcoming paper we discuss the applications to maps with higher genera.

\subsection{Stanley trees}
\label{sec:trees}

\begin{figure}
    {\includegraphics[clip, trim=1.7cm 1.3cm 5.5cm 2.2cm,angle=-90, width=0.55\textwidth]{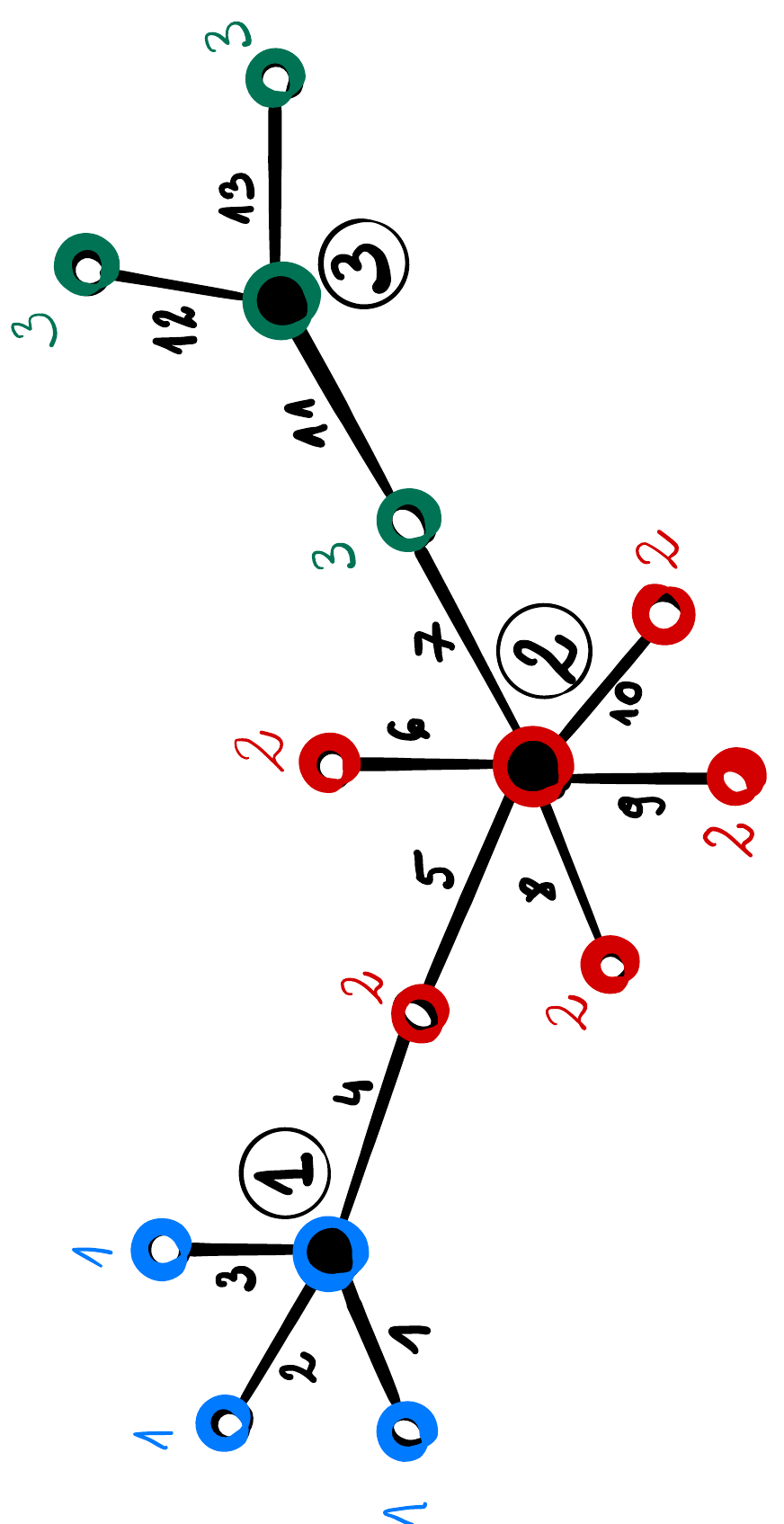}}
    
\caption{An example of a Stanley tree of type $(3,5,3)$. The circled numbers
    indicate the labels of the black vertices. The black numbers indicate the
    labels of the edges. The colors (blue for $1$, red for $2$, green for $3$)
    indicate the values of the function $f$ on white vertices.} \label{fig:plot1}
\end{figure}

Let $T$ be a \emph{bicolored} plane tree, i.e., a plane tree with each vertex
painted black or white and with edges connecting the vertices of opposite
colors. We assume that the tree has $k$ edges labeled with the numbers
$1,\dots,k$. We also assume that it has $n$ black vertices labeled with the
numbers $1,\dots,n$. The white vertices are not labeled. Being \emph{a plane
    tree} means that the set of edges surrounding any given vertex is equipped with
a cyclic order related to visiting the edges in the counterclockwise order. In
our context the structure of the plane tree can be encoded by a pair of
permutations $(\sigma_1,\sigma_2)$ with $\sigma_1,\sigma_2\in\Sym{k}$ such that
the cycles of $\sigma_1$ (respectively, the cycles of $\sigma_2$) correspond to
labels of the edges surrounding white (respectively, black) vertices. We define
the function $f$ that to each white vertex associates the maximum of the labels
of its black neighbors. We will say that $T$ is a \emph{Stanley tree of type}
    \[ (b_1,\dots,b_n):=\Big( \big| f^{-1}(1) \big|, \dots, \big|f^{-1}(n) \big| \Big); \]
in other words the \emph{type} gives the information about the number of the
white vertices for which the function $f$ takes a specified value.
\cref{fig:plot1} gives an example of a Stanley tree of type $(3,5,3)$.

Since for a tree the total number of the black and the white vertices 
is equal to the number of the edges plus $1$, it follows that 
\begin{equation}
    \label{eq:sumofb}
     b_1+\cdots+b_n + n = k+1. 
\end{equation}
Note that the definition of \emph{a Stanley tree of a given type} depends implicitly 
on the value of~$k$; in the following we will always assume that $k$ is given by~\eqref{eq:sumofb}.

By $\mathcal{T}_{b_1,\dots,b_n}$ we denote the set of Stanley trees of a specific type $(b_1,\dots,b_n)$.

\subsection{Coefficients of the top-degree $\mathbf{p}$-square-free monomials}

It turns out that in the analysis of the Stanley polynomials it is enough to restrict
attention to \emph{the $\mathbf{p}$-square-free monomials}, i.e., the monomials of the
form 
\begin{equation}
    \label{eq:monomial}
    p_1 \cdots p_n q_1^{b_1} \cdots q_n^{b_n} 
\end{equation}
with integers $b_1,\dots,b_n\geq 0$; see \cite[Section~4]{Dolega2010} and
\cite{Sniady2016} for a short overview. The following lemma is a reformulation
of a result of Rattan \cite{Rattan:2008wg}; it gives a combinatorial
interpretation to the coefficients standing at these $\mathbf{p}$-square-free
monomials that are of top-degree in the special case of the Stanley
polynomials $\Ch_k$ that correspond to a partition with a single part.

\begin{lem}
    \label{lem:coeff_stanley}
    For all integers $b_1,\dots,b_n\geq 0$ such that 
    \begin{equation}
        \label{eq:sum-of-b}
        b_1+\cdots+b_n+n=k+1
    \end{equation} 
the $\mathbf{p}$-square-free coefficient of the Stanley character polynomial is given
by
\begin{equation}
    \label{eq:stanley-feray} \left[p_1 \cdots p_n q_1^{b_1} \cdots q_n^{b_n}
\right]   \Ch_k\big(\mathbf{p}\times \mathbf{q}  \big) = 
\pm \frac{1}{(k-1)!} \left| \mathcal{T}_{b_1,\dots,b_n} \right|.
\end{equation}
\end{lem}
In order to be concise we will not discuss the sign on the right-hand side. The
above lemma is a special case of a general formula conjectured by Stanley
\cite[Conjecture~3]{Stanley2006} and proved by F\'eray \cite{Feray2010} and
therefore we refer to it as the Stanley--F\'eray character formula. This general
formula is applicable also when the assumption \eqref{eq:sum-of-b} is not
fulfilled; in this case on the right-hand side of \eqref{eq:stanley-feray} the
Stanley trees should be replaced by \emph{unicellular maps} with some additional
decorations; see \cref{sec:stanley-trees-revisited} for more details.

\medskip

There is another way of calculating the coefficient on the left-hand side
of~\eqref{eq:stanley-feray}; we shall review it in the following. The homogeneous
part of degree $k+1$ of the multivariate polynomial $\Ch_k\big(\mathbf{p}\times
\mathbf{q}  \big)$ (i.e., its homogeneous part of the top degree) is called \emph{the
    free cumulant} and is denoted by 
\mbox{$R_{k+1}\big(\mathbf{p}\times \mathbf{q} 
\big)$}. 

\begin{example}
    \label{ex:stanley2}
We continue with the notations from \cref{example:stanley}; it follows that 
\begin{align*} 
    R_2  & = p_1q_1 + p_2q_2,\\[1ex]
    R_3  & = -p_1^2 q_1 + p_1 q_1^2 - 2p_1 p_2 q_2 - p_2^2 q_2 + p_2 q_2^2,\\[1ex]
    R_4   & = p_1^3 q_1 - 3 p_1^2 q_1^2 + 3 p_1^2 p_2 q_2 + p_1 q_1^3 - 3 p_1 q_1 p_2 q_2 + 3 p_1 p_2^2 q_2 \\ &   
    - 3 p_1 p_2 q_2^2
    + p_2^3 q_2 - 3 p_2^2 q_2^2 + p_2 q_2^3 ,\\[1ex]
    R_5   & = -p_1^4 q_1 + 6 p_1^3 q_1^2 - 4 p_1^3 p_2 q_2 - 6 p_1^2 q_1^3 + 12 p_1^2 q_1 p_2 q_2 - 6 p_1^2 p_2^2 q_2\\ 
    &   +  6 p_1^2 p_2 q_2^2 + p_1 q_1^4 - 4 p_1 q_1^2 p_2 q_2 + 4 p_1 q_1 p_2^2 q_2 - 4 p_1 q_1 p_2 q_2^2 - 4 p_1 p_2^3 q_2\\ & 
    + 14 p_1 p_2^2 q_2^2 - 4 p_1 p_2 q_2^3 - p_2^4 q_2 + 6 p_2^3 q_2^2 - 6 p_2^2 q_2^3 + p_2 q_2^4.
\end{align*}
\end{example}

Free cumulants were first defined in the context of Voiculescu's free
probability theory and the random matrix theory (see \cite[Sections~1.3 and
3.4]{Dolega2010} for references); in the context of the representation theory of
the symmetric groups they were introduced in the fundamental work of Biane
\cite{Biane1998}. From the defining property of the free cumulant it follows that
\begin{equation}
    \label{eq:squarefree}
    \left[p_1 \cdots p_n q_1^{b_1} \cdots q_n^{b_n} \right]   \Ch_k\big(\mathbf{p}\times \mathbf{q}  \big) 
=
\left[p_1 \cdots p_n q_1^{b_1} \cdots q_n^{b_n} \right]   
R_{k+1}\big(\mathbf{p}\times \mathbf{q}  \big),
\end{equation}
provided that \eqref{eq:sum-of-b} holds true.

Dołęga, F\'eray and \sniady \cite[Section~3.2]{Dolega2010} introduced another
convenient family $S_2,S_3,\dots$ of functions on the set of Young diagrams
that has the property that for any \emph{strictly positive} exponents
$b_1,\dots,b_n\geq 1$ the coefficient of the corresponding
$\mathbf{p}$-square-free monomial \eqref{eq:monomial} in any finite product
\[
    S_2^{\alpha_2} S_3^{\alpha_3} \cdots = \left[ S_2\big(\mathbf{p}\times
\mathbf{q}  \big) \right]^{\alpha_2} \left[ S_3\big(\mathbf{p}\times
\mathbf{q}  \big) \right]^{\alpha_3} \cdots
\] 
(for any sequence of integers $\alpha_2,\alpha_3,\ldots\geq 0$ such that only
finitely many of its entries are non-zero) takes a particularly simple form,
cf.~\cite[Theorem~4.2]{Dolega2010}. On the other hand, the free cumulant
$R_{k+1}$ can be written as an explicit polynomial in the functions
$S_2,S_3,\dots$, cf.~\cite[Proposition~2.2]{Dolega2010}. By combining these two
results it follows that for any \emph{strictly positive} integers
$b_1,\dots,b_n\geq 1$ such that \eqref{eq:sum-of-b} holds true, the right-hand
side of \eqref{eq:squarefree} is equal to
\begin{equation}
    \label{eq:squarefree-2}
     \left[p_1 \cdots p_n q_1^{b_1} \cdots q_n^{b_n} \right]   
R_{k+1}\big(\mathbf{p}\times \mathbf{q}  \big) = (-k)^{n-1}.
\end{equation}
An astute reader may verify that this result indeed holds true for the data
from \cref{ex:stanley2}, and also that the assumption that $b_1,\dots,b_n$ are
\emph{strictly} positive cannot be weakened.

By combining Equations \eqref{eq:stanley-feray}--\eqref{eq:squarefree-2} we
obtain the following result.
\begin{corollary}
    \label{cor:number-of-stanley-trees} For any strictly positive integers
$b_1,\dots,b_n\geq 1$ such that \eqref{eq:sum-of-b} holds true, the number
of the Stanley trees of type $(b_1,\dots,b_n)$ is equal to
\begin{equation}
    \label{eq:magic-number}
     \left| \mathcal{T}_{b_1,\dots,b_n} \right| = (k-1)!\  k^{n-1}.
\end{equation} 
\end{corollary}
\begin{remark}
If some of the entries of the sequence $b_1,\dots,b_n$ are equal to
zero then the number of the Stanley trees of type $(b_1,\dots,b_n)$ takes a more
complicated form that can be extracted by an application of
\cite[Lemma~4.5]{Dolega2010}.
\end{remark}

\subsection{The main result: a bijective proof}
\label{sec:main-result-bijective-proof}

The above sketch of the proof of \cref{cor:number-of-stanley-trees} has a
disadvantage of being purely algebraic. The main result of the current paper
(see \cref{thm:thm1} below) is its new, bijective proof. We are not aware of
previous bijective proofs of this result in the literature.

\medskip

Clearly, the right-hand side of \eqref{eq:magic-number} can be interpreted as
the cardinality of some very simple combinatorial sets, such as the Cartesian
product of the set of long cycles in the symmetric group $\Sym{k}$ with the
Cartesian power $\{1,\dots,k\}^{n-1}$. We are not aware of a direct bijection
between $\mathcal{T}_{b_1,\dots,b_n}$ and some Cartesian product of this
flavor. Also, and more importantly, such Cartesian products do not generalize
well to the context of the maps with higher genus, which will be discussed in
\cref{sec:outlook}.

For these reasons,  as the first step towards the bijective proof we will look
for another natural class of combinatorial objects the cardinality of which is
given by the right-hand side of \eqref{eq:magic-number}.

\subsection{Minimal factorizations of long cycles}
\label{sec:factorizations}

We fix an integer $k\geq 1$ and denote by $\Sym{k}$ the corresponding symmetric
group. For a permutation $\pi\in \Sym{k}$ we define its \emph{norm} $\| \pi \|$
as the minimal number of factors necessary to write $\pi$ as the product of
transpositions. (The name \emph{length} is more common in the literature, but
it is confusing in the context of the phrase \emph{cycle of a given length};
see below.) We say that a permutation $\pi\in\Sym{k}$ is a \emph{cycle of
    length $\ell$} if it is of the form $\pi=(x_1,\dots,x_\ell)$. One can prove
that in this case the \emph{norm} of the cycle $\|\pi \|=\ell-1$ is equal to
its \emph{length} less~$1$.

\medskip

Let $a_1,\dots ,a_n\geq 2$ be integers. We say that a tuple
$(\sigma_1,\dots,\sigma_n)$ is a \emph{factorization of a long cycle of type
    $(a_1,\dots,a_n)$} if $\sigma_1,\dots,\sigma_n\in\Sym{k}$ are such that the
product $\sigma_1 \cdots \sigma_n$ is a cycle of length $k$ and $\sigma_i$ is a
cycle of length $a_i$ for each choice of $i\in\{1,\dots,n\}$.

By the triangle inequality, 
\[  \sum_{i=1}^n (a_i-1) =  \sum_{i=1}^n \| \sigma_i \| \geq \| \sigma_1 \cdots \sigma_n \| = k-1.\]
In the current paper we concentrate on \emph{minimal factorizations}, which
correspond to the special case when the above inequality becomes saturated and
\begin{equation}
    \label{eq:sum-of-a}
     \sum_{i=1}^{n}(a_i-1) = k-1. 
\end{equation}
By $\mathcal{C}_{a_1,\dots ,a_n}$ we denote the set of such minimal
factorizations of a long cycle of type $(a_1,\dots ,a_n)$. The name
\emph{minimal factorization} is motivated by the special case when the lengths
$a_1=\cdots=a_n=2$ are all equal to $2$ and hence $\pi=\sigma_1 \cdots 
\sigma_{n}$ is a factorization into a product of transpositions; then
\eqref{eq:sum-of-a} is satisfied if and only if $n$, the number of factors,
takes the smallest possible value.

\medskip

Biane \cite{Biane1996} extended the previous result of D\'enes \cite{Denes:1959uv}
and proved that the number of minimal factorizations $\sigma=\sigma_1\dots
\sigma_n$ of a \emph{fixed} cycle~$\sigma$ of length $k$ into a product of cycles
of lengths $a_1,\dots ,a_n \geq 2$ for which \eqref{eq:sum-of-a} is fulfilled is equal
to $k^{n-1}$. Since in the symmetric group $\Sym{k}$ there are $(k-1)!$ such
cycles $\sigma$ of length $k$, it follows that the total number of the minimal
factorizations of a long cycle of type $(a_1,\dots,a_n)$ is equal to
\begin{equation}
    \label{eq:formula-for-c}
     \left| \mathcal{C}_{a_1,\dots ,a_n} \right| = (k-1)! \ k^{n-1} 
\end{equation}
and therefore coincides with the right-hand side of \eqref{eq:magic-number}. Our
new proof of \cref{cor:number-of-stanley-trees} will be based on an explicit
bijection between the set $\mathcal{T}_{b_1,\dots ,b_n}$ of Stanley trees of some
specified type and the set $\mathcal{C}_{a_1,\dots ,a_n}$ of minimal
factorizations of a long cycle of some specified type; see \cref{thm:thm1} for
more details.

\medskip

As a side remark we note that the bijective proof of \eqref{eq:formula-for-c}
provided by Biane \cite{Biane1996} can be used to construct an explicit
bijection between $ \mathcal{C}_{a_1,\dots ,a_n}$ and the Cartesian product that we
mentioned in \cref{sec:main-result-bijective-proof}. By combining Biane's
bijection with the one provided by \cref{thm:thm1} one could get a (very
complicated) bijection between $\mathcal{T}_{b_1,\dots,b_n}$ and the Cartesian
product from \cref{sec:main-result-bijective-proof}.

\subsection{Outlook: permutations,  plane trees, maps} 
\label{sec:outlook}

For simplicity, in the current paper we consider only \emph{the first-order}
asymptotics of the character \eqref{eq:char-multi} of the symmetric group on a
cycle $\pi=(k)$, which corresponds to the coefficients of the Stanley
polynomial \eqref{eq:stanley-feray} appearing at the top-degree monomials
\eqref{eq:sum-of-b}. In the light of the aforementioned open problems that
concern the fine structure of the symmetric group characters $\Ch_k$ evaluated
on a cycle (see \cite[Conjecture~2.4]{Goulden2007} and \cite{Lassalle2008}) it
would be interesting to extend the results of the current paper to the
coefficients of \emph{general} \mbox{$\mathbf{p}$-square-free} monomials of the
Stanley character polynomial. We will keep this wider perspective in mind in
what follows.

Each of the two sets $\mathcal{T}_{b_1,\dots ,b_n}$ and
$\mathcal{C}_{a_1,\dots,a_n}$ that appear in our main bijection has an
algebraic facet and a geometric facet. In the following we will revisit the
links between these facets. These geometric facets will be essential for the
bijection that is the main result of the current paper.

\subsubsection{Stanley trees, revisited}
\label{sec:stanley-trees-revisited}

The general form of the Stanley--F\'eray character formula (see
\cite[Conjecture~3]{Stanley2006} and F\'eray \cite{Feray2010}) gives an
explicit combinatorial interpretation to the coefficient of an \emph{arbitrary}
monomial in the Stanley polynomial $ \Ch_k \big(\mathbf{p}\times \mathbf{q} 
\big) $, nevertheless it seems that in applications only
\mbox{$\mathbf{p}$-square-free} monomials are really useful; see
\cite[Section~4]{Dolega2010} and \cite{Sniady2016}. It turns out that in
general the coefficient
\begin{equation}
    \label{eq:coefficient}
      \left[p_1 \cdots p_n q_1^{b_1} \cdots q_n^{b_n} \right]   \Ch_k\big(\mathbf{p}\times \mathbf{q}  \big) 
\end{equation}
is equal (up to the $\pm$ sign) to the number of triples 
$(\sigma_1,\sigma_2,f_2)$ such that:
\begin{itemize}
    \item $\sigma_1,\sigma_2\in\Sym{k}$ are permutations with the property that their product
    \[ \sigma_1 \sigma_2 =(1,2,\dots,k) \]
    is a specific cycle of length $k$; 
    
    \item $f_2$ is a bijection between the set of cycles of the permutation
$\sigma_2$ and the set $\{1,\dots,n\}$ (we can think that $f_2$ is a \emph{labeling}
of the cycles of $\sigma_2$);
    
    \item we define the function $f_1$ on the set of cycles of the permutation $\sigma_1$ by setting
    \begin{multline*} f_1(c_1) = \max \big\{ f_2(c_2) : \text{$c_2$ is a cycle of $\sigma_2$} \\
        \text{such that the cycles $c_1$ and $c_2$ are \emph{not} disjoint} \big\} 
         \quad \text{if $c_1$ is a cycle of $\sigma_1$;}
    \end{multline*}
    we require that for each $i\in\{1,\dots,n\}$ the cardinality of its preimage
is given by the appropriate exponent of the variable $q_i$ in the monomial:
    \[ \big| f_1^{-1}(i) \big| = b_i.\]
\end{itemize}

To this algebraic object $(\sigma_1,\sigma_2,f_2)$ one can associate a geometric
counterpart, namely a \emph{bicolored map}. More specifically, it is a graph
drawn on an oriented surface with the edges labeled with the elements of the set
$\{1,\dots,k\}$. Each white vertex (respectively, each black vertex) corresponds
to some cycle of the permutation $\sigma_1$ (respectively, to some cycle of the
permutation $\sigma_2$) so that the counterclockwise cyclic order of the edges
around the vertex coincides with the cyclic order of the elements of the set
$\{1,\dots,k\}$ that are permuted by the cycle; see
\cite[Section~6.4]{Sniady:2013tb}. 

We assume that the surface on which the graph is drawn is \emph{minimal},
i.e.,~after cutting the surface along the edges, each connected component is
homeomorphic to a disc; we call such connected components \emph{faces} of the
map. The product $\sigma_1 \sigma_2=(1,2,\dots,k)$ consists of a single cycle
which geometrically means that our map has exactly one face; in other words it
is \emph{unicellular}.

The bijection $f_2$ geometrically means that the black vertices of our map are
labeled by the elements of the set $\{1,\dots,n\}$. Using such a geometric
viewpoint, $f_1$ becomes a function on the set of white vertices that to a given
white vertex associates the maximum of the labels (given by $f_2$) of its
neighboring black vertices.

\medskip

By counting the white and the black vertices it follows that the total number of
the vertices is equal to
\[  b_1+\cdots+b_n+ n. \]
A simple argument based on the Euler characteristic shows that for a
unicellular map this number of vertices is bounded from above by $k+1$ (which
is the number of the edges plus one) and the inequality becomes saturated
(i.e., the equality \eqref{eq:sum-of-b} holds true) if and only if the surface
has genus zero, i.e., it is homeomorphic to a sphere.

In the following we consider the case when \eqref{eq:sum-of-b} indeed holds
true. It is conceptually simpler to consider such a map drawn on the sphere as
drawn on the plane; being unicellular corresponds to the map being a tree. It
follows that in this case the geometric object associated above to the triple
$(\sigma_1,\sigma_2,f_2)$ that contributes to the coefficient
\eqref{eq:coefficient} coincides with the Stanley tree of type
$(b_1,\dots,b_n)$.

\medskip

The above discussion motivates the notion of the Stanley trees and shows which
more general geometric objects should be investigated in order to study more
refined asymptotics of the characters of the symmetric groups.

\subsubsection{Minimal factorizations}
\label{sec:geometry-minimal}

The geometric object that can be associated to a minimal factorization
$\sigma_1,\dots,\sigma_n\in\Sym{k}$ of a long cycle of type $(a_1,\dots,a_n)$
is a graph with $k$ white vertices (labeled with the elements of the set
$\{1,\dots,k\}$) and $n$ black vertices (labeled with the elements of the set
$\{1,\dots,n\}$). We connect the black vertex $i$ with the white vertices
$\sigma_{i,1},\dots,\sigma_{i,a_i}$ that correspond to the elements of the
cycle $\sigma_i=(\sigma_{i,1},\dots,\sigma_{i,a_i})$. This graph is clearly
connected, it has $k+n$ vertices and it has $a_1+\cdots+a_n$ edges; from the
minimality assumption \eqref{eq:sum-of-a} it follows that the graph is, in
fact, a tree. We may encode the cycles $\sigma_{1},\dots,\sigma_n$ by drawing
the tree on the plane in such a way that the counterclockwise order of the
white vertices surrounding a given black vertex $i$ corresponds to the cycle
$\sigma_i$; see \cref{fig:minimal-factorization} for an example. On the other
hand, we have a freedom of choosing the cyclic order of the edges around the
white vertices. In this way a minimal factorization of a long cycle can be
encoded (in a non-unique way) by a plane tree with labeled white vertices and
labeled black vertices. Later on we will remove this ambiguity by choosing the
cyclic order around the white vertices in some canonical way.

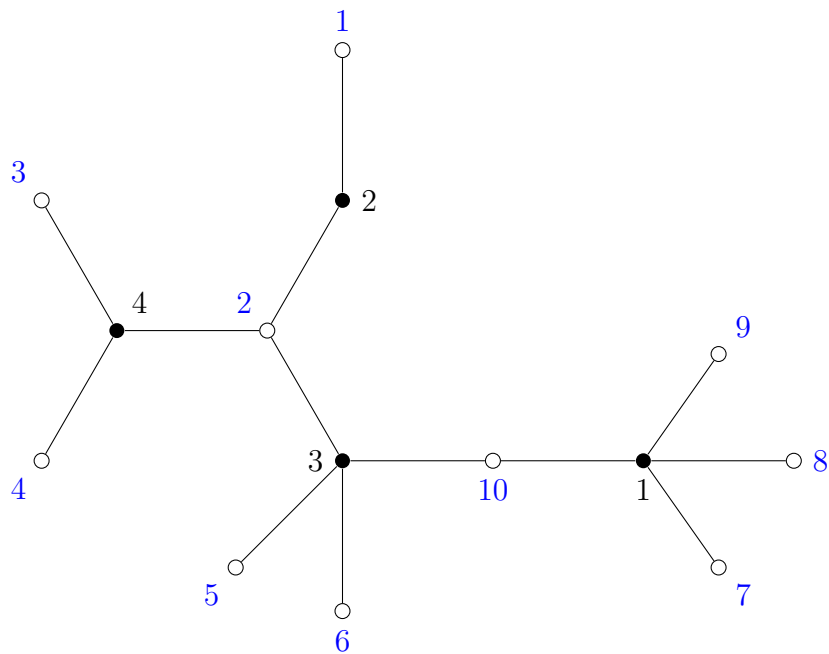
\begin{figure}[t]
\begin{tikzpicture}[scale=2]
    
    \node[circle, draw=black, inner sep=2pt,label={[label distance=0ex]120:$\textcolor{blue}{2}$}] (white2) at (0, 0) {};

    \node[circle, fill=black, inner sep=2pt,label={[label distance=0ex]60:$4$}] (black4) at ++(-1, 0) {};
    \draw (white2) -- (black4);
    
    \node[circle, draw=black, inner sep=2pt,label={[label distance=0ex]120:$\textcolor{blue}{3}$}] (white3) at  (-1.5,0.866) {}; 
    \node[circle, draw=black, inner sep=2pt,label={[label distance=0ex]-120:$\textcolor{blue}{4}$}] (white4) at  (-1.5,-0.866) {}; 
    
    \draw (white3) -- (black4);
    \draw (white4) -- (black4);
    
    \node[circle, fill=black, inner sep=2pt,label={[label distance=0ex]0:$2$}] (black2) at ++(0.5, 0.866) {};
    \node[circle, draw=black, inner sep=2pt,label={[label distance=0ex]90:$\textcolor{blue}{1}$}] (white1) at  (0.5,1.866) {}; 
    
    \draw (white1) -- (black2) -- (white2);
    
    \node[circle, fill=black, inner sep=2pt,label={[label distance=0ex]180:$3$}] (black3) at ++(0.5, -0.866) {};
    \node[circle, draw=black, inner sep=2pt,label={[label distance=0ex]-90:$\textcolor{blue}{6}$}] (white6) at  (0.5,-1.866) {}; 
    \node[circle, draw=black, inner sep=2pt,label={[label distance=0ex]-130:$\textcolor{blue}{5}$}] (white5) at  (0.5-0.71,-0.866-0.71) {}; 
    
    \draw (white6) -- (black3) -- (white5);
    \draw (white2) -- (black3);
    
    \node[circle, draw=black, inner sep=2pt,label={[label distance=0ex]-90:$\textcolor{blue}{10}$}] (white10) at  (0.5+1, -0.866) {}; 
    \node[circle, fill=black, inner sep=2pt,label={[label distance=0ex]-90:$1$}] (black1) at (0.5+2, -0.866) {};
    
    \node[circle, draw=black, inner sep=2pt,label={[label distance=0ex]-45:$\textcolor{blue}{7}$}] (white7) at  (0.5+2+0.5, -0.866-0.71) {}; 
    \node[circle, draw=black, inner sep=2pt,label={[label distance=0ex]0:$\textcolor{blue}{8}$}] (white8) at  (0.5+2+1, -0.866) {}; 
    \node[circle, draw=black, inner sep=2pt,label={[label distance=0ex]45:$\textcolor{blue}{9}$}] (white9) at  (0.5+2+0.5, -0.866+0.71) {}; 
    
    \draw (black3) -- (white10) -- (black1) -- (white7);
    \draw (white8) -- (black1) -- (white9);
    
\end{tikzpicture}
    
    \caption{One of the plane trees that correspond to the minimal
    factorization $\sigma_1 \cdots \sigma_n= \sigma_1 \sigma_2 \sigma_3
    \sigma_4 = (1,2,\dots,10)\in\Sym{k}$ of a long cycle with
    $\sigma_1=(7,8,9,10)$, $\sigma_2=(1,2)$, $\sigma_3=(2,5,6,10)$,
    $\sigma_4=(2,3,4)$, $k=10$ and $n=4$. The black vertices correspond to the
    cycles $\sigma_1,\dots,\sigma_4$. The white vertices correspond to the
    elements of the set $\{1,\dots,10\}$ on which acts the symmetric group
    $\Sym{10}$. The cyclic order of the edges around the black vertices is
    determined by the cycles $\sigma_1,\dots,\sigma_4$. The cyclic order of the
    edges around the white vertices is arbitrary.}
\label{fig:minimal-factorization}
\end{figure}

The original permutation $\sigma_i$ can be recovered from the tree by reading
the counterclockwise cyclic order of the labels of the neighbors surrounding
the black vertex $i$; see \cref{fig:minimal-factorization}.

\subsection{Overview of the paper} 

In \cref{sec:algorithm} we state our main result (\cref{thm:thm1}) about
existence of a bijection between certain sets of minimal factorizations and
Stanley trees; the bijection itself is constructed in
\cref{sec:first,sec:second,sec:spine,sec:rib}. It is quite surprising that a
bare-boned description of the bijection (without the proof of its correctness)
is quite short, nevertheless this algorithm creates a quite complex dynamics,
as can be seen by the length of the description of the inverse map.
Additionally, in \cref{sec:correctness-spine,sec:correctness-rib} we prove that
this algorithm is well defined.

As the first step towards the proof of \cref{thm:thm1},
in \cref{sec:proof} we show \cref{lm:lem1}, which states 
that the output of our algorithm indeed is a Stanley tree of a specific type.

\cref{sec:alternative} contains an alternative description of the bijection.

In \cref{sec:inverse} we construct the inverse map.

Finally, in \cref{sec:conclusion} we complete the proof of \cref{thm:thm1}.

\pagebreak

\section{The main result: bijection between Stanley trees \\ and minimal factorizations of long cycles}
\label{sec:algorithm}

The following is the main result of the current paper.

\begin{thm}
    \label{thm:thm1}
Let $n\geq 2$ and $b_1,\dots,b_n\geq 1$ be integers.
We define the integers $a_1,\dots,a_n$ by
\begin{equation}
    \label{eq:a-and-b}
     a_i = \begin{cases} 
              b_i+1 & \text{if } i \in \{1,n\}, \\
              b_i+2 & \text{otherwise}.
           \end{cases} 
\end{equation} 
Then the algorithm $\CT$ presented below gives a bijection between the set 
 $\mathcal{C}_{a_1,\dots ,a_n}$ of minimal factorizations (see \cref{sec:factorizations})
and the set $\mathcal{T}_{b_1,\dots ,b_n}$ of Stanley trees
(see \cref{sec:trees}).
\end{thm}

Note that both the notion of a Stanley tree as well as the notion of the minimal
factorization implicitly depend on the value of $k$ given, respectively, by
\eqref{eq:sumofb} and \eqref{eq:sum-of-a}. In our context, when
\eqref{eq:a-and-b} holds true, these two values of $k$ coincide.

\subsection{The first step of the algorithm $\CT$: from a factorization to a tree with repeated edge labels}
\label{sec:first}

In the first step of our algorithm $\CT$ to a given minimal
factorization $(\sigma_1,\dots,\sigma_n)\in \mathcal{C}_{a_1,\dots ,a_n}$ 
we will associate a bicolored plane tree $T_1$ with
labeled black vertices and labeled edges. The remaining part of the current
section is devoted to the details of this construction.

\subsubsection{The tree $T_0$}
\label{sec:tt0}

Just like in \cref{sec:geometry-minimal}
we start by creating a graph $T_0$ with $n$ black vertices labeled $1,\dots,n$ and
with $k$ white vertices labeled $1,\dots,k$, where $k$ is given by~\eqref{eq:sum-of-a}.
Each black vertex $i$
corresponds to the cycle $\sigma_i=(\sigma_{i,1}, \dots, \sigma_{i,a_i})$ and
so we connect the black vertex $i$ with the white vertices $\sigma_{i,1},
\dots, \sigma_{i,a_i}$. 
By the same argument as in \cref{sec:geometry-minimal} this graph is, in fact, a tree.

In order to give this tree  the structure of a \emph{plane tree} we need to
specify the cyclic order of the edges around each vertex. Just like in
\cref{sec:geometry-minimal} we declare that going counterclockwise around the
black vertex~$i$ the cyclic order of the labels of the white neighbors should
correspond to the cyclic order $\sigma_{i,1}, \dots, \sigma_{i,a_i}$. The
cyclic order around the white vertices is more involved and we present it in
the following.

\medskip

The path between the two black vertices with the labels $1$ and $n$ will be
called \emph{the spine}; on \cref{subfig:ex1} it is drawn as the horizontal red path. 
There will be
two separate rules that determine the cyclic order of the edges around a given
white vertex, depending whether the vertex belongs to the spine or not.

For each white vertex that is \emph{not} on the spine we declare that going
counterclockwise around it, the labels of its black neighbors should be arranged in
the increasing way (for example, the neighbors of the white vertex $6$ on
\cref{subfig:ex1} listed in the counterclockwise order are $3,4,6$). 

\begin{figure}
    {\includegraphics[clip, trim=0cm 9cm 0cm 0cm, width=0.5\textwidth]{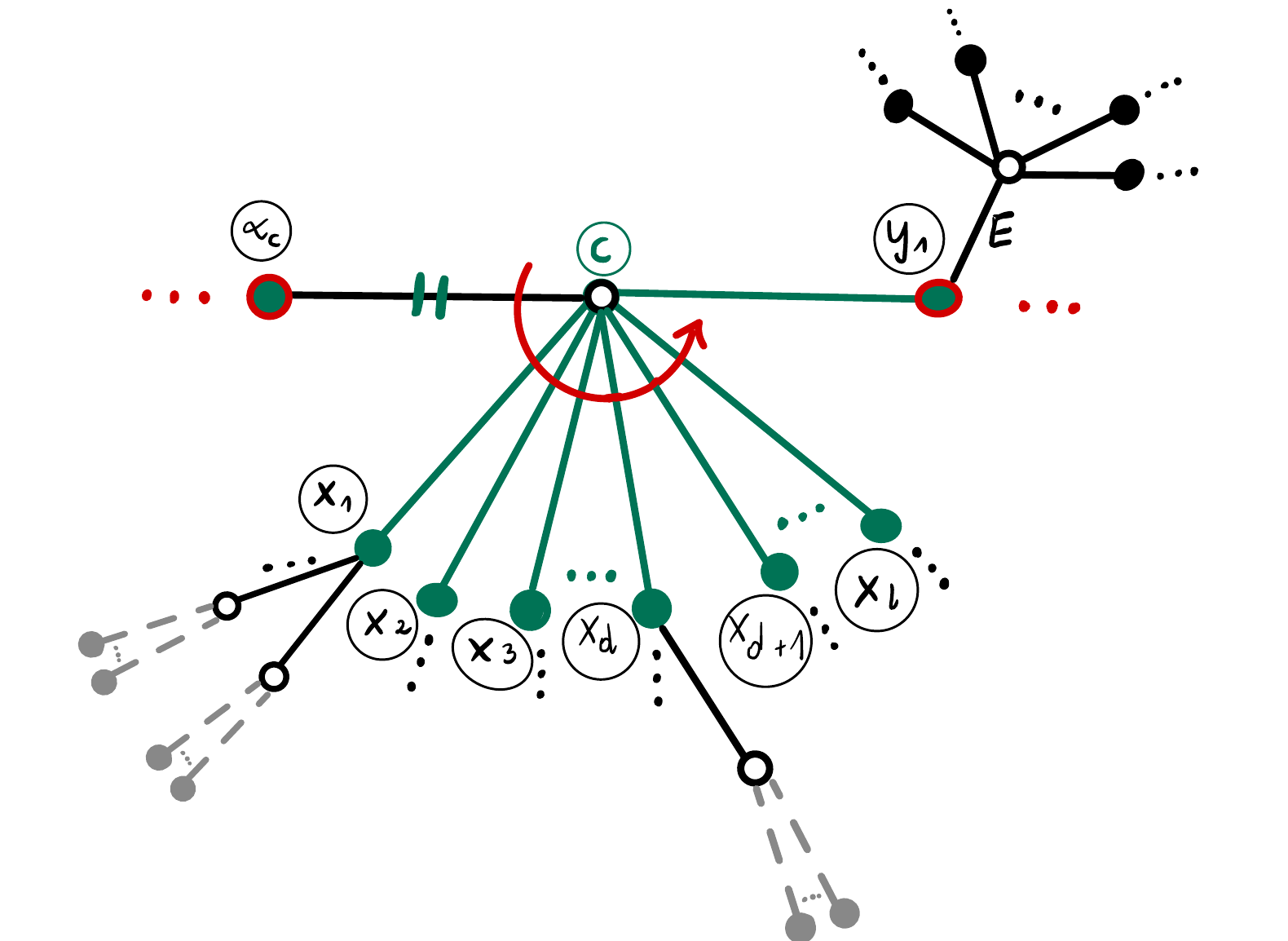}}
    
\caption{The structure of a white spine vertex $c$ in the plane tree $T_0$. The
    labels of the black spine vertices $\alpha_c,y_1\in\{1,\dots,n\}$ fulfill
    $\alpha_c< y_1$. There are $l\geq 0$ non-spine neighbors of the vertex $c$,
    denoted $x_1,\dots,x_l$. They are all placed (in the counterclockwise cyclic
    order) after $\alpha_c$ and before $y_1$. It~may happen that $l=0$ and there
    are no non-spine neighbors of $c$. For the details see \cref{sec:tt0}.}

\label{fig:untouched0}
\end{figure}

For each white vertex $c$ that belongs to the spine there are exactly two black
neighbors that belong to the spine; we denote their labels by $\Root_c$ and
$y_1$ with $\Root_c<y_1$; see \cref{fig:untouched0}. Going counterclockwise
around $c$, all non-spine edges should be inserted after $\Root_c$ and
before~$y_1$. Their order is determined by the requirement that---after
neglecting the vertex~$y_1$---the cyclic counterclockwise order of the
remaining vertices should be increasing. For example, for the white vertex $1$
on \cref{subfig:ex1} we have $\Root_1=7$, $y_1=11$ and the counterclockwise
cyclic order of the non-$y_1$ black neighbors is $4,7,13$.

\medskip

\begin{figure}
    \subfloat[]{\label{subfig:ex1}
    \includegraphics[width=0.45\textwidth]{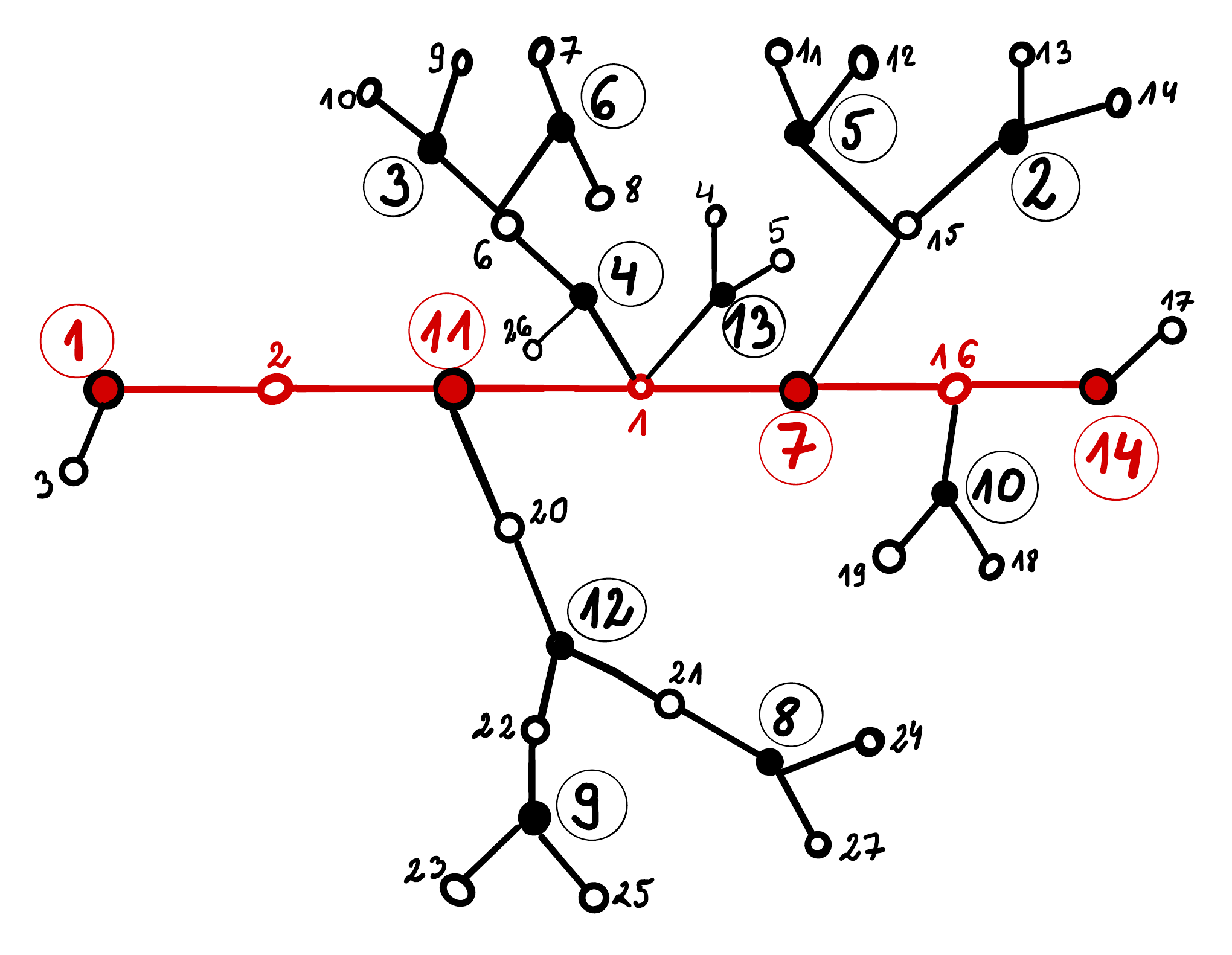}} \quad
\subfloat[]{\label{subfig:ex2}
    \includegraphics[width=0.45\textwidth]{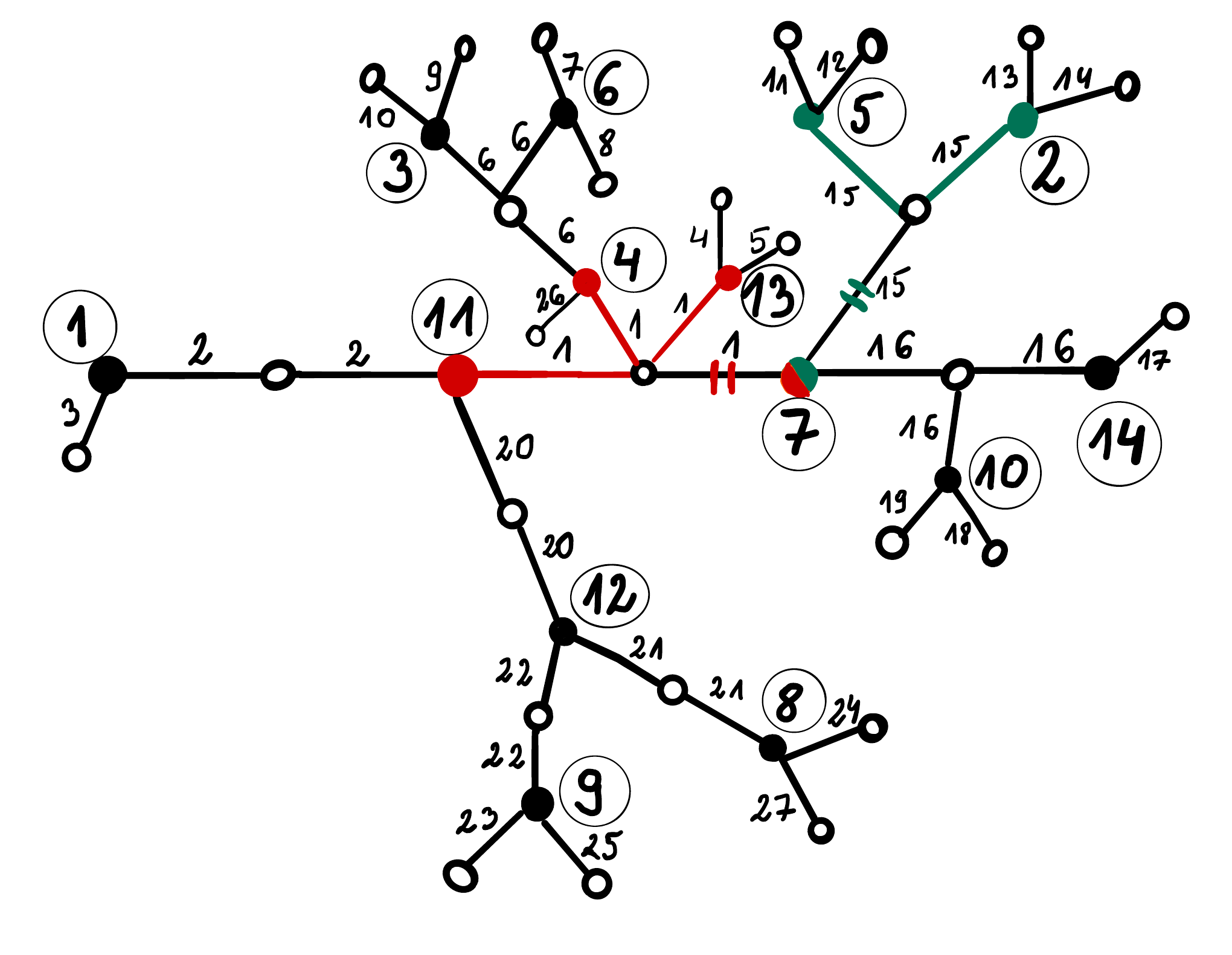}}
\caption{\protect\subref{subfig:ex1} The output $T_0$ of the first step of
    the algorithm $\CT$ applied to the minimal factorization
    \eqref{eq:example}. The spine is the horizontal red path between the
    vertices $1$ and $14$. 
    \protect\subref{subfig:ex2} The tree $T_1$ with some
    sample clusters highlighted. The red color indicates the cluster $1$, while
    the green indicates the cluster  $15$. The double transverse lines identify
    the roots of the respective clusters.} 
\label{fig:ex12}
\end{figure}

For example, \cref{subfig:ex1} gives the tree $T_0$ that corresponds to 
\[ n=14,\quad 
k=27,\qquad a_1=a_{14}=2,\quad a_2=\cdots=a_{13}=3\] 
and the minimal factorization $(\sigma_1,\dots,\sigma_{14})\in
\mathcal{C}_{2,3^{12},2}$ 
\begin{multline}
    \label{eq:example}
    \sigma_1=(2,3), \quad 
    \sigma_2=(13,15,14), \quad
    \sigma_3=(6,9,10),\quad
    \sigma_4=(1,6,26), \\
    \sigma_5=(11,15,12), \quad
    \sigma_6=(6,8,7), \quad
    \sigma_7=(1,16,15) \quad
    \sigma_8=(21,27,24), \\
    \sigma_9=(22,23,25), \quad
    \sigma_{10}=(16,19,18), \quad
    \sigma_{11}=(2,20,1),\\
    \sigma_{12}=(20,22,21), \quad
    \sigma_{13}=(1,5,4), \quad
    \sigma_{14}=(16,17).
\end{multline}

We denote by $T_1$ the tree $T_0$ in which each edge is labeled by its white
endpoint and then all labels of the white vertices are removed; see
\cref{subfig:ex2,fig:ex2B}. The tree $T_1$ is the output of the first step of the
algorithm $\CT$.

\subsubsection{Information about the initial tree $T_1$}

In  the current section we will define certain sets and functions that describe
the shape of the initial tree~$T_1$. In the language of programmers: we will
create variables $B_c$, $\Root_c$, $\mathcal{B}$, and $\mathfrak{C}$ that will
not change their values during the execution of the algorithm~$\CT$.

\medskip

For $c \in\{1,\dots,k\}$  by \emph{the cluster $c$} we mean the set of edges
that carry the label $c$, together with their black endpoints. We will also
say that $c$ is the label of this cluster. All edges in a given cluster have
the same white endpoint; this property will be preserved by the action of our
algorithm. This common white vertex will be called \emph{the center} of the
cluster. For example, in \cref{subfig:ex2} the cluster $1$ is drawn in red,
while the cluster $15$ is drawn in green. We denote by
$B_c\subseteq\{1,\dots,n\}$ the set of labels of the black vertices in the
cluster $c$ in the tree~$T_1$. 

Each cluster may contain either zero, one, or two black spine vertices.
A cluster is called a \emph{spine cluster} if it contains exactly two black spine
vertices. The set of labels of such spine clusters will be denoted by 
$\mathfrak{C}\subseteq\{1,\dots,k\}$.

\begin{figure}
    \centering \includegraphics[width=0.8\textwidth]{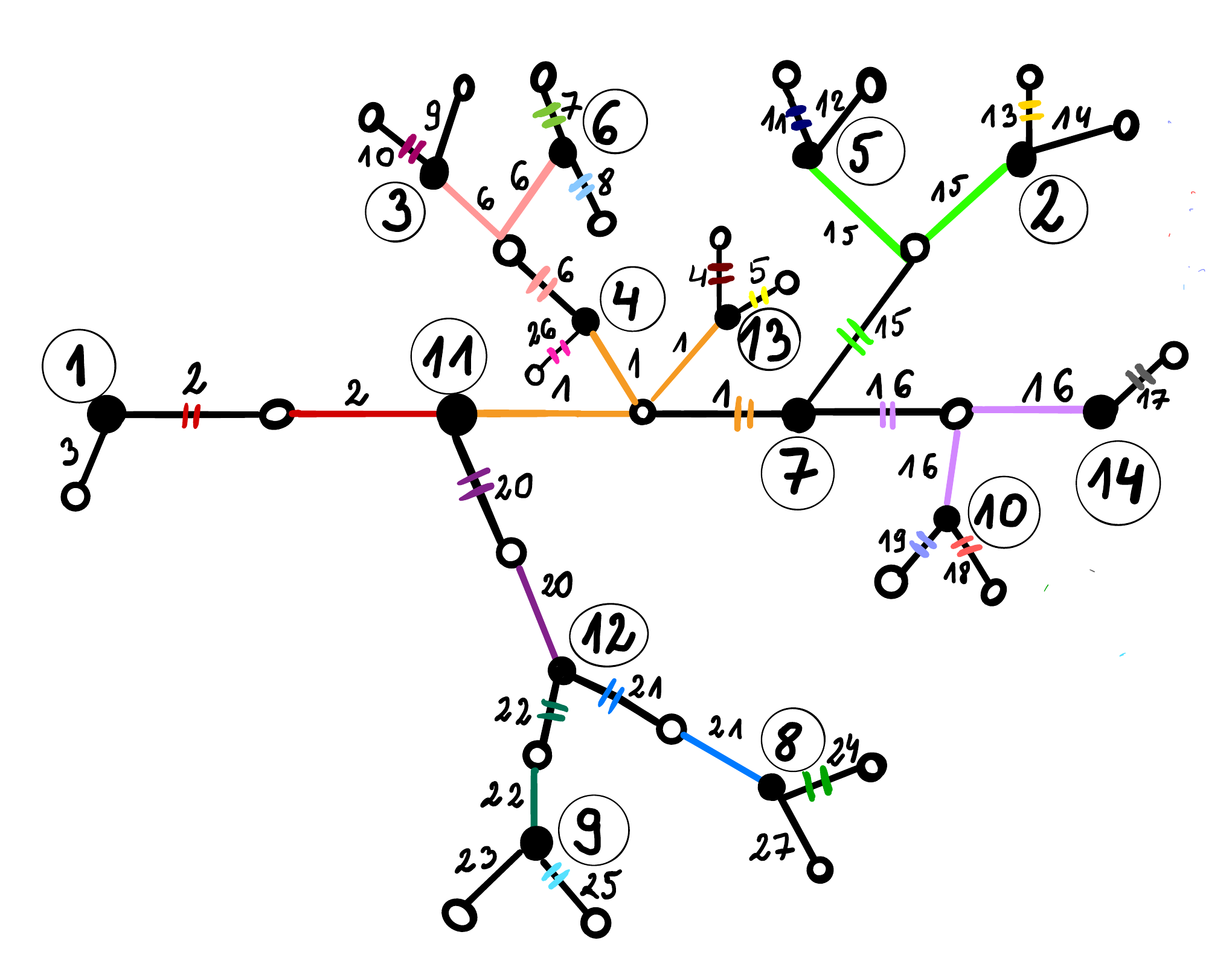} 
    
    \caption{The tree $T_1$ that is the starting point of the second step of
    the algorithm $\CT$ for the minimal factorization \eqref{eq:example}.
    Essentially this is an enlarged version of \cref{subfig:ex2} with some
    additional highlights. For the sake of clarity we paint the non-root edges
    of each cluster in one color, while the root edge is marked with two
    transverse lines of the same color.}
    
    \label{fig:ex2B}
\end{figure}

We orient the non-spine edges of the tree so that the arrows point towards the
spine. The \emph{root of a non-spine cluster} is defined as the unique edge
outgoing from the center. The \emph{root of a spine cluster} is defined as the
edge with the smallest label of the black endpoint among the two spine edges in
the cluster; with the notations of \cref{fig:untouched0} it is the edge between
the vertices $c$ and $\alpha_c$. For example, in \cref{fig:ex2B} the root of a
cluster is marked with two transverse lines. Heuristically, our strategy will
be to keep removing the edges from the cluster; the root is the unique cluster
edge that will remain at the end. The notion of the root will not be used in the
description of the algorithm $\CT$, nevertheless it will be a convenient tool
for proving later its correctness.

The black end of the root of a cluster $c \in\{1,\dots,k\}$ in the tree $T_1$
will be called \emph{the anchor} of the cluster $c$. Its label will be denoted
by $\Root_c\in\{1,\dots,n\}$. By definition, the label of the anchor will not
change during the execution of the algorithm. However, it may happen at later
steps that the anchor (i.e.,~the black vertex that carries the label
$\alpha_c$) is no longer one of the endpoints of the root. When it does not
lead to confusions we will identify the anchor (understood as a black vertex)
with its label $\alpha_c$.

By $\mathcal{B}\subseteq\{1,\dots,n\}$ we denote the set of the
labels of black spine vertices. 

A cluster is called a \emph{leaf} if it contains exactly one edge. 
For example, in \cref{subfig:ex2,fig:ex2B} the cluster $10$ is a leaf.

\medskip

For the example from \cref{subfig:ex2,fig:ex2B} we have:
\begin{multline}
    \label{eq:step}
    \mathcal{B}=\{1,7,11,14\}, \quad 
    \mathfrak{C}=\{1,2,16\}, \\
    B_1=\{4,7,11,13\}, \quad
    B_2=\{1,11\}, \quad 
    B_6=\{3,4,6\}, \quad 
    B_{15}=\{2,5,7\}, \\ 
\shoveright{ 	B_{16}=\{7,10,14\}, \quad
	B_{20}=\{11,12\}, \quad
    B_{21}=\{8,12\}, \quad
    B_{22}=\{9,12\},}\\
\shoveleft{    \Root_1=7, \quad
    \Root_2=1, \quad  
    \Root_6=4, \quad
    \Root_{15}=7,}  \\
    \Root_{16}=7, \quad
    \Root_{20}=11, \quad
    \Root_{21}=12, \quad
    \Root_{22}=12. 
\end{multline}
We did not list the values of $B_i$ and $\Root_i$ for white vertices $i$ that
are non-spine leaves because these values will not be used by our algorithm.

\medskip

Recall that we orient the non-spine edges of the tree so that the arrows point
towards the spine. The set of non-spine clusters is partially ordered by the
orientations of the edges as follows: a~cluster $c_1$ is a predecessor of a
cluster $c_2$ if the path from the spine to the center of the cluster $c_2$
passes through the center of the cluster~$c_1$. This partial order can be
extended to a linear order (not necessarily in a unique way). Let $\Sigma$ be
the sequence of the clusters that are non-spine and non-leaf, arranged
according to this linear order. For example, for the tree $T_1$ shown on
\cref{subfig:ex1} we can choose $\Sigma=(6,15,20,21,22)$.
\label{text:definiton-of-Sigma}

\begin{figure}
    \centering {\includegraphics[clip, trim=3cm 0cm 0cm 5cm,angle=-90,width=0.4\textwidth]{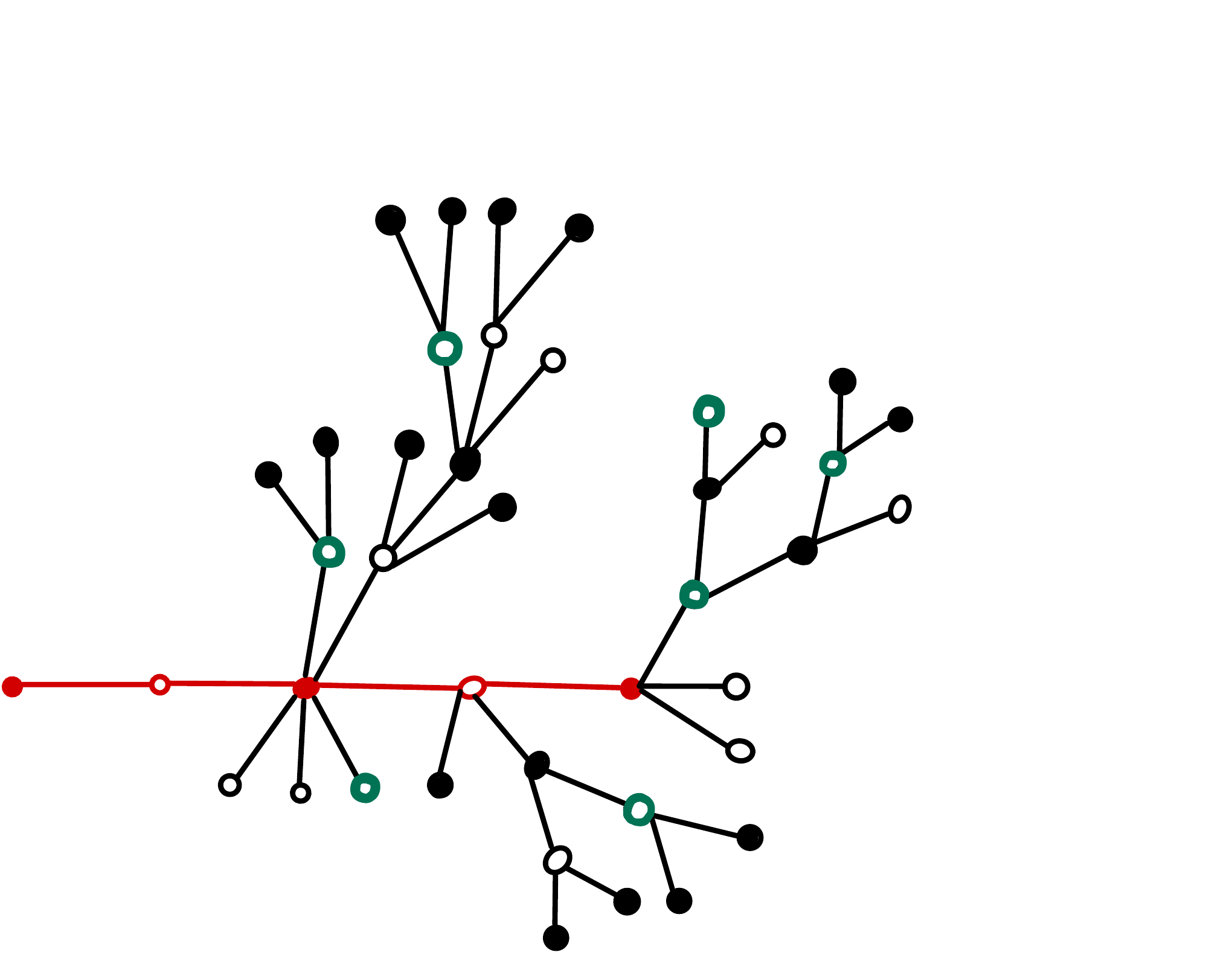}}   
    \caption{The tree $T_1$ shown without the labels. The red edges indicate the spine. 
        The thick green empty circles indicate white leftist vertices.}
    \label{fig:left}
\end{figure}

The aforementioned orientations of the edges allow us also to use some
vocabulary taken from the theory of \emph{rooted trees}, at least for the
vertices which are away from the spine. For example, the \emph{parent} of a
non-spine vertex $v$ is the unique vertex $w$ such that there is an oriented
edge from $v$ to $w$. The terms such as \emph{children} and \emph{descendants}
of a non-spine vertex are defined accordingly. 

Let $c_1,c_2$ be clusters in the tree $T_1$ and let $v_1,v_2$ be their centers.
We say that \emph{the cluster $c_2$ is a child of the cluster $c_1$} or, in
other words, \emph{the cluster $c_1$ is the parent of the cluster $c_2$} if
there exists a black vertex $w$ that is a child of the vertex $v_1$ and such
that $v_2$ is a child of $w$, i.e.,~$v_2$ is a grandchild of $v_1$.

In the plane tree $T_1$ we can split the set of white non-spine vertices into
the set of \emph{leftist} vertices and the set of \emph{non-leftist} vertices.
A white non-spine vertex is called \emph{leftist} in the tree $T_1$ if it is a
leftmost child of its parent when viewed from the point of view of the parent.
Note that a black spine vertex may have \emph{two} leftmost children located on
either side of the spine. For example, for the tree $T_1$ shown on
\cref{fig:left}, the white leftist vertices are drawn as thick green empty
circles. We say that a cluster in $T_1$ is \emph{leftist} if its center is a
leftist vertex.

\subsection{The second step of the algorithm $\CT$: 
    from a tree with repeated edge labels to a tree with unique edge labels}
\label{sec:second}

The starting point of the second step of our algorithm $\CT$ is the
bicolored plane tree~$T_1$ with black vertices labeled $1,\dots,n$ and with edges
labeled with the numbers $1,\dots,k$; note that the edge labels are repeated.
Our goal in this second step of the algorithm is to transform the tree so that
the edge labels are not repeated.

We declare that at the beginning all clusters of the tree $T_1$ are
\emph{untouched}. Also, each non-spine black vertex is declared to be
\emph{untouched}. On the other hand, each black spine vertex is declared to be
\emph{touched}.

The second step of the algorithm will consist of two parts: firstly we apply
\emph{the spine treatment} (\cref{sec:spine}), then we apply \emph{the rib
    treatment} (\cref{sec:rib}). In fact, these two parts are very similar: each of
them consists of an external loop that has a nested internal loop; one could
merge these two parts and regard them as an instance of a single external loop
that treats the spine vertices and the non-spine vertices in a slightly
different way.

\medskip

During the action of the forthcoming algorithm some edges will be removed from
each cluster. The root of the cluster may change during the execution of the
algorithm. Also, as we already mentioned, the anchor of a cluster may no longer
be the black endpoint of the root. On the positive side, the following
invariant guarantees that some properties of a cluster will persist (the proof
is postponed to \cref{lem:invariants-ok} and \cref{prop:invariants2-ok}).

\begin{invariant}
   \label{invariant}
    At each step of the algorithm and for each cluster $c$ the following properties hold true:
    \begin{enumerate}[label=(I\arabic*)]
 \item the edges of the cluster $c$ have a common white endpoint (called, as before, 
 		\emph{the center} of the cluster),       
 \item the anchor of $c$ is a black vertex that is connected by
    an edge with the center of $c$,
   \item the root of $c$ is one of the edges that form the cluster $c$,
   \item\label{invariant:4} if the black endpoint of the root of $c$ 
   is not equal to the anchor $\Root_c$  
   then this black endpoint is touched. 
\end{enumerate}
\end{invariant}

Be advised that distinct clusters may at later stages of the algorithm share
the same center. 

The algorithm will be described in terms of two operations, called \emph{bend}
and \emph{jump}; we present them in the following. Each of them decreases the
number of the white vertices by~$1$, as well as decreases the number of edges
by~$1$. The edge that disappears has a repeated label, in this way the set of
the edge labels remains unchanged.

\subsection{The building blocks of the second step: bend $\Bend_{x,y}$}

\begin{figure}
    \centering \subfloat[]{\label{subfig:exA} {\includegraphics[clip,trim=0cm
        1.9cm 3.5cm 1.5cm,width=0.45\textwidth]{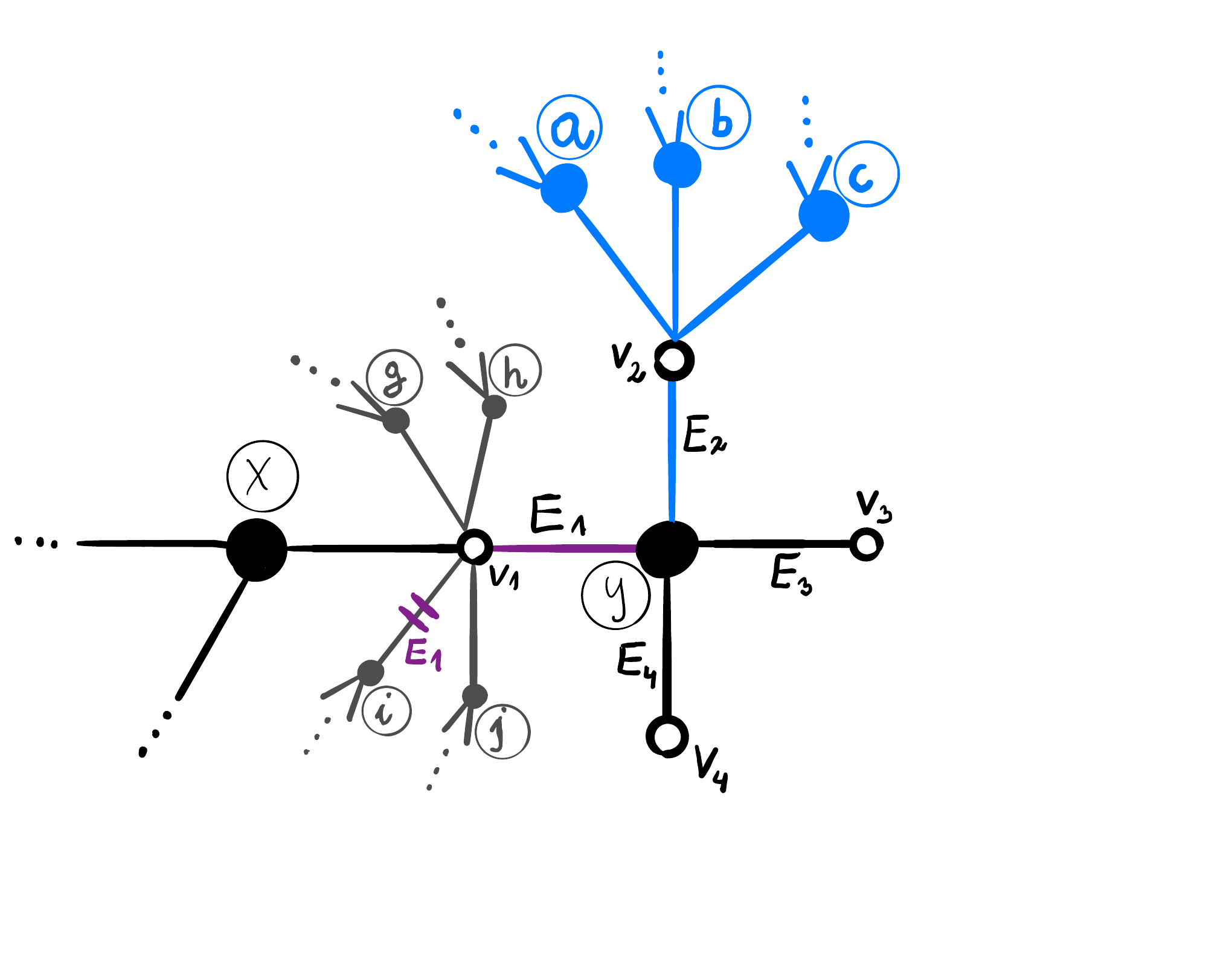}} } \quad
\subfloat[]{\label{subfig:exB} {\includegraphics[clip,trim=0cm 2.5cm 3.5cm
        0.9cm, width=0.45\textwidth]{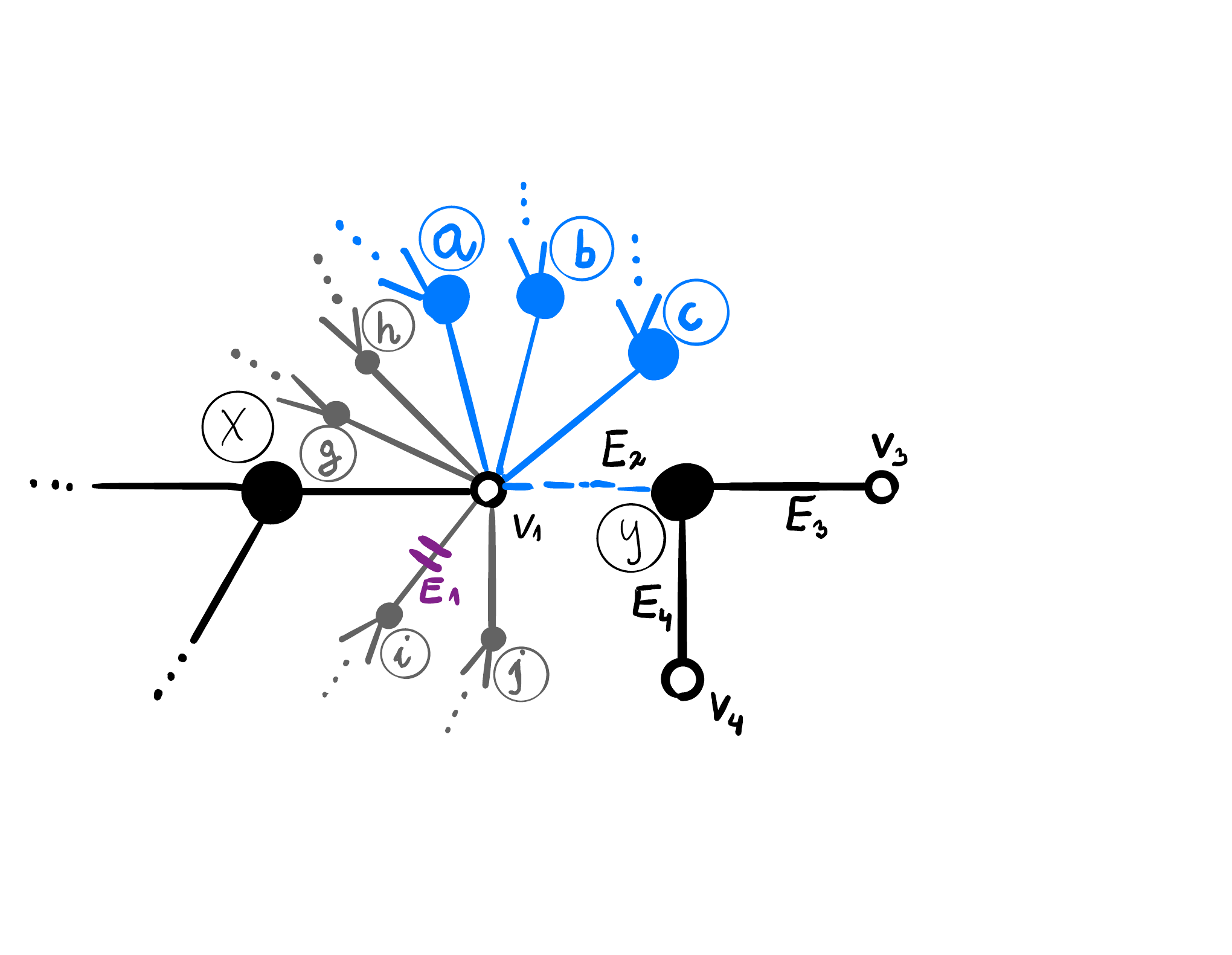}} }
\caption{\protect\subref{subfig:exA} The initial configuration of the tree
    before the bend operation $\Bend_{x,y}$ was applied. The vertex $y$ belongs to
    the cluster $E_1$.  The center of $E_1$ is equal to $v_1$. The anchor of $E_1$
    is equal to $x$. The black endpoint of the root of $E_1$ is equal to $i$. We
    assume that $y\notin\{x,i\}$. The neighbors of~$y$, listed in the clockwise
    order, are: $v_1,v_2,\dots,v_d$, and the corresponding edges connecting $y$
    with the neighbors are: $E_1,\dots,E_d$. For a complete list of notations and
    assumptions see \cref{sec:assumptions-e}.
    \newline
    \protect\subref{subfig:exB} The output of $\Bend_{x,y}$. The edge
    connecting $v_2$ with $y$ was rotated counterclockwise so that the vertex
    $v_2$ was merged with $v_1$. For a detailed description see
    \cref{sec:output-of-bend}.}

\label{fig:ex}
\end{figure}

\subsubsection{Assumptions about the input of\/ $\Bend_{x,y}$.}
\label{sec:assumptions-e}

We list below the assumptions about the input of the operation \emph{bend}.
We also introduce some notations.

\begin{enumerate}[label=(B\arabic*)]
    \item \label{AE1} 
The operation $\Bend_{x,y}$ takes as an input a bicolored tree $T$ that is
assumed to be as in \cref{invariant}, together with a choice of two distinct
black vertices $x$, $y$. We assume that there is a cluster $E_1$ such that $x$
is the anchor of $E_1$ and $y$ belongs to $E_1$; see \cref{subfig:exA}. We
denote the center of $E_1$ by~$v_1$.

\item
\label{AE2}
We denote by $i$ the black endpoint of the root of the cluster $E_1$.
We assume that $i\neq y$.

\item \label{AE3} 
We also assume that the black vertex $y$ has degree $d\geq 2$; we denote the
edges around the vertex $y$ by $E_1,\dots, E_d$ (going clockwise, starting from
the edge $E_1$ between $v_1$ and $y$). We denote the white endpoint of the edge
$E_i$ by $v_i$.
\end{enumerate}

\subsubsection{The output of $\Bend_{x,y}$} 
\label{sec:output-of-bend}

The bend operation can be thought of as a  counterclockwise rotation of the
edge~$E_2$ around the vertex $y$ so that it is merged with the edge $E_1$; a
more formal description of the output of $\Bend_{x,y}$ is given as follows.

We remove the edge between the vertices $y$ and $v_2$. The label of the edge
between $v_1$ and $y$ is changed to~$E_2$. If the removed edge was the root of
the cluster~$E_2$, the aforementioned edge between $v_1$ and $y$ becomes the
new root of the cluster $E_2$.

Then we merge the vertex $v_2$ with the vertex $v_1$ in such
a way that going clockwise around $v_1$ the newly attached edges are
immediately before the edge $E_2$ (these edges are marked blue on
\cref{fig:ex}).

From the following on we declare that the cluster $E_1$ is \emph{touched}
and also the black vertex $y$ is \emph{touched}.
It is easy to check that the output tree still fulfills the properties from \cref{invariant}.

\medskip

We will say that some operation on the tree \emph{separates the root and the
    anchor in a cluster~$c$} if (i) \emph{before} this operation was applied the
anchor of $c$ was the black endpoint of the root of $c$, and (ii) \emph{after}
this operation is performed this is no longer the case. It is easy to check
that the following simple lemma holds true.

\begin{lemma}
    \label{lem:pomoc1}
The bend operation does not separate the root and the anchor in any cluster.
\end{lemma}

\subsection{The building blocks of the second step: jump $\Jump_{x,y}$}
\label{sec:jump}

\begin{figure}
    \centering \subfloat[]{\label{subfig:jumpA}
    {\includegraphics[width=0.4\textwidth]{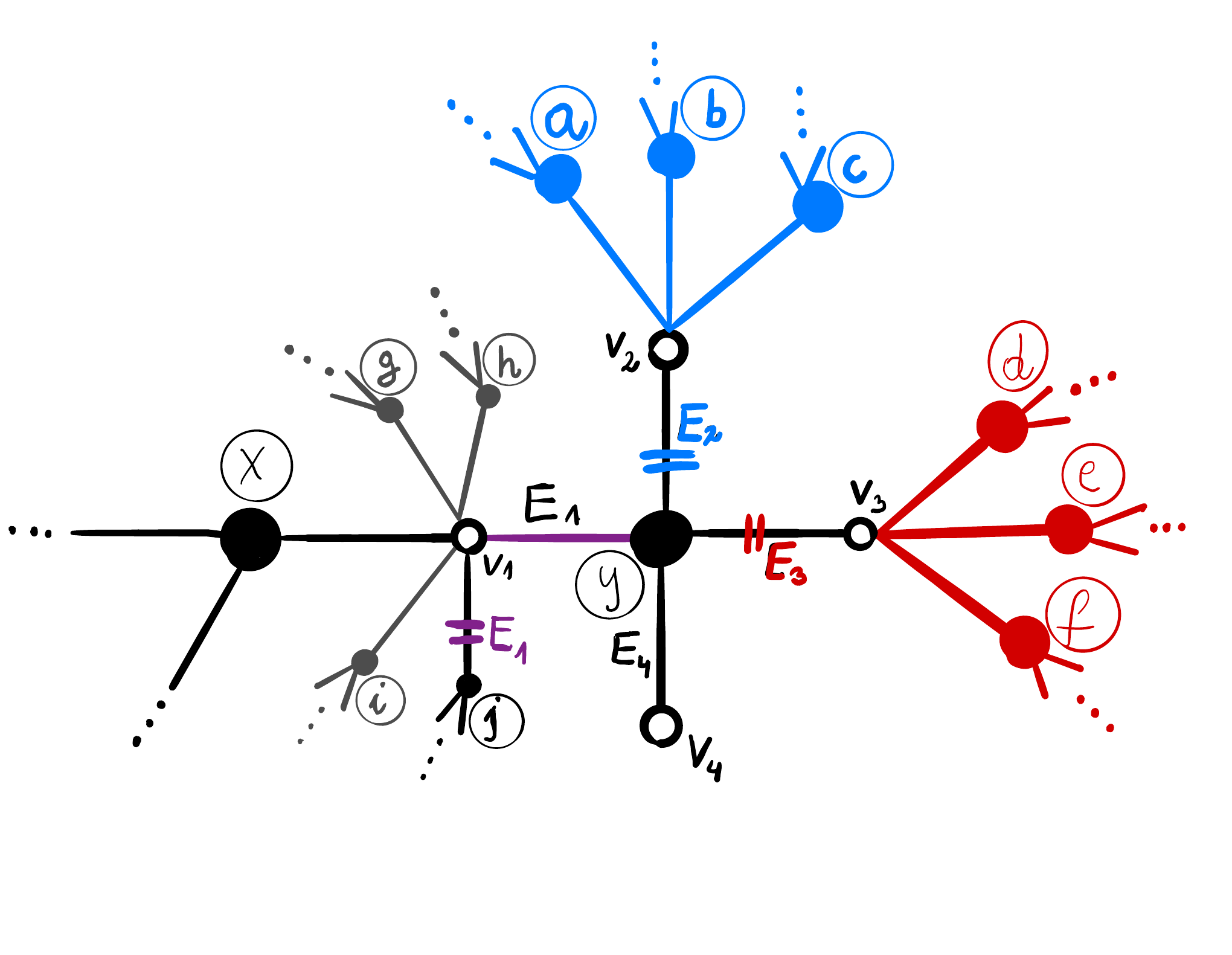}}  } \quad
\subfloat[]{\label{subfig:jumpB}
    {\includegraphics[width=0.4\textwidth]{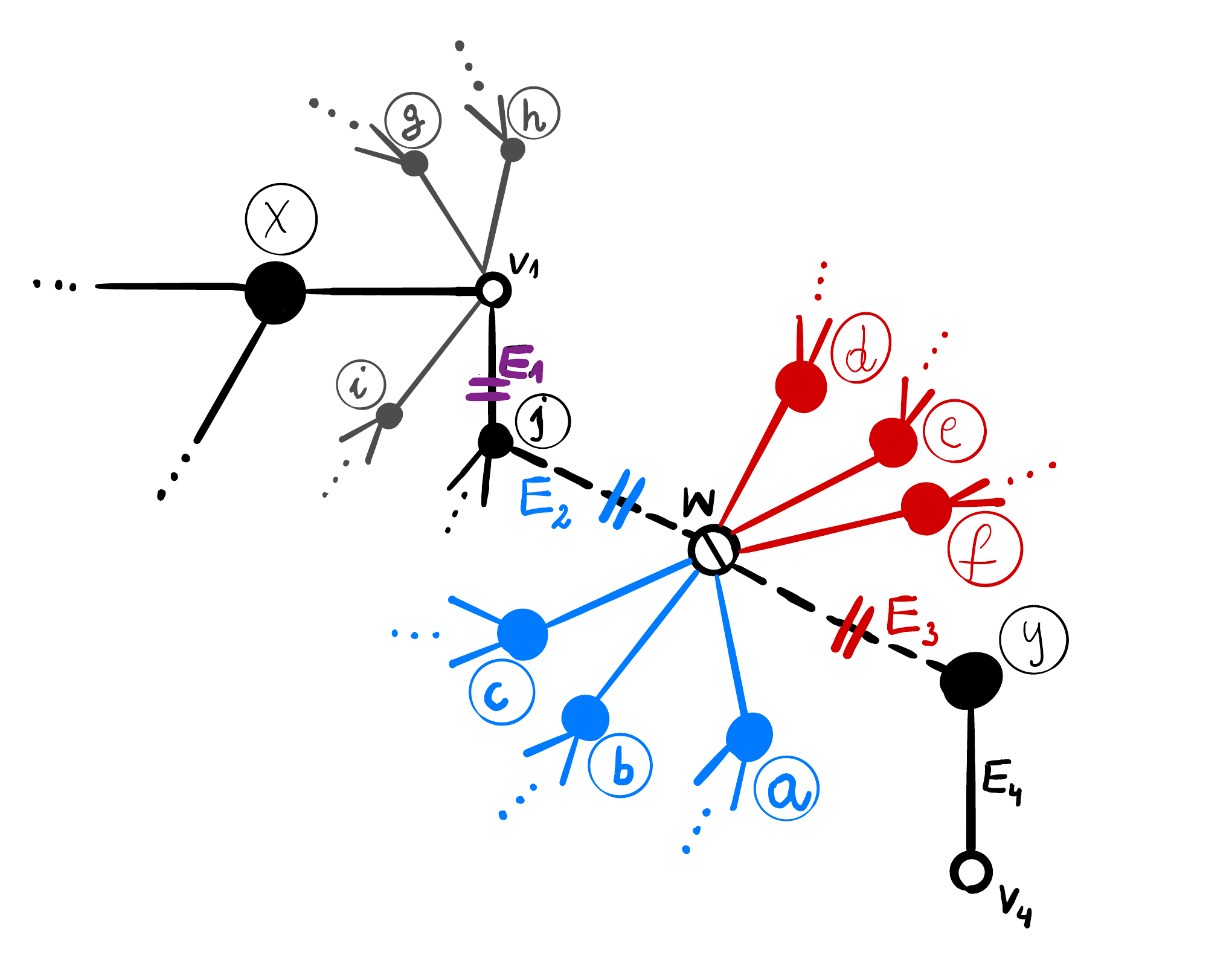}}}
\caption{\protect\subref{subfig:jumpA} The initial configuration of the
    tree before the jump operation~$\Jump_{x,y}$ was applied. The black
    vertex $y$ belong to the cluster $E_1$. The center of $E_1$ is
    equal to $v_1$, the anchor of this cluster is equal to $x$. The vertex $j$
    is the first immediately after $y$ if we visit the neighbors of $v_1$ in
    the clockwise order; we assume that $j$ is the black endpoint of the root
    of $E_1$. The neighbors of~$y$ listed in the clockwise order are:
    $v_1,v_2,\dots, v_d$. This figure depicts the special case when $j\neq x$. 
    In this case the assumption \ref{invariant:4} applied to the cluster $E_1$ 
    guarantees that the vertex $j$ is touched.  
    For the full list of assumptions and notations see
    \cref{sec:assumptions-j}. 
    \newline
    \protect\subref{subfig:jumpB}~The output of
    $\Jump_{x,y}$. The edges connecting $y$ with $v_1,v_2,v_3$ were removed.
    The vertices $v_2$ and $v_3$ were replaced by a new vertex $w$. The vertex
    $w$ is connected by new edges to $j$ and $y$. In a typical application, in the initial
    configuration in the cluster $E_2$ the anchor as well as the black endpoint
    of the root are both equal to $y$; then in the output the anchor of $E_2$
    is no longer the endpoint of the root. For a detailed description see
    \cref{sec:output-of-jump}.} \label{fig:jump}

    \centering \subfloat[]{\label{subfig:scjumpA}
    {\includegraphics[width=0.4\textwidth]{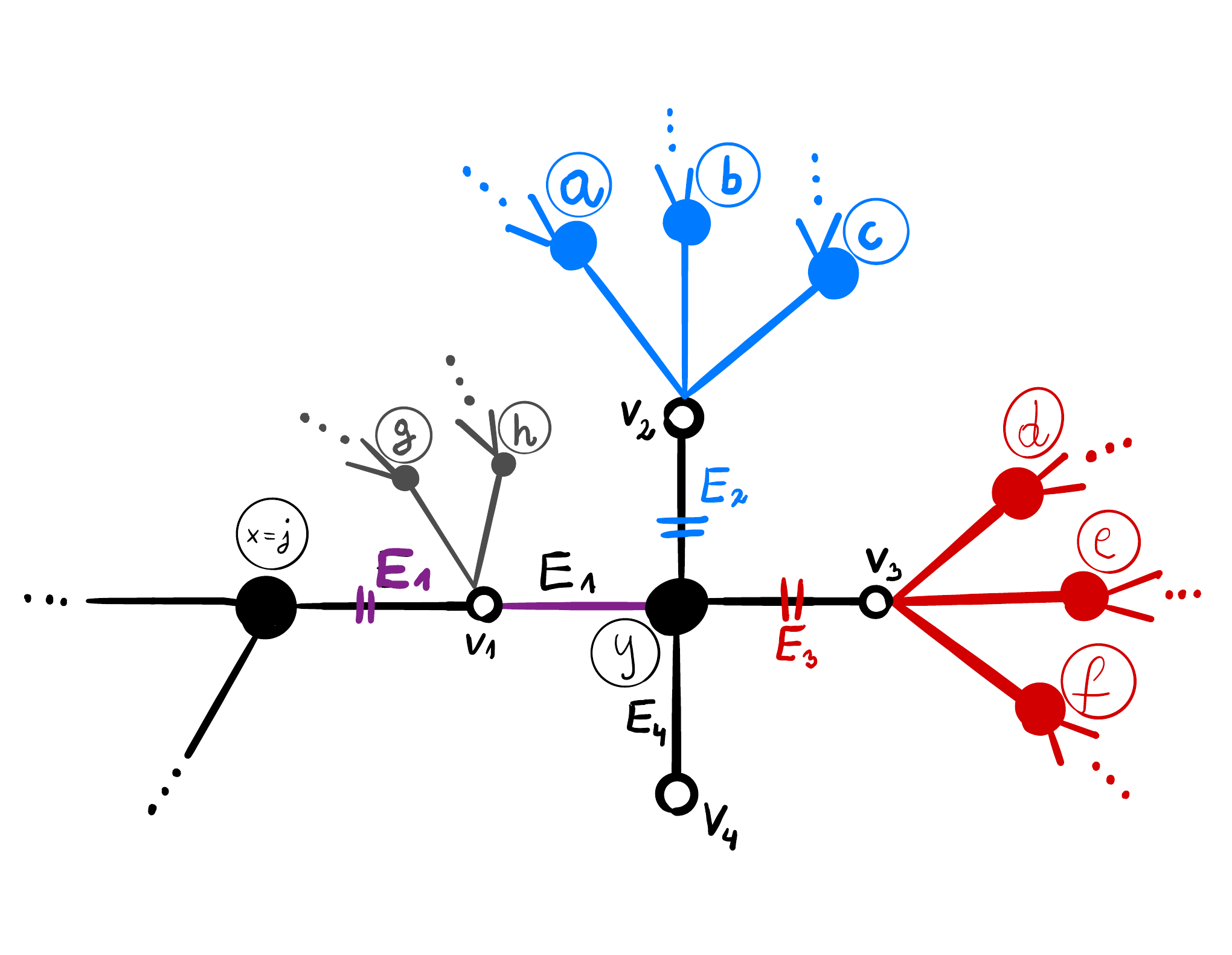}}  } \quad
\subfloat[]{\label{subfig:scjumpB}
    {\includegraphics[width=0.4\textwidth]{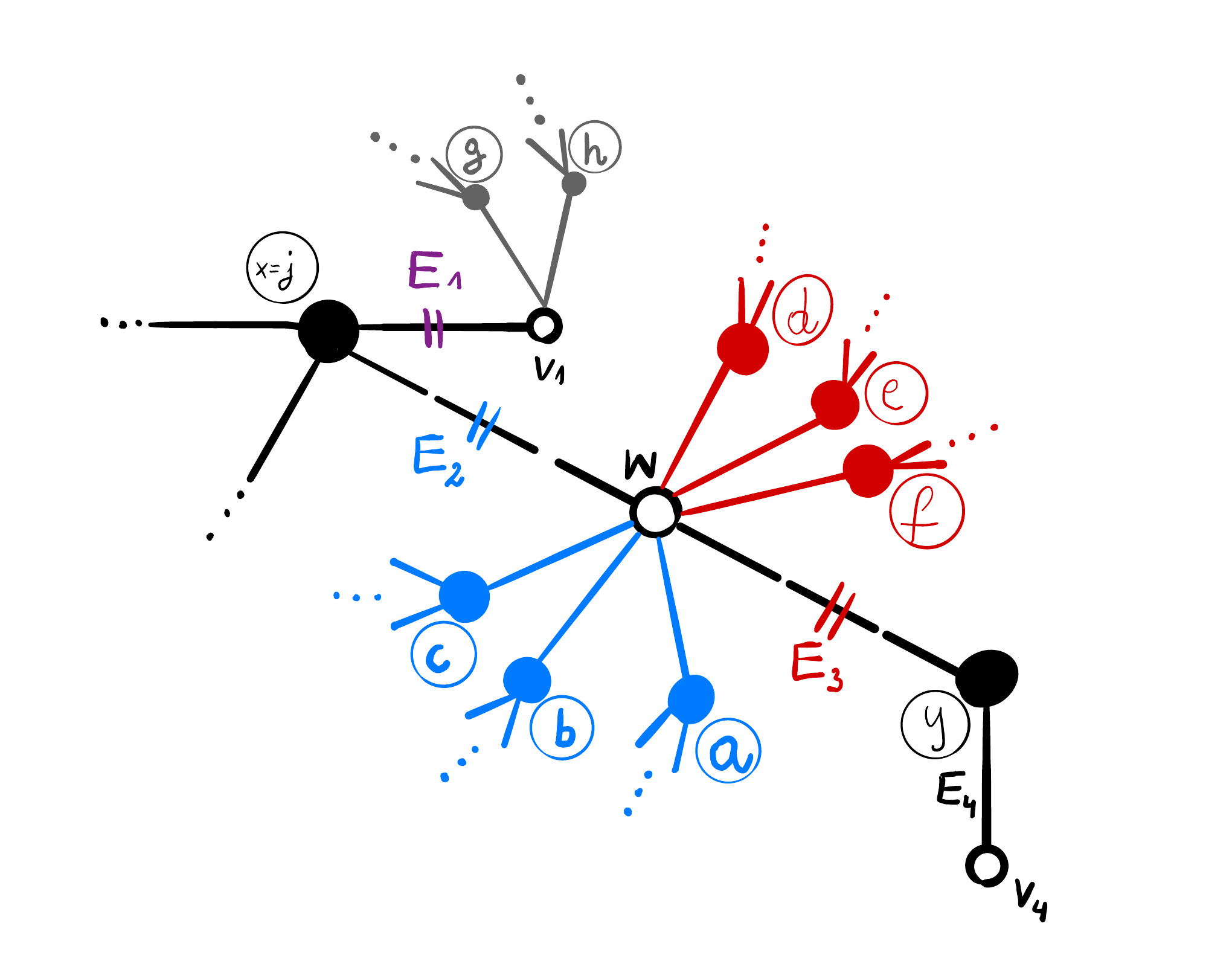}}}
\caption{\protect\subref{subfig:scjumpA} The initial configuration of the
    tree before the jump operation $\Jump_{x,y}$ was applied in the special
    case when $j=x$. For the notations see the caption of
    \protect\cref{fig:jump}. The vertex $x=j$ is assumed to be touched.
    \newline \protect\subref{subfig:scjumpB}~The output of $\Jump_{x,y}$ in the
    special case when $j=x$.} \label{fig:scjump}
\end{figure}

\subsubsection{Assumptions about the input of $\Jump_{x,y}$.}
\label{sec:assumptions-j}

We list below the assumptions about the input of the operation \emph{jump}. We
also introduce some notations. 

\begin{enumerate}[label=(J\arabic*)]
    \item \label{AJ1} The operation $\Jump_{x,y}$ takes as an input a bicolored
tree $T$ that is assumed to be as in \cref{invariant}, together with a
choice of two black vertices $x$, $y$. We assume that there is a cluster
$E_1$ such that $x$ is the anchor of $E_1$ and $y$ belongs to $E_1$; see
\mbox{\cref{subfig:jumpA,subfig:scjumpA}}. We denote by $v_1$ the center of
the cluster $E_1$. We also assume that the labels of the vertices
$x,y\in\{1,\dots,n\}$ fulfill $x<y$.

    \item \label{AJ2} We denote by $j$ the black neighbor of $v_1$ that---going
clockwise around the vertex~$v_1$---is immediately after $y$ (note that it
may happen that $j=x$; see \cref{subfig:scjumpA}). We assume that $j$ is
the black endpoint of the root of the cluster $E_1$; see
\cref{subfig:jumpA,subfig:scjumpA}.

    \item \label{AJ3}
We also assume that the black vertex $y$ has degree $d\geq 3$; we denote the
edges around the vertex $y$ by $E_1,\dots, E_d$ (going clockwise, starting from
the edge $E_1$). For $i\in\{1,\dots,d\}$ we denote the white
endpoint of the edge $E_i$ by $v_i$.

    \item \label{AJ4}
We assume that the vertex $x$ is touched.
\end{enumerate}

\subsubsection{The output of $\Jump_{x,y}$}
\label{sec:output-of-jump}

The output of $\Jump_{x,y}$ is defined as follows. We remove the three edges
connecting $y$ with the three vertices $v_1,v_2,v_3$. We create a new white
vertex denoted~$w$; this vertex is said to be \emph{artificial}; we will use
this notion later in the analysis of the algorithm.

We connect $w$ to the vertex~$j$ by a new edge that we label $E_2$; more
specifically, going clockwise around $j$ the newly created edge $E_2$ is
immediately after the edge $E_1$. This newly created edge replaces the removed
edge between $y$ and $v_2$, so if this removed edge was the root of the cluster
$E_2$, we declare that the new edge $E_2$ becomes now the new root of the
cluster $E_2$.

We also connect the new vertex $w$ to the vertex $y$ by a new edge that we
label $E_3$; the position of the edge $E_3$ in the vertex $y$ replaces the
three edges that were removed from~$y$. Again, this newly created edge replaces
the removed edge between $y$ and $v_3$, so if this removed edge was the root of
the cluster $E_3$, we declare that the new edge $E_3$ becomes now the new root
of the cluster $E_3$.

We merge the vertices $v_2$ and $v_3$ with the vertex $w$. More specifically,
the clockwise cyclic order of the edges around the vertex $w$ is as follows:
the edge $E_2$, then the edges from the vertex $V_3$ (listed in the clockwise
order starting from the removed edge $E_3$; on \cref{fig:jump,fig:scjump} these
edges are marked red), the edge $E_3$, then the edges from the vertex $v_2$
(listed in the clockwise order starting from the removed edge $E_2$; on
\cref{fig:jump,fig:scjump} these edges are marked blue); see
\cref{subfig:jumpB,subfig:scjumpB}.

From the following on we declare that the cluster $E_1$ is \emph{touched} and
also the black vertex $y$ is \emph{touched}.

\medskip

It is easy to check that the following simple lemma holds true.
\begin{lemma}\label{lem:pomoc2}
The jump operation separates the root and the anchor only in (at most) a single
cluster. This potentially exceptional cluster is the one that with the
notations of \cref{fig:jump,fig:scjump} is denoted by $E_2$, and this
separation occurs if and only if in the initial configuration both the anchor
of $E_2$ as well as the black endpoint of the root of $E_2$ are equal to $y$.
\end{lemma}

In fact, we will use the jump operation only in the context when it indeed
separates the root and the anchor in $E_2$. In this context $E_2$ will turn out
to be a leftist cluster.

\medskip

It is easy to check that the output tree still fulfills the properties from
\cref{invariant}. The only slightly challenging part of the proof concerns
verifying that the condition \ref{invariant:4} holds true for the cluster $E_2$
if the separation of the root and the anchor occurs. For this purpose we need
to show that vertex $j$ is touched. In the case when $j\neq x$ we use the
assumption that \ref{invariant:4} for the cluster $E_1$ was valid in the
initial configuration of the tree and hence $j$ was touched, as required. In
the remaining case when $j=x$ we use assumption \ref{AJ4}.

\subsection{The spine treatment}
\label{sec:spine}

\begin{figure}
    \centering
    \subfloat[]{\label{subfig:untouched1}
        \includegraphics[clip,trim=0cm 8cm 0cm 0cm,width=0.45\textwidth]{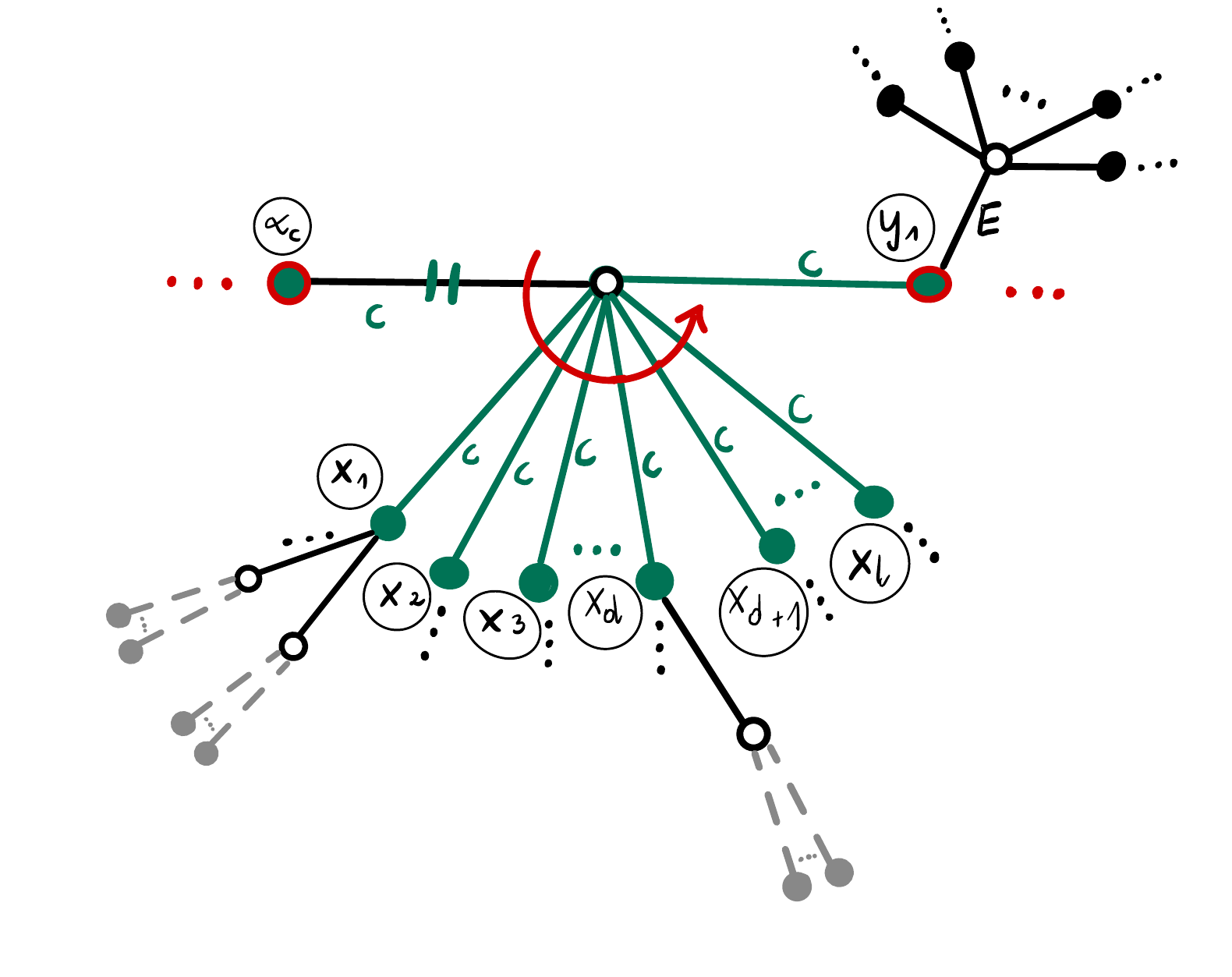}}
    \quad
    \subfloat[]{\label{subfig:untouched2}
        \includegraphics[clip,trim=0cm 8cm 0cm 0cm,width=0.45\textwidth]{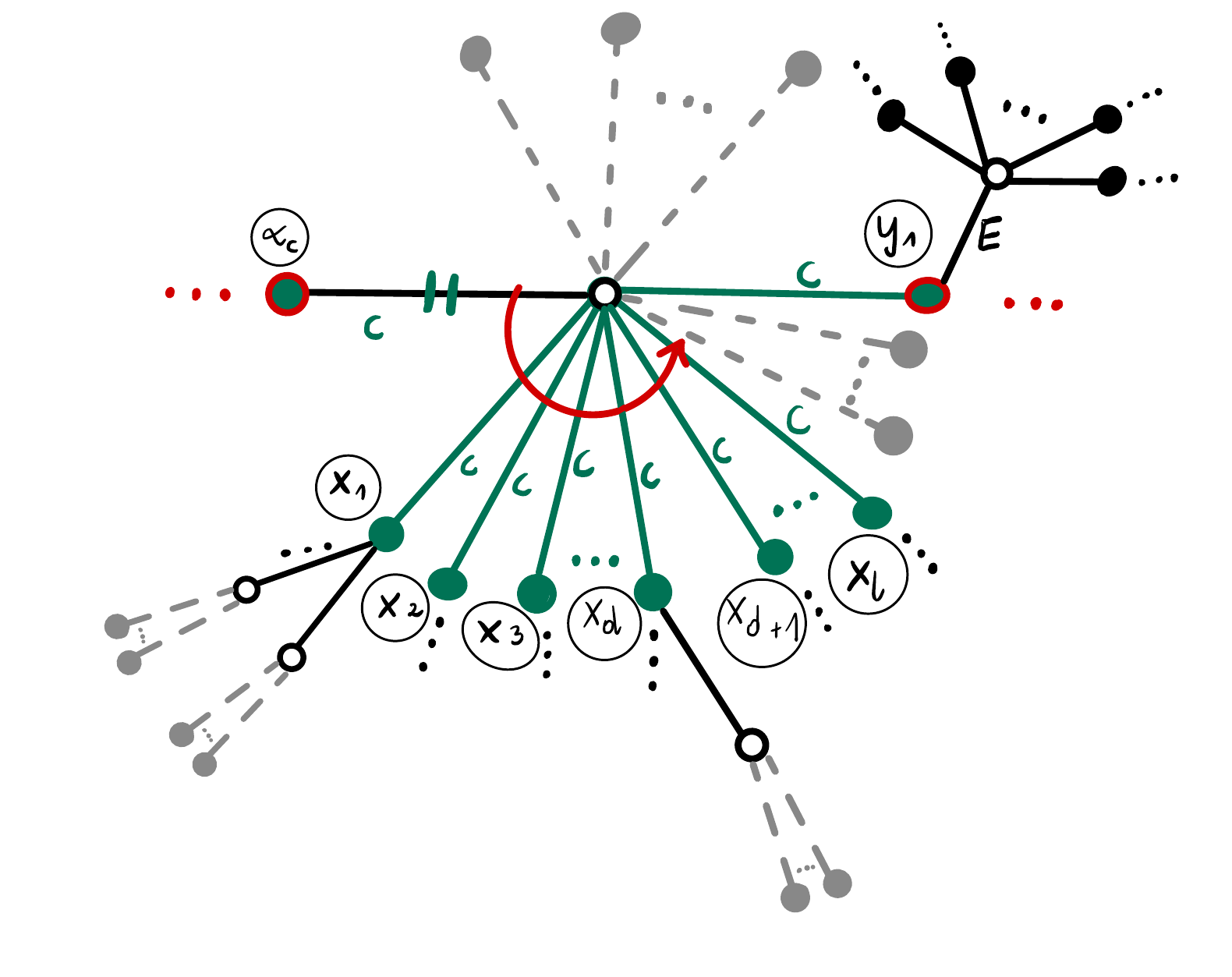}}
    \caption{
        \protect\subref{subfig:untouched1} A spine cluster $c$ of the tree $T_1$ corresponding to 
        the part of the tree $T_0$ on \cref{fig:untouched0}.
		The black vertices $\Root_c,y_1$ with $\Root_c<y_1$ are spine vertices. 
		The vertex $\Root_c$ is the anchor of the cluster $c$ as well as the black endpoint
		of the root of the cluster $c$. 
        The non-spine vertices in the cluster $c$ are denoted by $x_1,\dots,x_l$ with $l\geq 0$. 
        If $l>0$, we denote by $x_{d}=\min B_c\setminus\{\Root_c,y_1\}$ 
        the vertex with the minimal label among non-spine vertices in the cluster.
        \protect\subref{subfig:untouched2}~The structure of the cluster $c$ 
        at a later stage of the algorithm as long as it remains untouched. New edges attached 
        to the center of $c$ might have been created on either side of the edge connecting 
        the center with $y_1$. 
        See \cref{lem:invariants-ok} \ref{item:zal2} for details.}

    \label{fig:untouched}  
\end{figure}

For each spine cluster $C\in \mathfrak{C}$ we apply the following procedure
(the final output will not depend on the order in which we choose the clusters
from $\mathfrak{C}$). In the language of programmers we run the \emph{external
    loop} (or the \emph{main loop}) over the variable~$C$. Since the spine in the
tree $T_1$ is a path, the intersection $B_C\cap \mathcal{B}$ corresponds to the
labels of the two black spine vertices of the cluster $C$ in the tree $T_1$.
One of these two vertices is the anchor $\Root_C$ of the cluster, we denote the
other one by $y_1$; in this way $\Root_C<y_1$. As we will prove later (see
\cref{lem:invariants-ok}, property \ref{item:zal2}), at the time of the
execution of this particular loop iteration, the spine cluster $C$ is of the
form shown on \cref{subfig:untouched2}.

We run the following \emph{internal loop} over the variable $y\in
B_C\setminus\{\Root_C\}$ (with the ascending order). If  $y=y_1$ or if the
vertex labels $y,\Root_C\in\{1,\dots,n\}$ fulfill $y<\Root_C$ then we apply
$\Bend_{\Root_C,y}$;  otherwise we apply $\Jump_{\Root_C,y}$.

\begin{example}
We continue the example from \cref{subfig:ex2}. We recall that $\mathfrak{C}=\{1,2,16\}$.

\begin{figure}
    \centering
    \subfloat[]{\label{subfig:ex3}
        \includegraphics[width=0.47\textwidth]{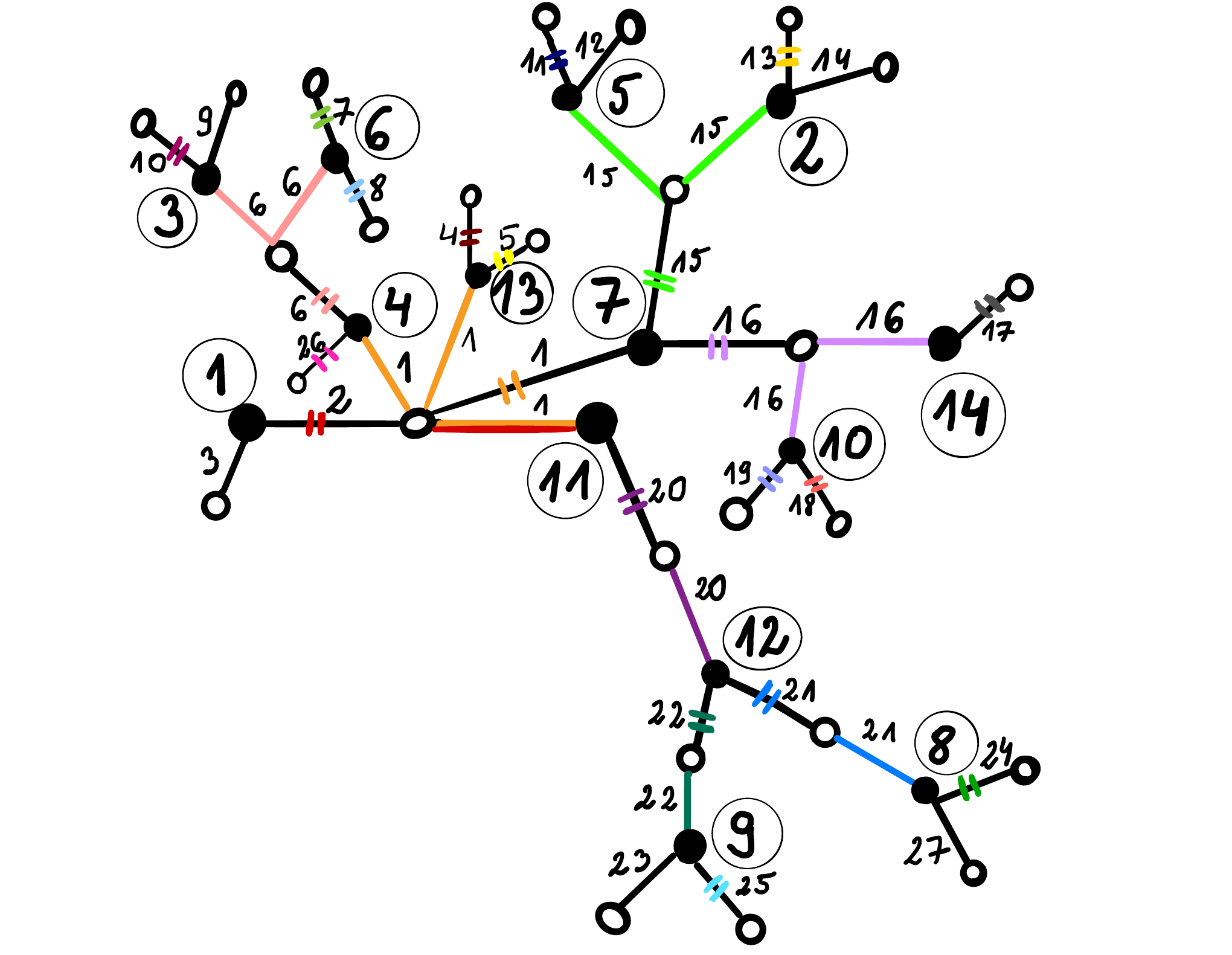}}
    \quad
    \subfloat[]{\label{subfig:ex4}
        \includegraphics[width=0.47\textwidth]{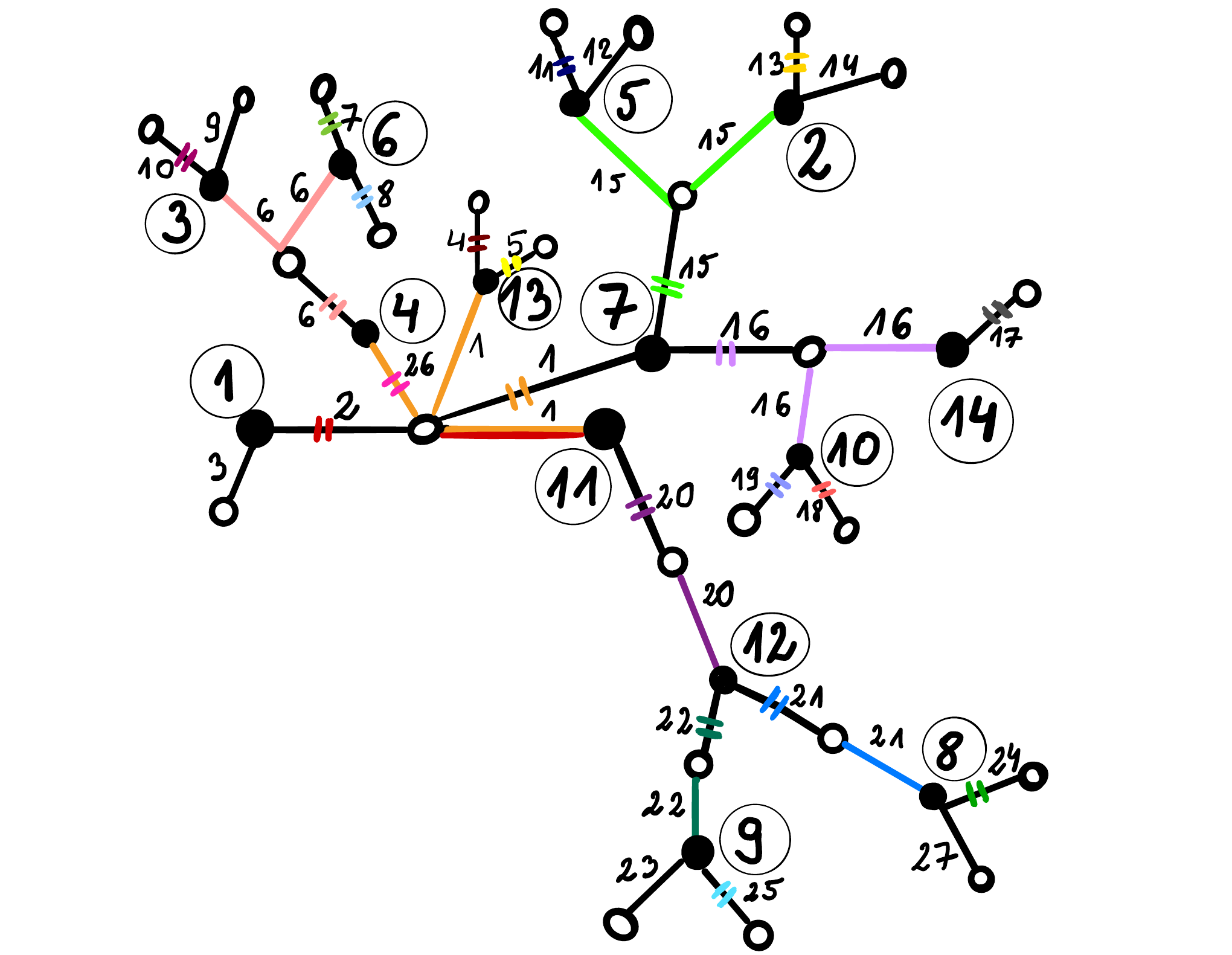}}
    \caption{
    \protect\subref{subfig:ex3} The output of $\Bend_{1,11}$ applied to the tree from \protect\cref{fig:ex2B}. 
    \newline
    \protect\subref{subfig:ex4} The output of $\Bend_{7,4}$ applied to the tree from \protect\subref{subfig:ex3}.}
    \label{fig:ex34}
\end{figure}

For $C=2$ we recall that $\Root_2=1$ so we have $y_1=11$. Since $B_2\setminus \{1\}=\{11\}$, the
internal loop is applied once with $y=11$. As a result we apply $\Bend_{1,11}$; see
\cref{subfig:ex3}.

\begin{figure}
    \centering
    \subfloat[]{\label{subfig:ex5}
        \includegraphics[width=0.47\textwidth]{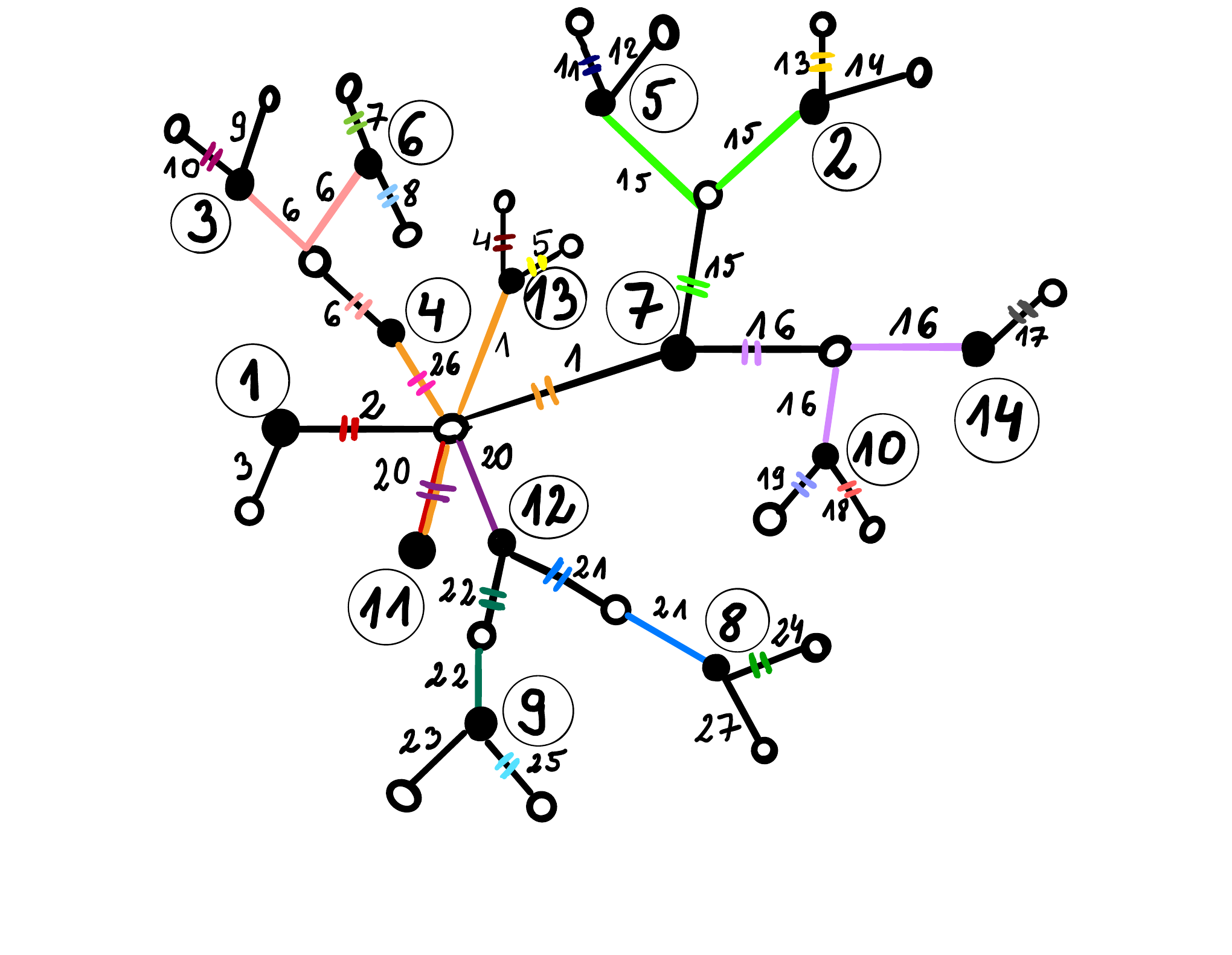}}
    \quad
    \subfloat[]{\label{subfig:ex6}
        \includegraphics[width=0.47\textwidth]{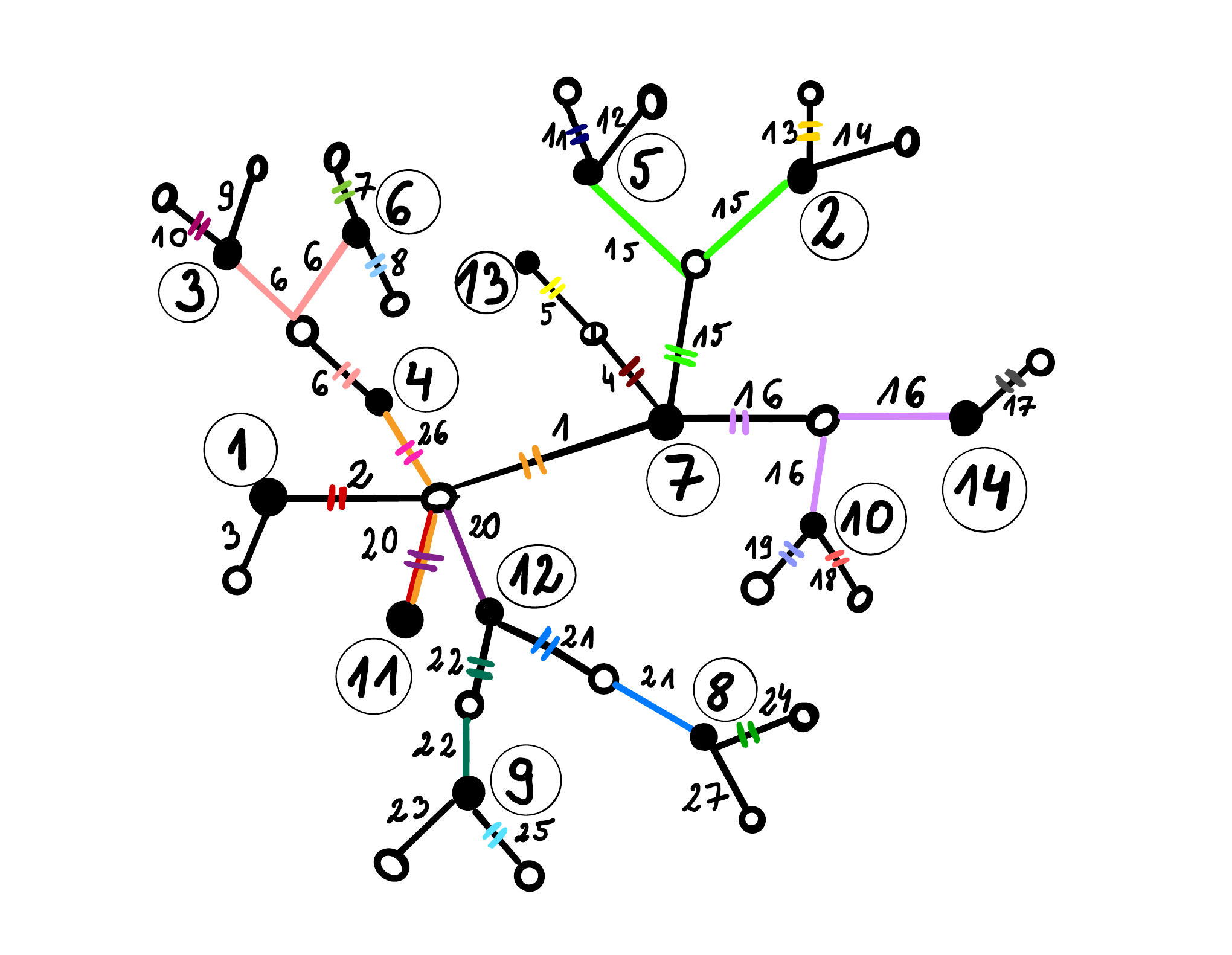}}
    \caption{
    \protect\subref{subfig:ex5} The output of $\Bend_{7,11}$ applied to the tree from \protect\cref{subfig:ex4}.
    \newline
     \protect\subref{subfig:ex6} The output of $\Jump_{7,13}$ applied to the tree from \protect\subref{subfig:ex5}.}
    \label{fig:ex56}
\end{figure}

For $C=1$ we recall that $\Root_1=7$ so we have $y_1=11$.
Since $B_1\setminus \{7\}=\{4,11,13\}$ the internal loop runs over:
$\bullet$ $y=4$
and  we apply $\Bend_{7,4}$,
see \cref{subfig:ex4};
$\bullet$~$y=11$ and we apply $\Bend_{7,11}$; see \cref{subfig:ex5};
$\bullet$~$y=13$ and we apply $\Jump_{7,13}$; see \cref{subfig:ex6}.

\begin{figure}
    \centering
    \subfloat[]{\label{subfig:ex7}
        \includegraphics[width=0.47\textwidth]{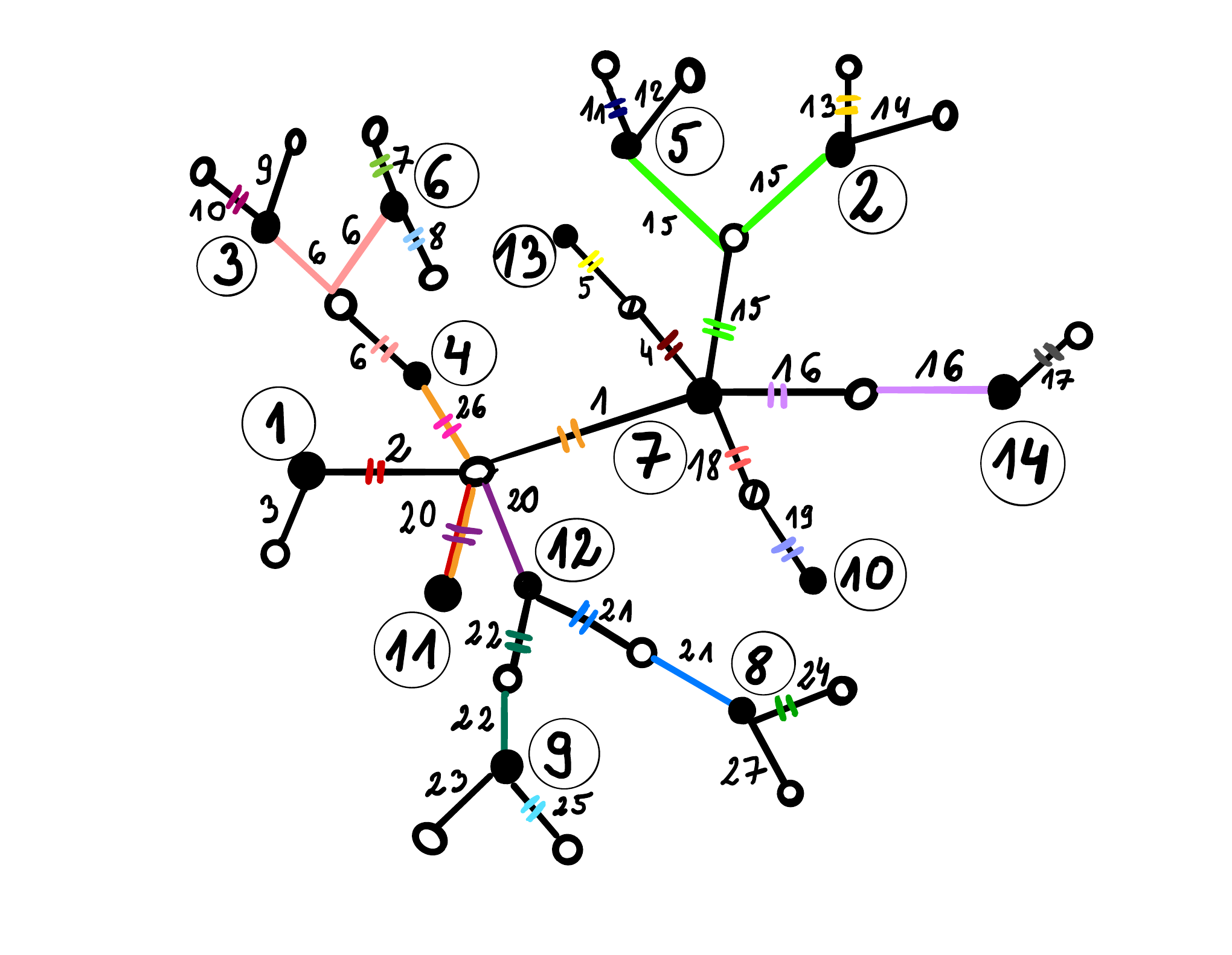}}
    \quad
    \subfloat[]{\label{subfig:ex8}
        \includegraphics[width=0.47\textwidth]{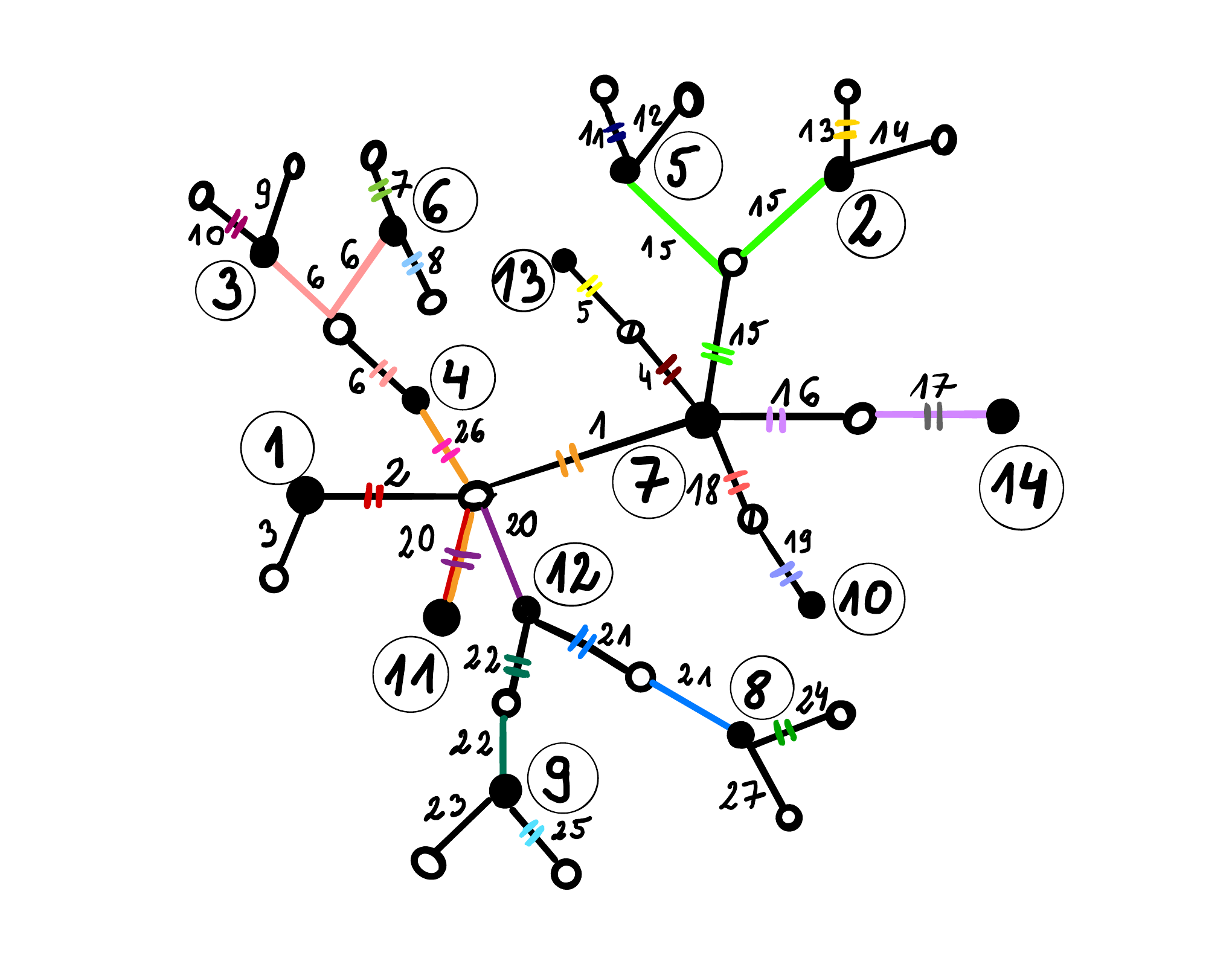}}
    \caption{
    \protect\subref{subfig:ex7} The output of $\Jump_{7,10}$ applied to the tree from \protect\cref{subfig:ex6}.
    \linebreak
     \protect\subref{subfig:ex8} The output of $\Bend_{7,14}$ applied to the tree from \protect\subref{subfig:ex7}.}
    \label{fig:ex78}  
\end{figure}

For $C=16$ we recall that $\Root_{16}=7$ so we have $y_1=14$. 
Since $B_{16}\setminus \{7\}=\{10,14\}$ the internal loop runs over:
$\bullet$~$y=10$ and we apply $\Jump_{7,10}$; see \cref{subfig:ex7};
$\bullet$~$y=14$ and we apply $\Bend_{7,14}$; see \cref{subfig:ex8}.

\end{example}

\subsection{Correctness of the spine treatment algorithm}
\label{sec:correctness-spine}

Before we present the remaining part of the algorithm (in \cref{sec:rib}) we will
show that the above spine treatment is well-defined in
the sense that the assumptions for the bend and jump operations
(\cref{sec:assumptions-e,sec:assumptions-j}) are indeed fulfilled during the spine
treatment. We will show the following stronger result.

\begin{propos}
    \label{lem:invariants-ok} 
    
    At each step of the spine treatment algorithm, the
current value of the tree~$T$ fulfills the following properties.
    \begin{enumerate}[label=(P\arabic*)]
        \item \label{item:zal1}
    The tree $T$
fulfills the properties described in \cref{invariant}. 

\item \label{item:zal2} 
Let $c$ be a spine cluster of the initial tree $T_1$; we
denote by $\Root_c,y_1$ the black spine vertices in this cluster with $\Root_c<y_1$.
Assume that the cluster $c$ is still untouched in the tree $T$.

\smallskip

Then the vertex $\Root_c$ is both the anchor of the cluster $c$ as well as the black
endpoint of the root of the cluster $c$.

\smallskip

Additionally, we compare:
\begin{enumerate}[label=(\roman*)]
    \item \label{item:cluster-i} the cyclic order of the black vertex labels of
the edges surrounding the center of the cluster $c$ in the original tree
$T_1$ (i.e.~the cyclic order of the black endpoints' labels of the cluster
$c$ in~$T_1$); see \cref{subfig:untouched1}
\end{enumerate}
with
\begin{enumerate}[resume,label=(\roman*)]
    \item \label{item:cluster-ii}
    the cyclic order of the black endpoints' labels of the edges surrounding
    the center of the cluster $c$ in the tree $T$; see \cref{subfig:untouched2}.
\end{enumerate}
Then \ref{item:cluster-ii} is obtained from \ref{item:cluster-i} by adding some
additional vertices that do not belong to the cluster $c$; these additional
vertices occur on either side of the edge connecting the center of the cluster
$c$ with $y_1$; see \cref{fig:untouched}.

\item \label{item:zalNEW} The set of clusters which are \emph{touched} coincides
with the set of values that the variable $C$ (in the external loop of the spine
treatment algorithm) took in the past. 

\item \label{item:zal3} 
For each \emph{bend} operation (respectively, for each \emph{jump} operation)
performed during the spine treatment algorithm the assumptions
\ref{AE1}--\ref{AE3} (respectively, the assumptions \ref{AJ1}--\ref{AJ4}) are
fulfilled; in this way each \emph{bend}/\emph{jump} operation is well-defined.

\end{enumerate}
\end{propos}
\begin{proof}
In the first part of the proof we will show that \ref{AE3}, \ref{AJ3} and \ref{AJ4} are
fulfilled in each step of the spine treatment. 

\medskip

Note that during the spine treatment algorithm we perform operations
$\Bend_{x,\cdot}$ and $\Jump_{x,\cdot}$ for $x\in\mathcal{B}$. From the very
beginning each black spine vertex is touched, so the assumption \ref{AJ4} is
fulfilled in each step of the spine treatment.

\medskip

In order to show \ref{AE3} and \ref{AJ3}
we will use the observation that when one of the operations
$\Bend_{\cdot,y}$ or $\Jump_{\cdot,y}$ is performed, the degree of each black
vertex that is different from $y$ increases or remains the same. 

Consider some black vertex $y$ of the initial tree $T_1$ that does not belong
to the spine. In this case the vertex $y$ belongs to at most one spine cluster,
so during the spine treatment algorithm we perform at most one operation of the
form $\Bend_{\cdot,y}$ or $\Jump_{\cdot,y}$. Therefore, the degree of $y$ (at
the time when this unique $\Bend_{\cdot,y} / \Jump_{\cdot,y}$ operation is
performed) is at least its degree in the initial tree $T_1$. The latter degree, by
construction, is equal to the length of the cycle $\sigma_y$ so it is equal to
the parameter $a_y$ that appears in the statement of \cref{thm:thm1}. By
\eqref{eq:a-and-b} we get $a_y=b_y+2\ge 3$, as required.

Consider now some black vertex $y$ of the initial tree $T_1$ that belongs to
the spine. In this case the vertex $y$ belongs to at most two spine clusters.
It follows that we perform no operations of the form $\Jump_{\cdot,y}$ at all,
and we perform at most two operations of the form~$\Bend_{\cdot,y}$. In the
case when we perform a single such an operation $\Bend_{\cdot,y}$, the
assumption \ref{AE3} is fulfilled because the degree of $y$ at the time when
this operation is performed is at least $a_y\ge b_y+1\geq 2$, by the same
argument as in the previous paragraph. The other case when two such bend
operations $\Bend_{\cdot,y}$ are performed can occur only if $y\notin\{1,n\}$
is not one of the endpoints of the spine; it follows therefore that the initial
degree of the vertex $y$ is equal to $a_y=b_y+2\ge 3$. The bend operation
$\Bend_{\cdot,y}$ decreases the degree of the black vertex $y$ by $1$.
Therefore, at the time when the first operation $\Bend_{\cdot,y}$ performed,
the degree of $y$ is at least $a_y\ge 3$ while when the second such an
operation is performed, the degree $y$ is at least $a_y-1\ge 2$, as required.
This concludes the proof that the assumptions \ref{AE3} and \ref{AJ3} are
fulfilled in each step of the spine treatment algorithm.

\bigskip

In the second part of the proof, in order to show
\ref{item:zal1}--\ref{item:zal3} we will use induction over the variable $C$ in
the main loop of the spine treatment algorithm. Our inductive assumption is
that at the time when the iteration of the main loop starts or ends, the
properties \ref{item:zal1}--\ref{item:zalNEW} are fulfilled.

\medskip

\emph{The induction base.} The input tree $T_1$ clearly fulfills the conditions 
\ref{item:zal1}--\ref{item:zalNEW}.

\smallskip

\label{page:inductive-starts} \emph{The inductive step.} We consider some
moment in the action of the spine treatment algorithm when a new iteration of
the external loop begins. By the inductive assumption the current value of the
tree $T$ fulfills the conditions \ref{item:zal1}--\ref{item:zalNEW}. We denote
by $C$ the current value of the variable over which the external loop runs. The
assumption \ref{item:zalNEW} implies that the cluster $C$ is untouched; it
follows that the assumption \ref{item:zal2} is applicable to this cluster and
hence the cluster $C$ has the form shown on \cref{subfig:untouched2}. We denote
by $x_1,\dots,x_l$ with $l\geq 0$ the non-spine vertices of the cluster listed
in the counterclockwise order, cf.~\cref{fig:untouched}. If $l>0$ we denote by
$x_d=\min B_C\setminus\{\Root_C,y_1\}$ the vertex with the minimal label among
non-spine black vertices in this cluster. In the special case when $l=0$ and
there are no non-spine vertices in the cluster $c$, the vertex $x_d$ is not
well-defined and the following analysis requires very minor adjustments.

\smallskip

We will follow the action of the internal loop (over the variable $y$) and we
will show that in each step the properties \ref{item:zal1}, \ref{item:zal2} and
\ref{item:zal3} are fulfilled. A straightforward analysis shows that the
variable $y$ in the internal loop will take the following values (listed in the
chronological order):
\[ \underbrace{x_{d}, x_{d+1}, \dots, x_l,}_{
 \substack{\text{phase \ref{item:etap1},} \\[1ex] \text{elements that are} \\ \text{ smaller than $\alpha_C$}   }}
    \underbrace{ x_1, x_2, \dots,  x_i, }_{
\substack{ \text{phase \ref{item:etap2},} \\[1ex] \text{elements that are} \\ \text{greater than $\alpha_C$ and smaller than $y_1$}}} 
    \underbrace{y_1,}_{\text{phase \ref{item:etap3}}} \qquad
    \underbrace{x_{i+1}, \dots, x_{d-1}}_{
        \substack{ \text{phase \ref{item:etap4},} \\[1ex] \text{elements that are} \\ \text{greater than $y_1$}}}
\]
(note that in the exceptional case when $d=1$ this list has a slightly
different form that also depends whether $x_1<\Root_C$ or not), therefore the
execution of the internal loop can be split into the following four phases:
\begin{enumerate}[label=(\Roman*)]
\item \label{item:etap1} some number of the bend operations of the form $\Bend_{\Root_C,\cdot}$,
\item \label{item:etap2} some number of jump operations $\Jump_{\Root_C,\cdot}$,
\item \label{item:etap3} one bend operation $\Bend_{\Root_C,y_1}$,
\item \label{item:etap4} some more jump operations $\Jump_{\Root_C,\cdot}$.
\end{enumerate}
For proving \ref{item:zal3} note that the conditions \ref{AE3}, \ref{AJ3} and \ref{AJ4} are
already proved; thus in order to show that all bend/jump operations are
well-defined, it is enough to show \ref{AE1} and \ref{AE2}, respectively
\ref{AJ1}  and \ref{AJ2}. In the following we will analyze the phases
\ref{item:etap1}--\ref{item:etap4} one by one and we will check that it is
indeed the case. The evolution of the cluster $C$ over time in these four phases
is shown on \cref{subfig:untouched2,fig:untouched34,figure:untouched5-6,fig:untouched7}
with $c:=C$.

\smallskip

\begin{figure}
    \centering
    \subfloat[]{\label{subfig:untouched3}
        {\includegraphics[clip,trim=0cm 10cm 0cm 0cm,width=0.45\textwidth]{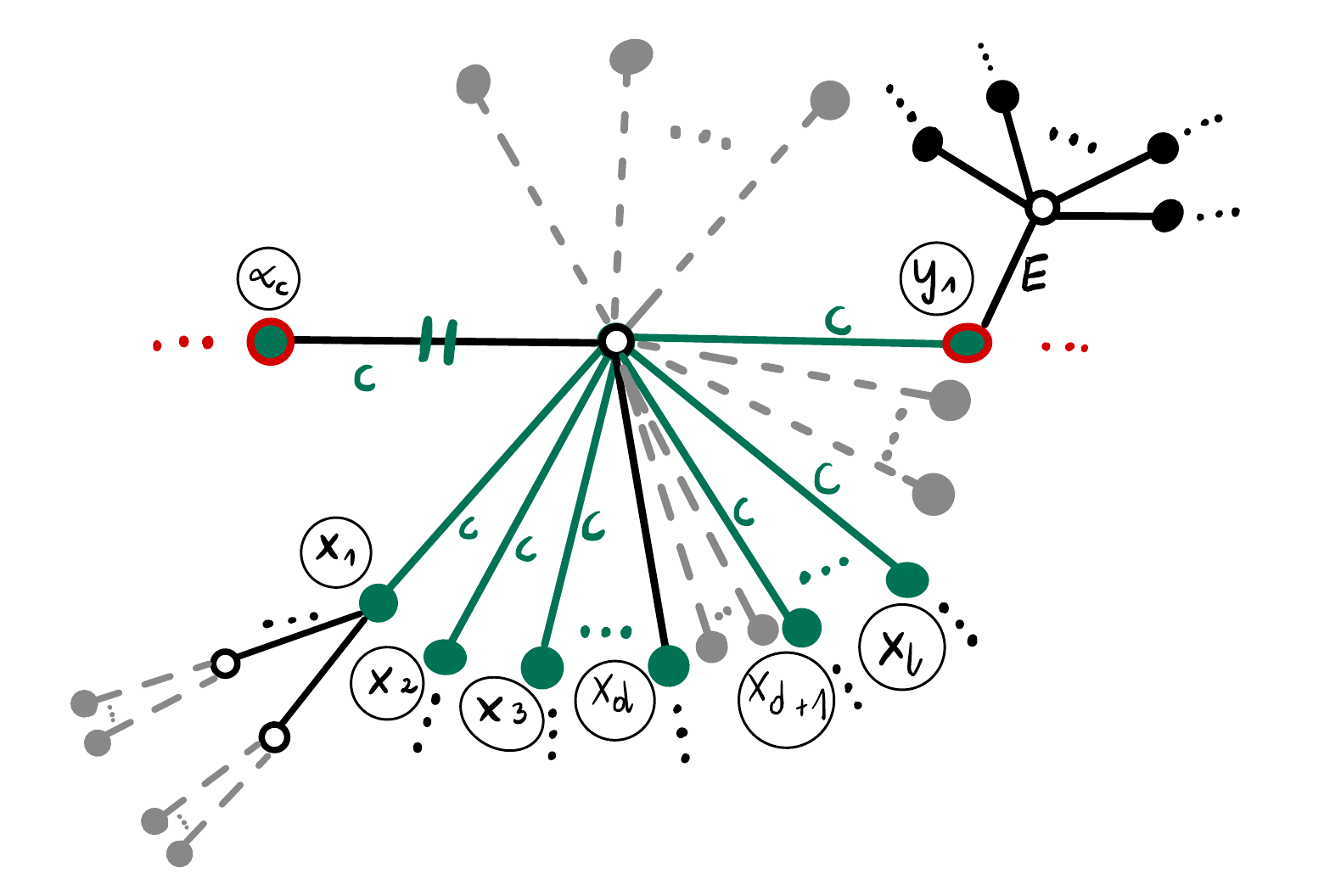}}}
    \quad
    \subfloat[]{\label{subfig:untouched4}
        {\includegraphics[clip,trim=0cm 10cm 0cm 0cm,width=0.45\textwidth]{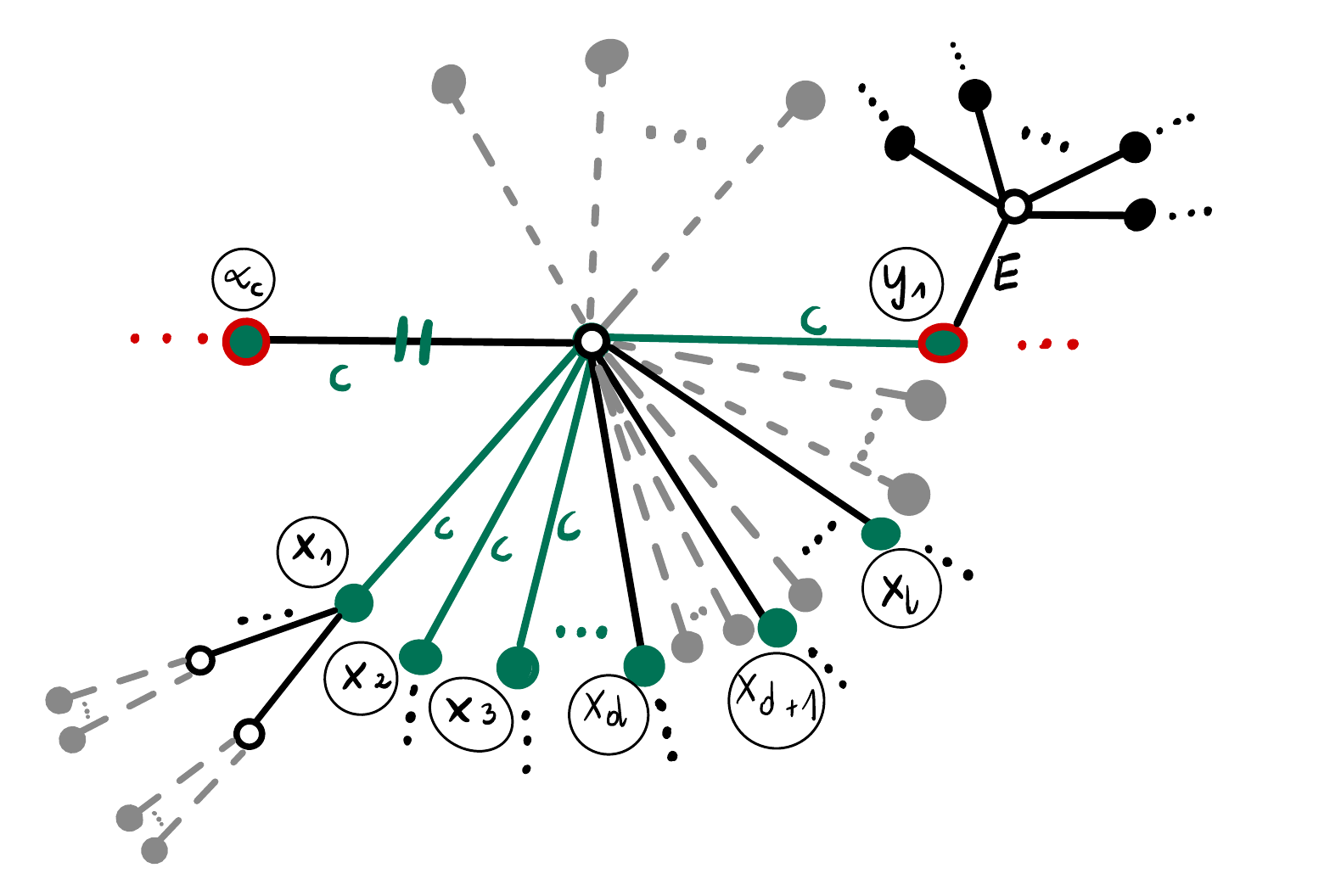}}}
    \caption{
        \protect\subref{subfig:untouched3} The structure of the spine cluster $c$
        from \cref{subfig:untouched2} after performing the first bend operation 
        in the phase \ref{item:etap1}, namely $\Bend_{\Root_c,x_{d}}$.
        This figure was obtained from
        \cref{subfig:untouched2} by adding some additional vertices that do not belong
        to the cluster $c$; going counterclockwise these additional vertices occur after 
        $x_{d}$ and before $x_{d+1}$. 
        \newline
        \protect\subref{subfig:untouched4} The structure of the cluster $c$ 
        after the completion of the phase~\ref{item:etap1}.
        This figure was obtained from
        \protect\subref{subfig:untouched3} by adding some additional vertices that do not belong
        to the cluster $c$; going counterclockwise these additional vertices occur after 
        $x_{d+i}$ and before $x_{d+i+1}$ for each $i\in\{1,\dots,l-d-1\}$,
        as well as after $x_{l}$ and before $y_1$.}
    \label{fig:untouched34}
\end{figure}

\begin{figure}                  
        \centering
    \subfloat[]{\label{subfig:untouched5}
        \includegraphics[clip,trim=0cm 11cm 0cm 0cm,width=0.45\textwidth]{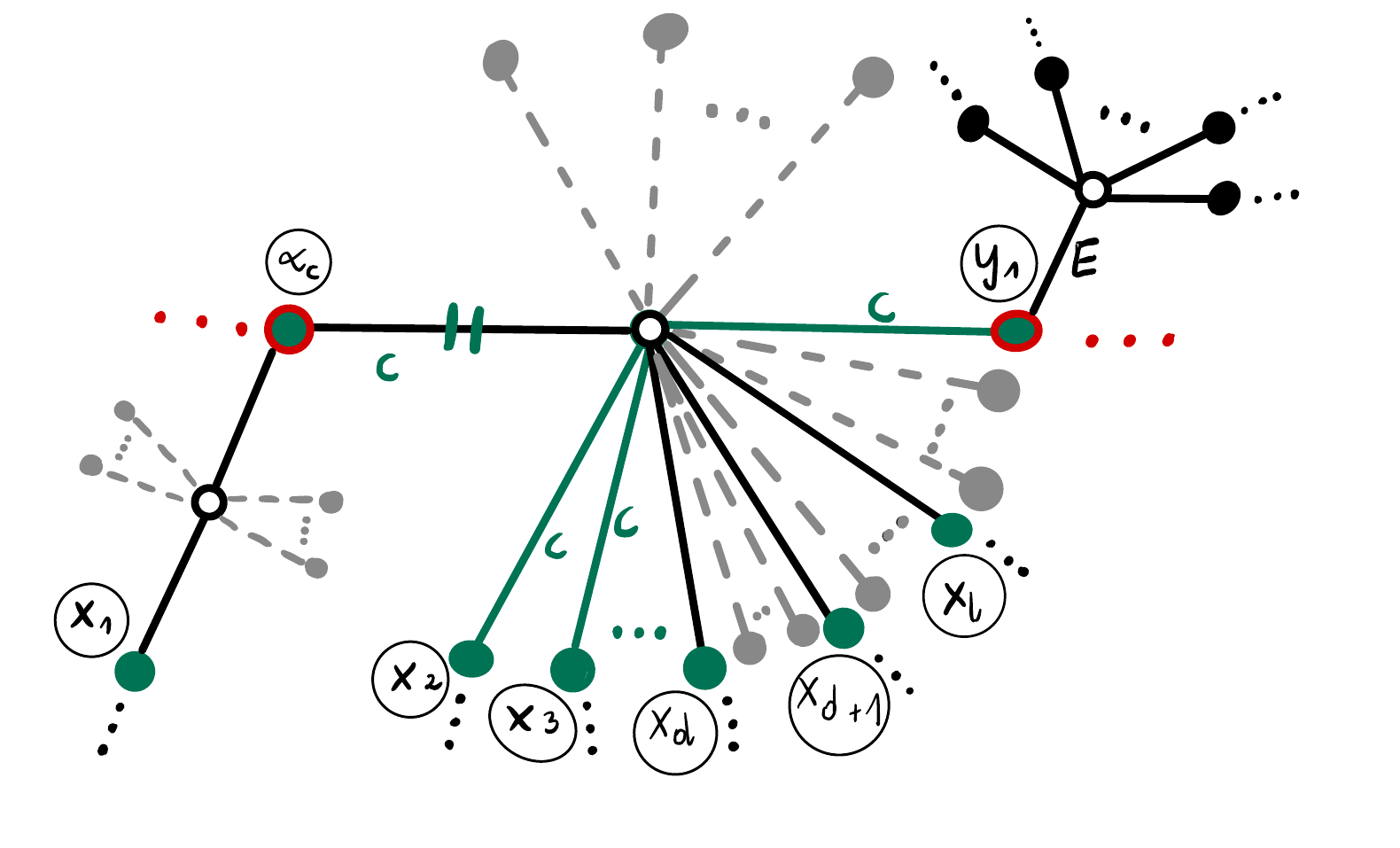}}
    \quad
    \subfloat[]{\label{subfig:untouched6}
        \includegraphics[clip,trim=0cm 11cm 0cm 0cm,width=0.45\textwidth]{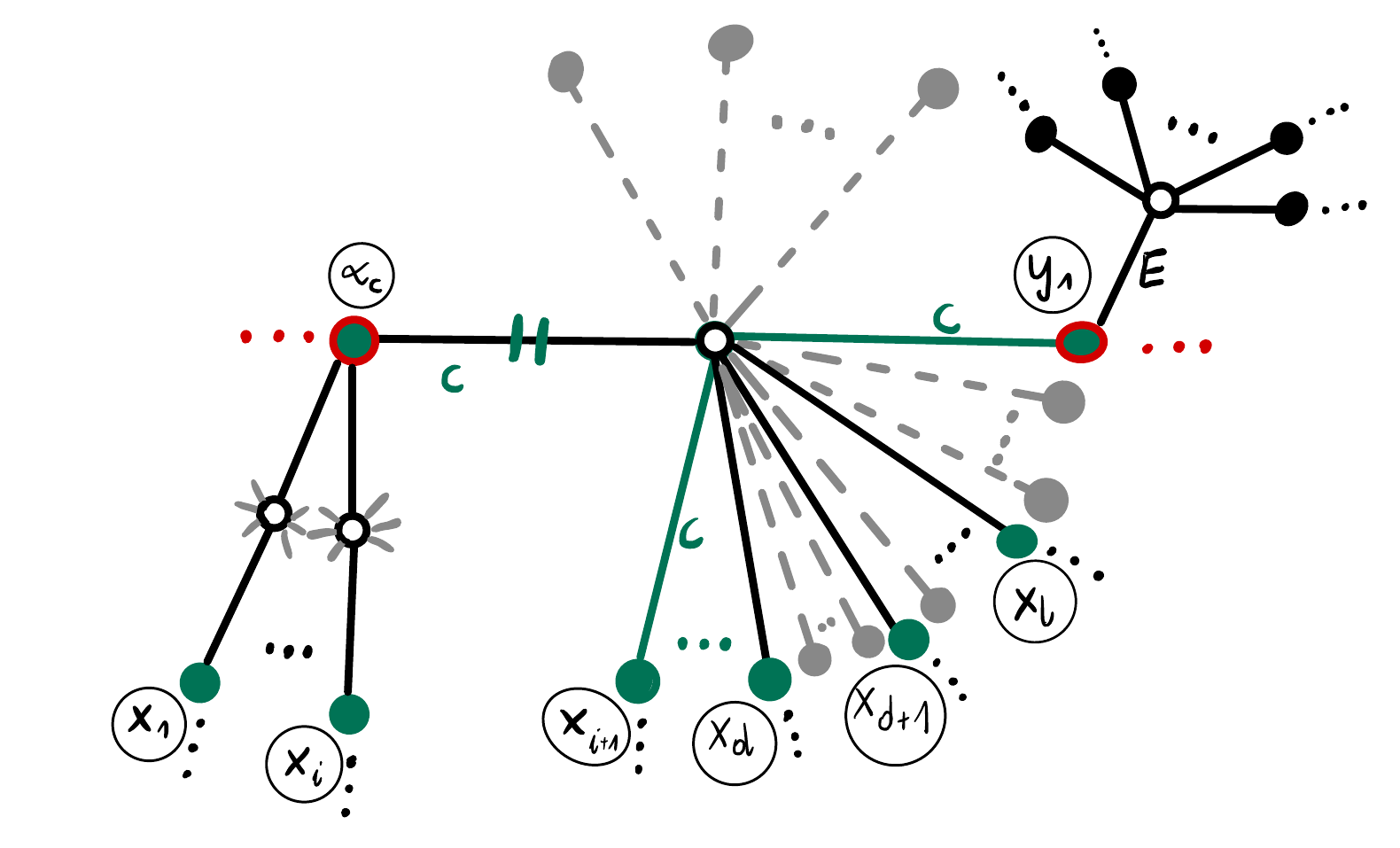}}
    \caption{
        \protect\subref{subfig:untouched5} The structure of the spine cluster $c$
        from \cref{subfig:untouched4} after performing the first jump operation 
        in the phase \ref{item:etap2}.
        This figure was obtained from
        \cref{subfig:untouched4} by removal of the edge $c$
         with its black endpoint $x_1$ from the center of the cluster $c$ 
         and by combining black $x_1$ and $\Root_c$ with a new white vertex.
         \newline
        \protect\subref{subfig:untouched6} The structure of the cluster $c$ from 
        \protect\subref{subfig:untouched5} after completion of the phase~\ref{item:etap2}.
         This figure is obtained from
         \protect\subref{subfig:untouched5} by removal of the edge $c$
         with its black endpoint $x_k$ from the center of the cluster $c$ 
         and by combining black $x_k$ and $\Root_c$ with a new white vertex 
         for each $k\in\{2,\dots,i\}$.}
     \label{figure:untouched5-6}
\end{figure}

We start with the first bend operation from the phase \ref{item:etap1}, namely
$\Bend_{\Root_c,x_{d}}$. (Note that in the exceptional case when $d=1$ and
$x_{1}>\Root_c$ the phase \ref{item:etap1} is empty and there is nothing to
prove.) The assumptions \ref{AE1}--\ref{AE2} are clearly fulfilled for this
first operation; see \cref{subfig:untouched2}. Looking at
\cref{subfig:untouched3}, it is easy to verify that the bend operation
preserves the properties from \cref{invariant}. We do not modify any other
spine clusters, so also the assumption \ref{item:zal2} is preserved for this
first bend operation.

It is easy to check that the above arguments remain valid also for the remaining bend
operations from the phase \ref{item:etap1}.

\smallskip

We move on to the first jump operation $\Jump_{\Root_C,x_1}$ from the phase
\ref{item:etap2}. The assumptions \ref{AJ1}--\ref{AJ2} are clearly fulfilled
for this operation $\Jump_{\Root_C,y}$; see \cref{subfig:untouched4}. (In the
exceptional case if (a) $d=1$ and $x_1<\Root_C$, or (b) $y_1<x_1$ the phase 
\ref{item:etap2} is empty and there is nothing to prove.)
We do not modify any other
spine clusters, so the assumption \ref{item:zal2} is still fulfilled. 

It easy to check that the above arguments remain valid also for the remaining jump
operations from the phase \ref{item:etap2}.

\smallskip

The phase \ref{item:etap3} consists of the single bend operation
$\Bend_{\Root_C,y_1}$, see \cref{subfig:untouched6}. Going counterclockwise
around $y_1$ we denote by $E$ the label of the edge that is immediately after
the edge $C$; see \cref{subfig:untouched6}. We remind that each bend operation
preserves the properties from \cref{invariant}. From \cref{fig:untouched7}, it
is easy to see that if the cluster $E$ is non-spine or spine and touched, then
we do not disturb any other  untouched spine clusters, so the assumption
\ref{item:zal2} is still fulfilled. If $E$ is an untouched spine cluster, then
either (a) the vertex $y$ is both the anchor of the cluster $E$ as well as the
black endpoint of the root of the cluster $E$, or (b) $y$ is the black spine
vertex of the cluster $E$ with the larger label. In both cases, the untouched
spine cluster $E$ still fulfills \ref{item:zal2}.

\smallskip

Finally, we move on to the phase \ref{item:etap4};
the arguments that we used in the phase \ref{item:etap2} are also applicable here.

\begin{figure}    
        \centering
        \includegraphics[clip,trim=0cm 11cm 0cm 0cm,width=0.6\textwidth]{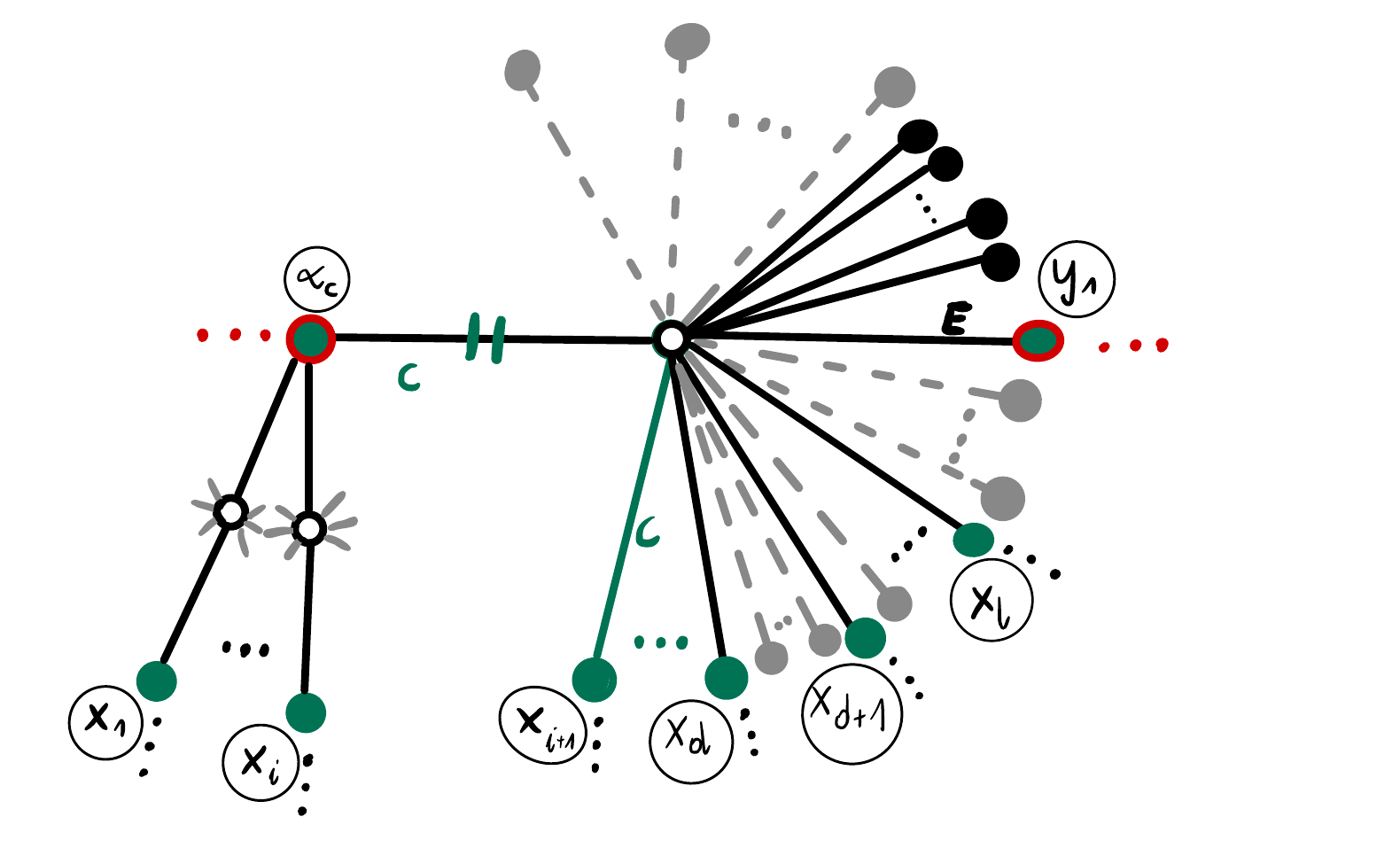}
    \caption{The structure of the cluster $c$ from 
    \cref{subfig:untouched6} after performing the phase \ref{item:etap3}.
    The \cref{fig:untouched7} is obtained from
    \cref{subfig:untouched6} by adding some additional vertices that do not belong
    to the cluster $c$; going counterclockwise these additional vertices occur after 
    $y_1$.}
    \label{fig:untouched7}  
\end{figure}

\smallskip

Note that during the execution of the internal loop we performed only
operations of the form $\Bend_{\Root_C,y}$ and $\Jump_{\Root_C,y}$ for $y\in
B_C$, so we touched exactly one cluster, namely~$C$. Therefore, in each step of
the internal loop the assumption \ref{item:zalNEW} is fulfilled. This completes
the proof of the inductive step.
\end{proof}

\label{page:inductive-ends}

\subsection{The rib treatment}
\label{sec:rib}

In the following we will use the word \emph{rib} as a synonym for \emph{non-spine}.

In the current section we continue the algorithm from \cref{sec:spine}.
For each successive cluster $C$ from the sequence $\Sigma$ we apply the following
procedure. In the language of programming we execute an external loop over the variable $C$.

We run the following internal loop over the variable $y\in B_C\setminus \{ \Root_c \}$ 
(with the ascending order):
if $y<\Root_c$ we apply $\Bend_{\Root_c,y}$; otherwise we apply  $\Jump_{\Root_c,y}$.

\medskip

\begin{example}
    
    We continue the example from \cref{subfig:ex8}. We recall that $\Sigma=(6,15,20,21,22)$.

    \begin{figure}
        \centering
        \subfloat[]{\label{subfig:ex9}
            \includegraphics[width=0.4\textwidth]{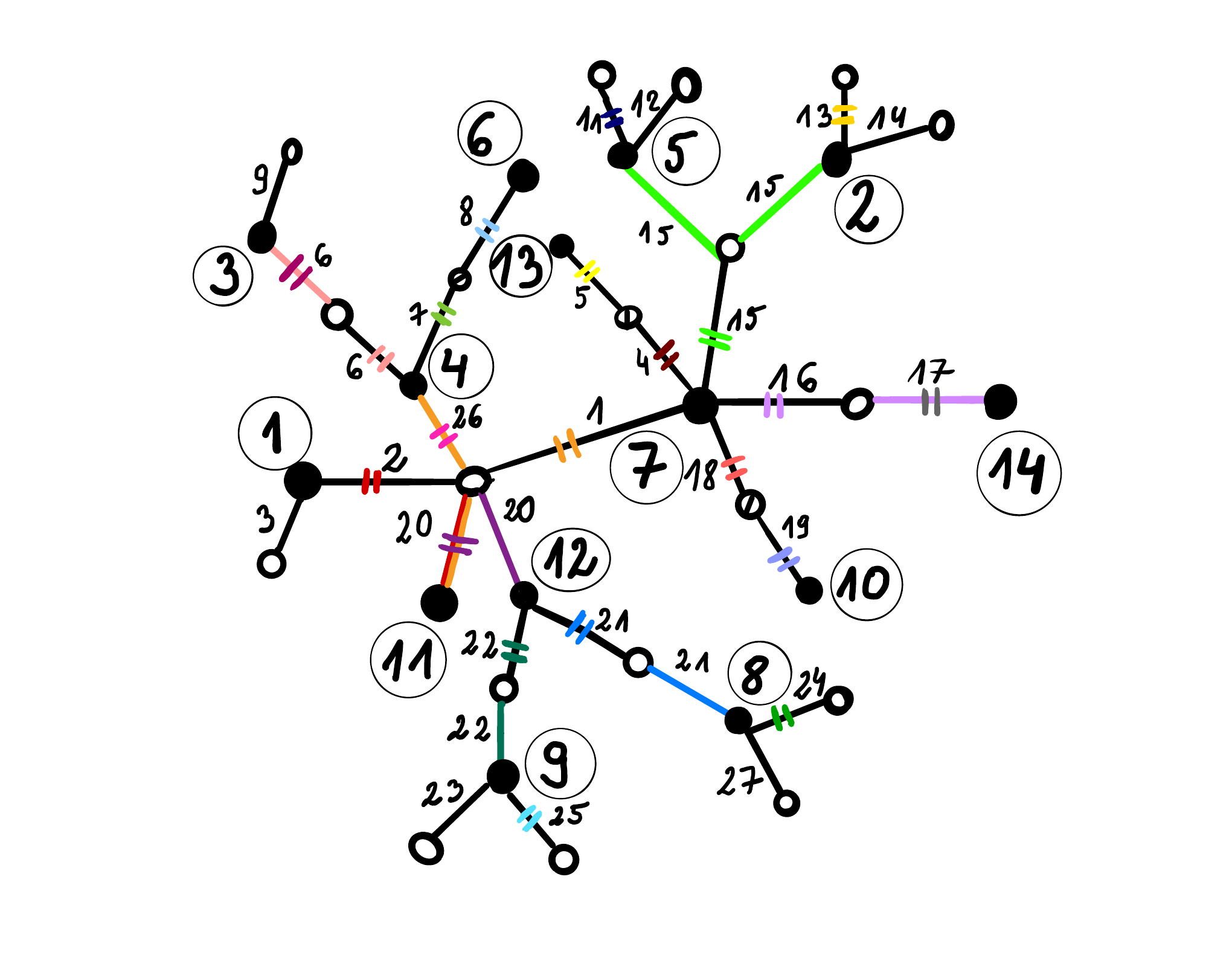}}
        \quad
        \subfloat[]{\label{subfig:ex10}
            \includegraphics[width=0.4\textwidth]{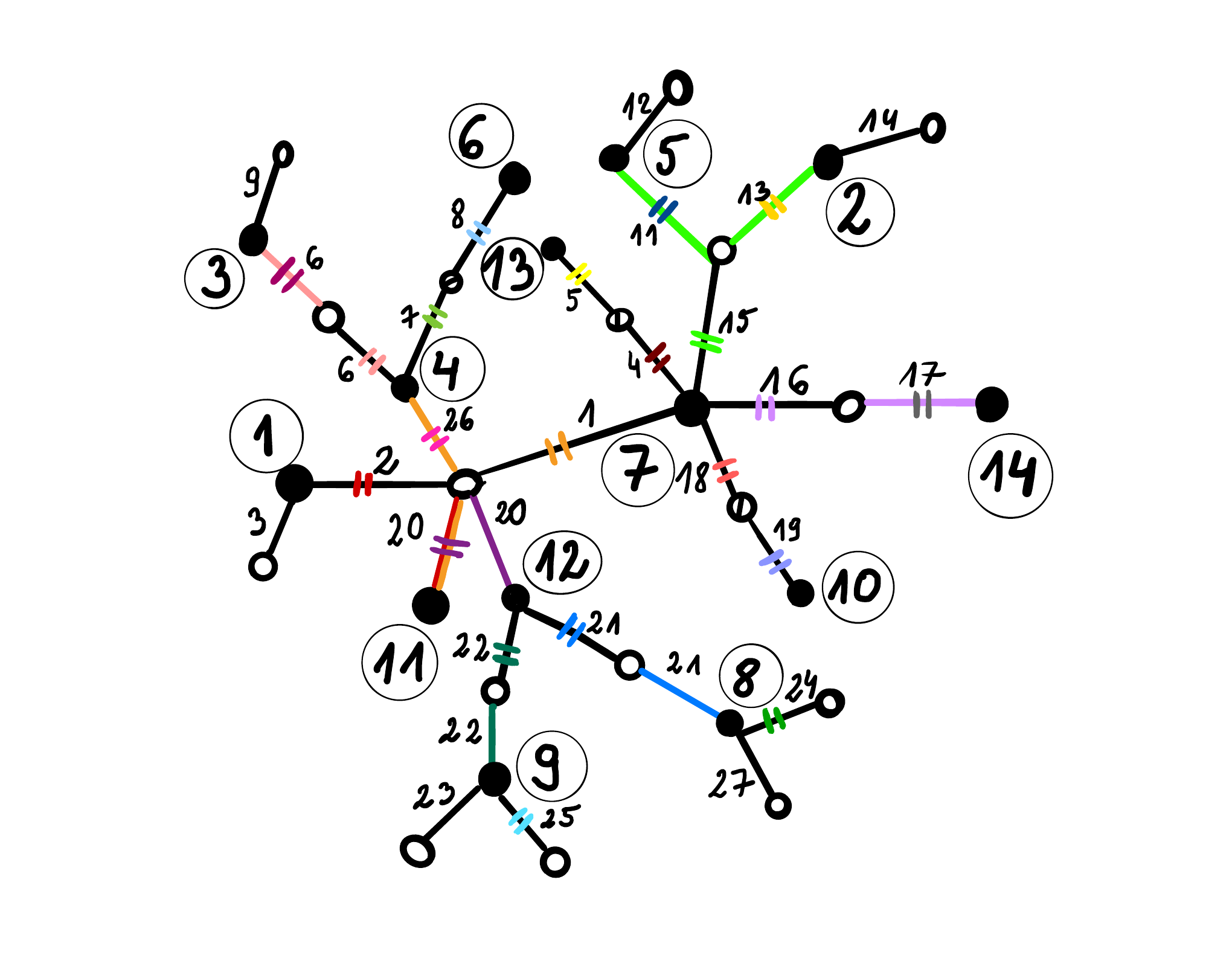}}
        \caption{
            \protect\subref{subfig:ex9} The outcome of applying two consecutive operations: 
            $\Bend_{4,3}$ and $\Jump_{4,6}$ to the tree from \protect\cref{subfig:ex8}.
            \newline
            \protect\subref{subfig:ex10} The outcome of applying two consecutive operations: 
            $\Bend_{7,2}$ and $\Bend_{7,5}$ to the tree from \protect\subref{subfig:ex9}.}
        \label{fig:ex910}
        \centering
        \subfloat[]{\label{subfig:ex11}
            \includegraphics[width=0.4\textwidth]{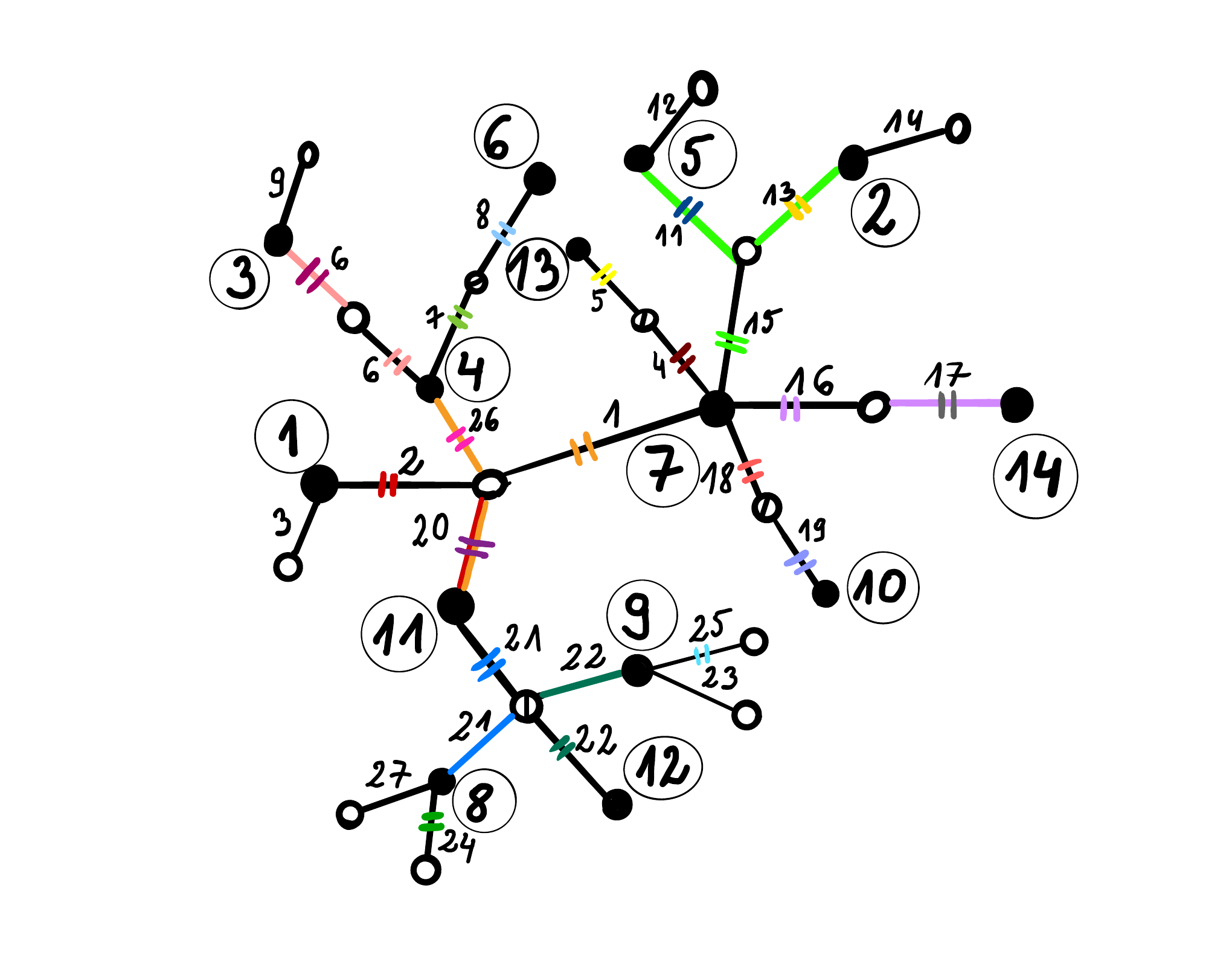}}
        \quad
        \subfloat[]{\label{subfig:ex12}
            \includegraphics[width=0.4\textwidth]{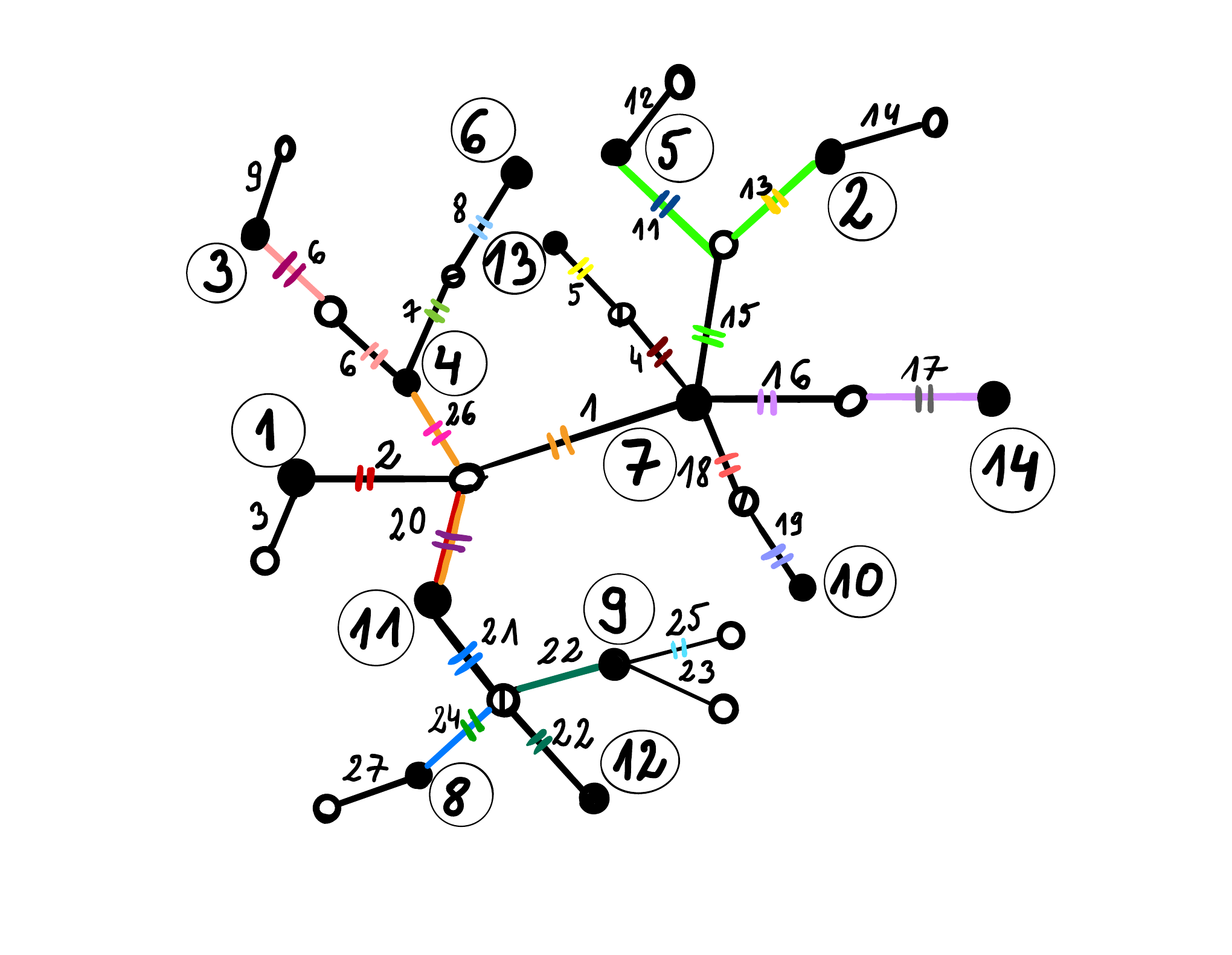}}
        \caption{
            \protect\subref{subfig:ex11} The output of $\Jump_{11,12}$ applied to the tree from \protect\cref{subfig:ex10}.
            \linebreak 
            \protect\subref{subfig:ex12} The output of $\Bend_{12,8}$ applied to the tree from \protect\subref{subfig:ex11}.}
        \label{fig:ex1112}
        \centering
        \subfloat[]{\label{subfig:ex13}
            \includegraphics[width=0.4\textwidth]{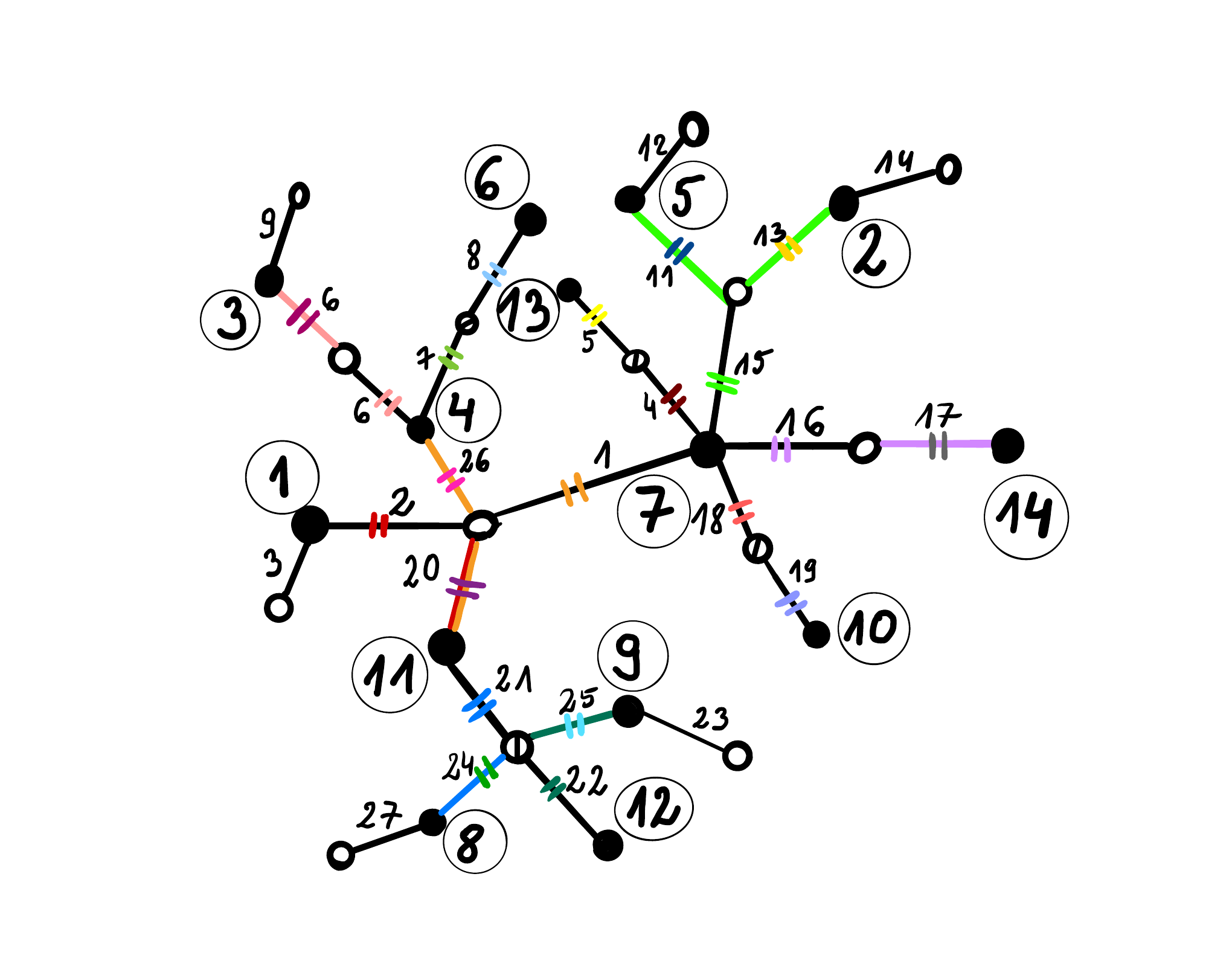}}
        \quad
        \subfloat[]{\label{subfig:ex14}
            \includegraphics[width=0.4\textwidth]{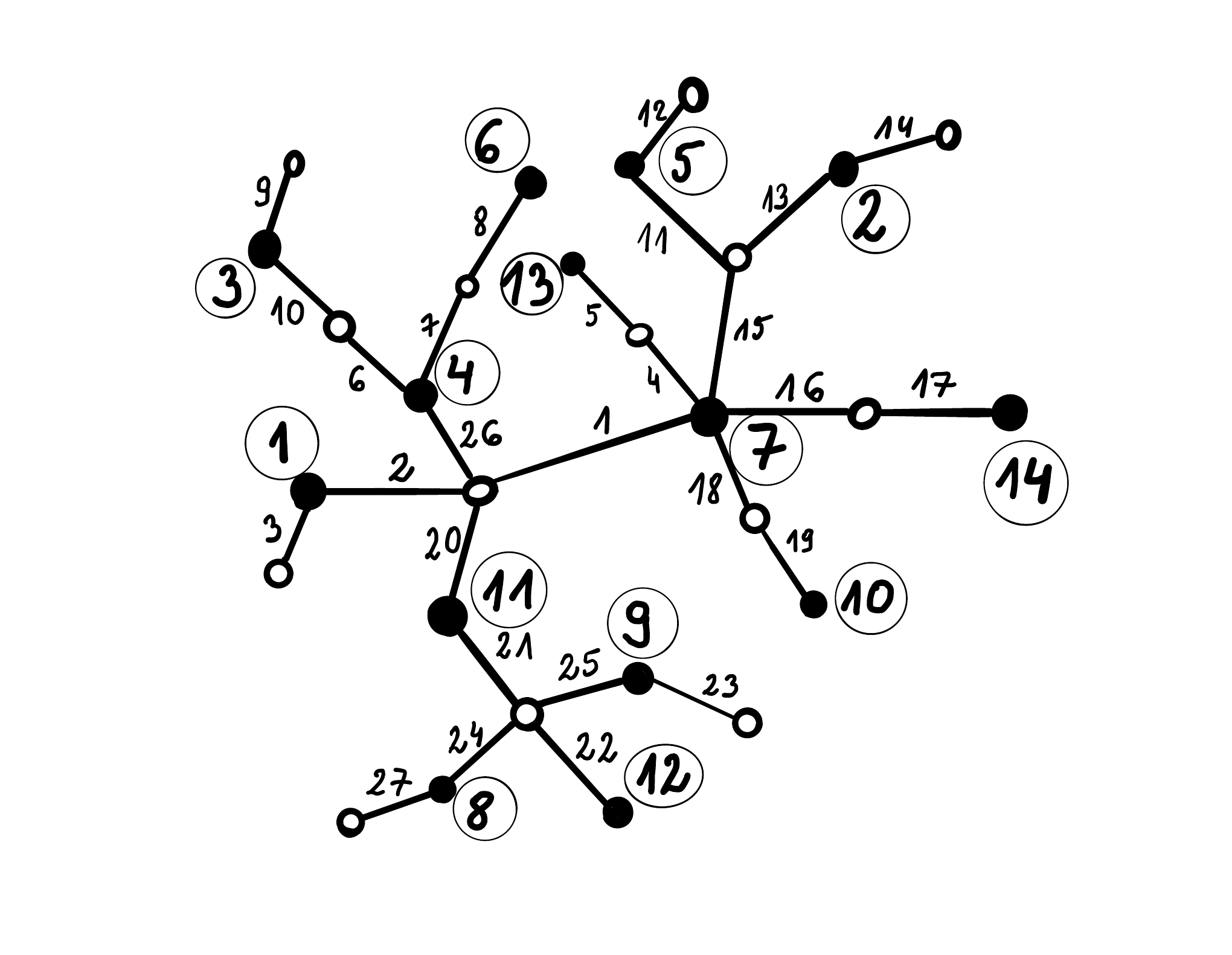}}
        \caption{
            \protect\subref{subfig:ex13} The output of $\Bend_{12,9}$ applied to the tree from \protect\cref{subfig:ex12}.
            \newline
             \protect\subref{subfig:ex14} The output of our algorithm $\CT$ applied to the minimal factorization \eqref{eq:example}. 
             It was created directly from the tree depicted on \protect\subref{subfig:ex1}.}
        \label{fig:ex1314}
    \end{figure}
    
    For $C=6$ we have $\Root_6=4$.
    Since $B_6\setminus \{4\}=\{3,6\}$ the loop runs over:
    $\bullet$~ $y=3$
    and  we apply $\Bend_{4,3}$; see \cref{subfig:ex9};
    $\bullet$ $y=6$ and we apply $\Jump_{4,6}$; see \cref{subfig:ex9}.
    
    For $C=15$ we have $\Root_{15}=7$.
    Since $B_{15}\setminus \{7\}=\{2,5\}$ the loop runs over:
    $\bullet$ $y=2$
    and  we apply $\Bend_{7,2}$; see \cref{subfig:ex10};
    $\bullet$ $y=5$ and we apply $\Bend_{7,5}$; see \cref{subfig:ex10}.

    For $C=20$ we have $\Root_{20}=11$. Since $B_{20}\setminus \{11\}=\{12\}$, the
    internal loop is applied once with $y=12$. As a result we apply $\Jump_{11,12}$; see
    \cref{subfig:ex11}.
    
    For $C=21$ we have $\Root_{21}=12$. Since $B_{21}\setminus \{12\}=\{8\}$, the
    internal loop is applied once with $y=8$. As a result we apply $\Bend_{12,8}$; see
    \cref{subfig:ex12}.

    For $C=22$ we have $\Root_{22}=12$. Since $B_{22}\setminus \{12\}=\{9\}$, the
    internal loop is applied once with $y=9$. As a result we apply $\Bend_{12,9}$; see
    \cref{subfig:ex13}.
    
    \cref{subfig:ex14} gives the output $T_2$ of our algorithm $\CT$ applied to
    the minimal factorization \eqref{eq:example}. The result is a Stanley tree of type $(1^{14})$.
    
\end{example}

\medskip

The resulting tree $T_2$ is the final output of our algorithm $\CT$. 
In \cref{sec:proof} we will show that it has the desired properties from \cref{thm:thm1}.

\subsection{Correctness of the rib treatment algorithm}
\label{sec:correctness-rib}

The main result of this section is \cref{prop:invariants2-ok}, which states
that the rib treatment presented in \cref{sec:rib} is well-defined in the sense
that the assumptions for the bend and jump operations
(\cref{sec:assumptions-e,sec:assumptions-j}) are indeed fulfilled during the
rib treatment.

\smallskip

For a black non-spine vertex $z$ by \emph{the branch defined by $z$} we mean
the edge outgoing from~$z$ in the direction of the spine together with all
edges and vertices that are its descendants. The branch defined by $z$ will
also be called \emph{the branch $z$}.

\pagebreak

\begin{lemma}
    \label{lm:lem2}
For each black non-spine vertex $z$ the branch $z$ does not change 
during the action of the algorithm $\CT$ 
until the vertex $z$ becomes touched.
\end{lemma}
\begin{proof}
As the first step we notice that the algorithm $\CT$ (i.e., the combined the
spine part and the rib part) can be regarded as a sequence of jump and bend
operations. We will use the induction over the number of bend/jump operations
that have been performed so far.

\medskip

At the beginning of the algorithm there is nothing to prove.

\medskip

Let us take an arbitrary tree transformation $\mathbb{T}$ (with
$\mathbb{T}=\Bend_{x,y}$ or $\mathbb{T}=\Jump_{x,y}$ for some black vertices
$x,y$) in our algorithm; we denote by $T$ the value of the tree before the
transformation $\mathbb{T}$ was applied. Our inductive hypothesis is that (i)
the branch $z$ was unchanged before $\mathbb{T}$ was applied, and (ii) the
vertex $z$ was untouched before $\mathbb{T}$ was applied. Let $E_1$ be the
cluster defined in \ref{AE1} in \cref{sec:assumptions-e}, respectively in
\ref{AJ1} in \cref{sec:assumptions-j}, i.e., the cluster that contains the
vertices $x$ and $y$.

In the case when $z=y$ then after performing the transformation $\mathbb{T}$
the vertex $z$ becomes touched and there is nothing to prove.

We will show that $x$ is touched. In the case when the black vertex $x$ in the
original tree $T_1$ was a spine vertex there is nothing to prove. Consider now
the case when $x$ in $T_1$ was a non-spine vertex; in this case the operation
$\mathbb{T}$ is a part of the rib treatment algorithm. Let $c$ be the white
vertex in the original tree $T_1$ that is the parent of the vertex $x$; it
follows that in some moment before the operation $\mathbb{T}$ was applied, the
variable in the external loop (either in the spine treatment algorithm or in
the rib treatment algorithm) took the value $C=c$ and one of the operations:
$\Jump_{\Root_c,x}$ or $\Bend_{\Root_c,x}$ was applied; since this moment the
vertex $x$ was touched, as claimed. Since $z$ is assumed to be not touched, it
follows that $z\neq x$.

\smallskip

\begin{figure}    
    \centering \includegraphics[clip,trim=5cm 0cm 0cm
    6.5cm, angle=-90,width=0.45\textwidth]{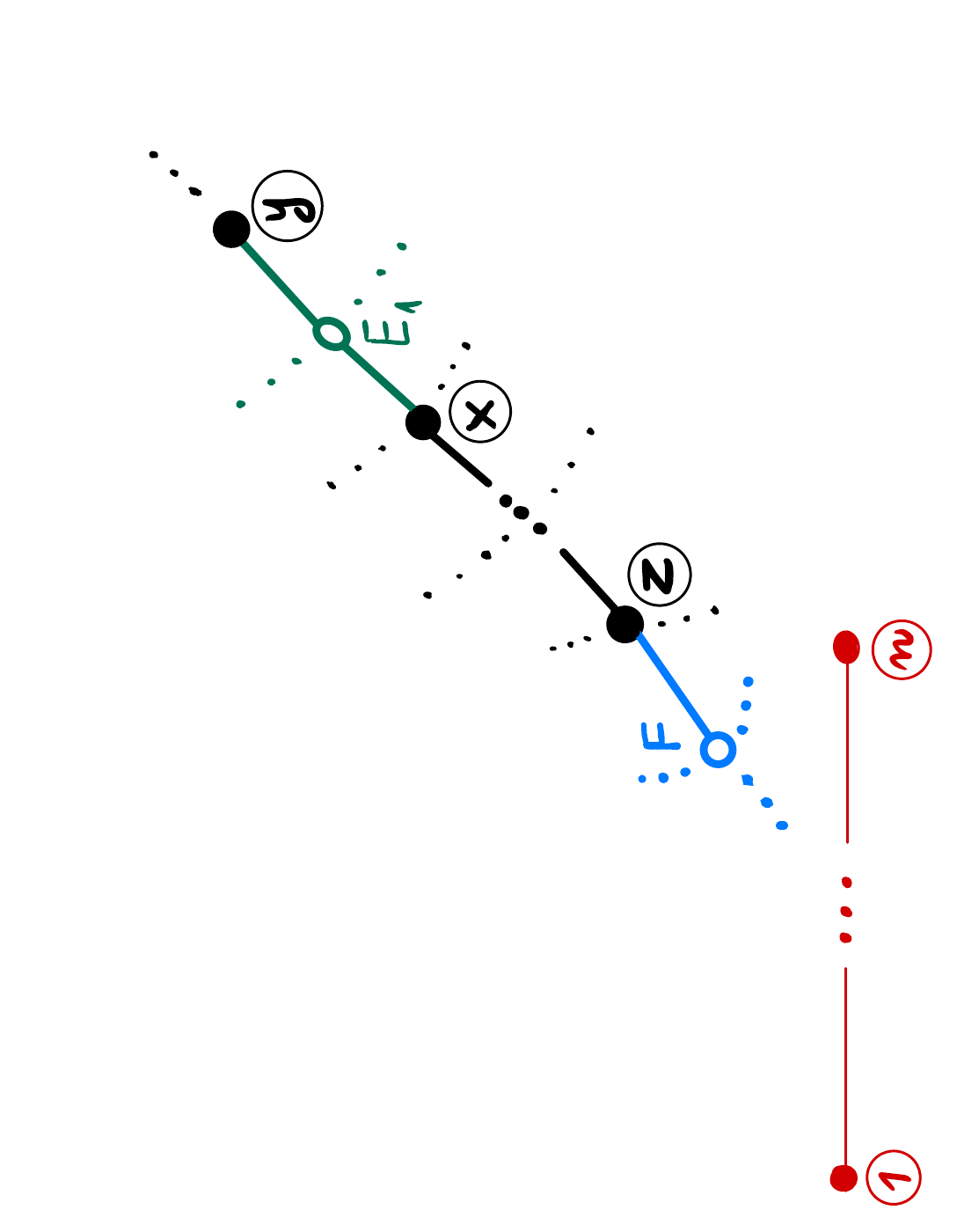} \caption{The hypothetical
        relative position of the cluster $E_1$ (green), the vertex $z$, the
        cluster $F$ (blue), and the spine (red) in the tree $T$ as long as the
        branch $z$ remains unchanged. We will prove that this configuration is \emph{not} possible.} 
        \label{fig:lemat2}
\end{figure}

The above discussion shows that we may assume that $z\notin\{x,y\}$. We denote
by $F$ the cluster in the tree $T_1$ that is defined by the edge outgoing from
the vertex $z$ in the direction of the spine; we recall that $E_1$ is the
cluster that contains the vertices $x$ and $y$. We will show that the cluster
$E_1$ is not contained in the branch $z$, i.e., the situation depicted on
\cref{fig:lemat2} is \emph{not} possible. By contradiction, suppose that the
cluster~$E_1$ is contained in the branch $z$. By the inductive assumption,
before $\mathbb{T}$ was applied, the branch~$z$ remained unchanged, so the
relative position of the spine, the vertex $z$ and the clusters $E_1$ and $F$
is the same in the initial tree $T_1$ and in the tree $T$; see
\cref{fig:lemat2}. It follows that $E_1$ is a non-spine cluster. Furthermore
either $F$ is a spine cluster, or $F$ is a non-spine cluster that is a
predecessor of $E_1$ in the sequence $\Sigma$. In particular, the
transformation $\mathbb{T}$ was performed during the rib treatment part of the
algorithm and the value of the variable~$C$ in the main loop at this moment was
equal to $C=E_1$. In one of the previous iterations of the main loop (either in
the spine treatment or in the rib treatment) the variable $C$ took the value
$C=F$; during this iteration of the loop the vertex $z$ became touched, which
contradicts the inductive hypothesis.

The above observation that $E_1$ is \emph{not} contained in the branch $z$
allows us to define \emph{the branch $z$} in an equivalent way by orienting all
edges of the tree towards the cluster $E_1$ and saying that the branch $z$
consists of the edge outgoing from $z$ in the direction of $E_1$ with all edges
and vertices that are its descendants. We will use this alternative definition
in the following.

\smallskip

In the case when $\mathbb{T}=\Bend_{x,y}$ (see \cref{subfig:exA}, where $z$ is
any black vertex different than $x$ and $y$; note that this black vertex may be
also in the part of the tree that was not shown), we can notice that the branch
$z$ still does not change after the application of $\mathbb{T}$; see
\cref{subfig:exB}, as required.

Consider the case when $\mathbb{T}=\Jump_{x,y}$ and (with the notations from
\cref{sec:assumptions-j}) $j=x$. See \cref{subfig:scjumpA}, where $z$ is any
black vertex different than $x$ and $y$. We can notice that the branch $z$
still does not change after the application of $\mathbb{T}$, as required; see
\cref{subfig:scjumpB}.

Consider now the remaining case when $\mathbb{T}=\Jump_{x,y}$ and $j\neq x$
(see \cref{fig:jump}). If $z\neq j$ then the branch $z$ still does not change
after the application of $\mathbb{T}$, as required. The case $z=j$ is is not
possible because \cref{invariant} \ref{invariant:4} applied to the cluster
$E_1$ implies that $z$ is touched which contradicts the inductive assumption.

This completes the proof of the inductive step.
\end{proof}

\pagebreak

\begin{propos}\label{prop:invariants2-ok}
For each operation \emph{bend} (respectively, \emph{jump}) performed during the
rib treatment algorithm the assumptions \ref{AE1}--\ref{AE3} (respectively, the
assumptions \ref{AJ1}--\ref{AJ4}) are fulfilled; in this way each
\emph{bend}/\emph{jump} operation is well-defined.
\end{propos}

\begin{figure}
    \centering \subfloat[]{\label{subfig:rib1} \includegraphics[clip,trim=5cm
    0cm 0cm 9cm,angle=-90,width=0.3\textwidth]{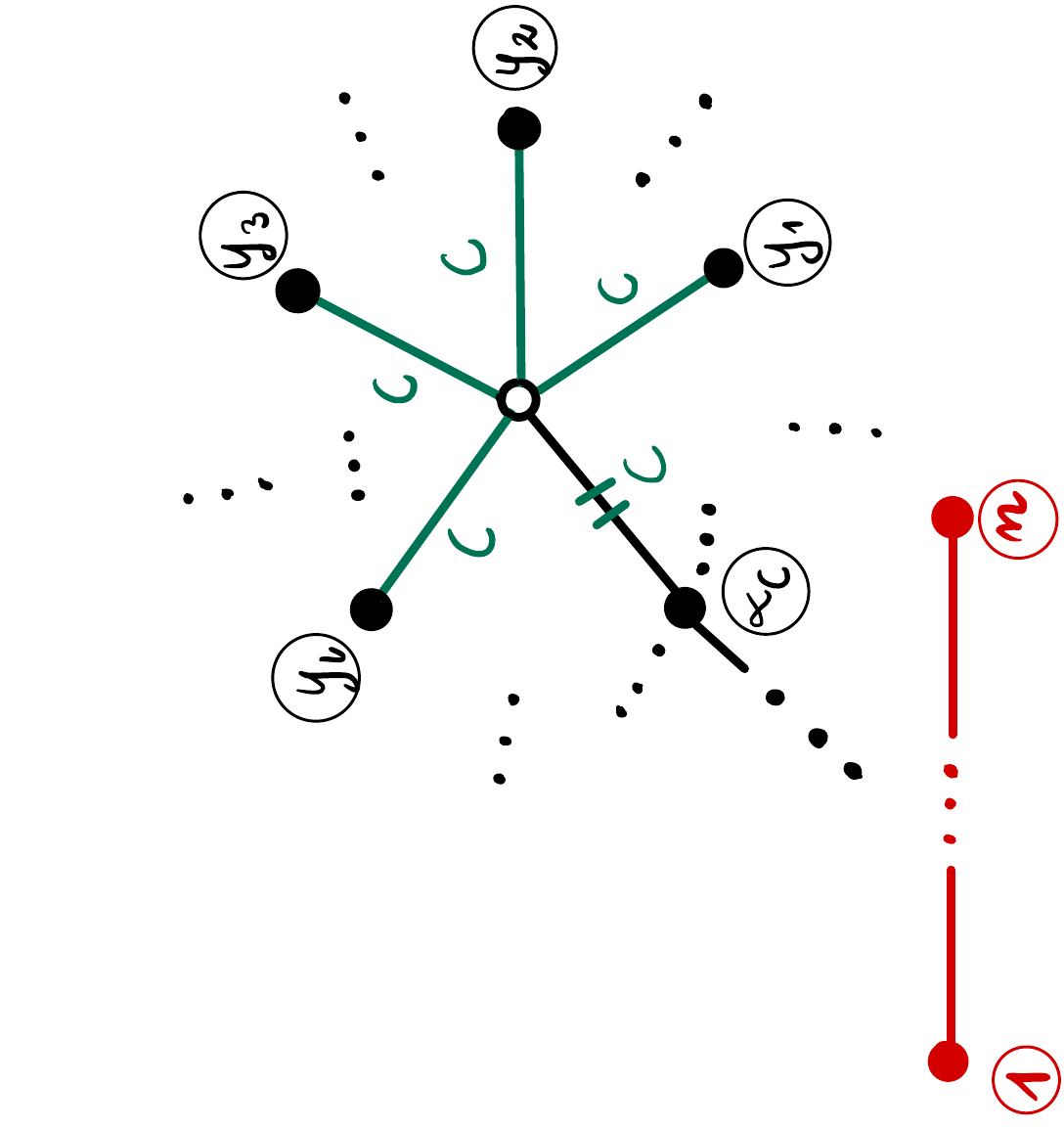}} \quad
\subfloat[]{\label{subfig:rib2} \includegraphics[clip,trim=5cm 0cm 0cm
    10cm,angle=-90,width=0.3\textwidth]{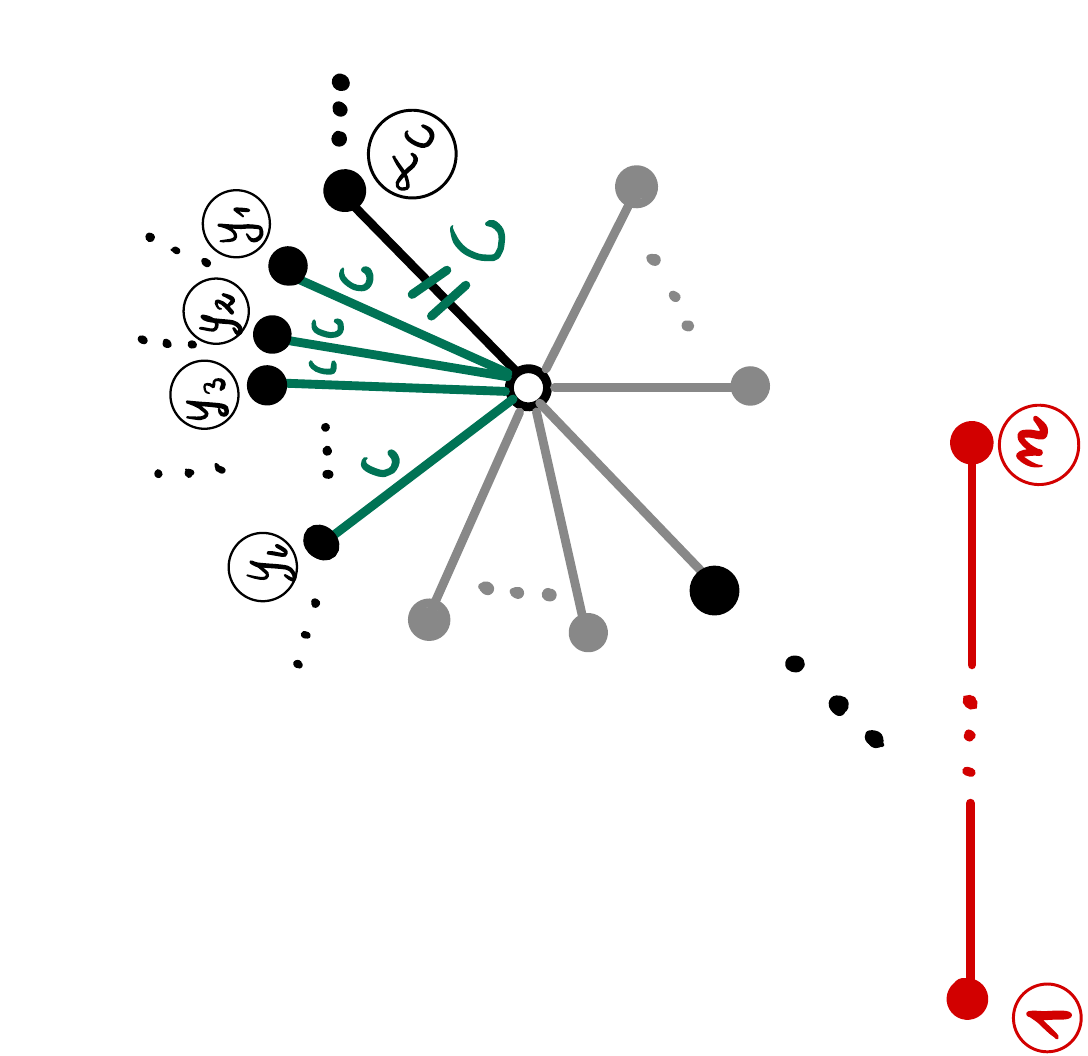}} \quad
\subfloat[]{\label{subfig:rib3} \includegraphics[clip,trim=5cm 0cm 0cm
    10cm,angle=-90,width=0.3\textwidth]{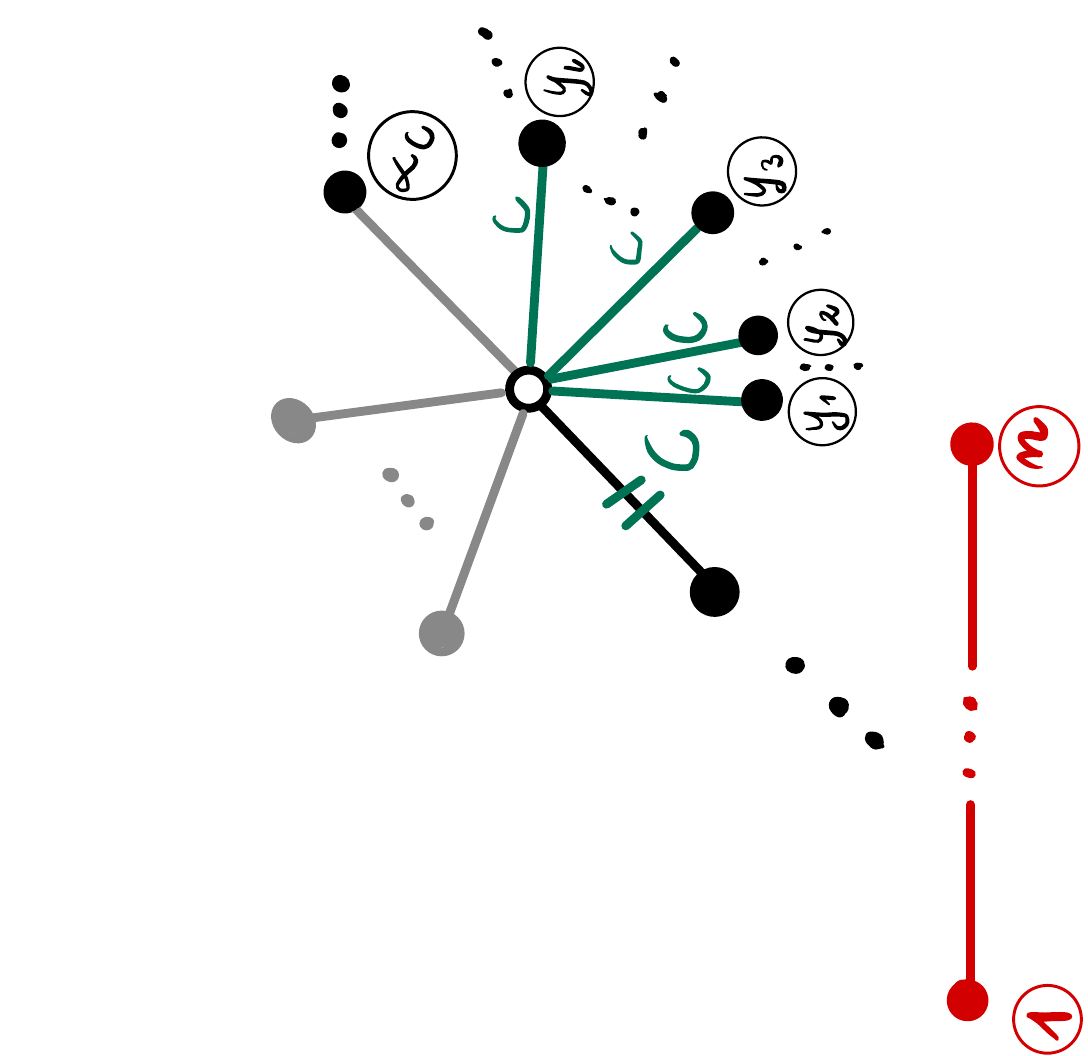}}
    \caption{ Three possible
    structures of the non-spine cluster $C$ at the beginning of the iteration
    of the external loop in the rib treatment algorithm. \newline
    \protect\subref{subfig:rib1} This structure coincides with the original
    structure of $C$ in the input tree~$T_1$. The edge outgoing from the center
    of the cluster $C$ in the direction of the spine is the root of the cluster
    $C$ and its black endpoint is the anchor~$\Root_C$ of the cluster. \newline
    \protect\subref{subfig:rib2}~This structure occurs as an outcome of either:
    (i) the bend operation in the parent cluster, provided that $C$ is the
    leftmost child, i.e.,~$C$ corresponds to the blue cluster $E_2$ 
    with the notations of \cref{fig:ex}, 
    or (ii) the jump operation in the parent cluster, provided that $C$ is the
    second leftmost child, i.e., $C$~corresponds to the red cluster $E_3$ with
    the notations from \cref{fig:jump,fig:scjump}. This structure was obtained
    from \protect\subref{subfig:rib1} by adding some additional vertices that
    do not belong to the cluster $C$; going clockwise these additional vertices
    occur after~$\Root_C$. The edge outgoing from the center of the cluster $C$
    in the direction of the spine does not belong to the cluster $C$. %
    \newline \protect\subref{subfig:rib3}~This structure occurs as an outcome
    of the jump operation in the parent cluster, provided that $C$ is the
    leftmost child, i.e., $C$ corresponds to $E_1$ with the notations from
    \cref{fig:jump,fig:scjump}. The root edge does not correspond to any of
    the edges that formed $C$ in the original tree $T_1$; in particular the
    black endpoint of the root does not belong to the set
    $B_C=\{\Root_C,y_1,\dots,y_l\}$. The root of the cluster $C$ is in the
    direction of the spine. More specifically, the counterclockwise cyclic
    order of the edges with their black endpoints around the center of cluster
    $C$ is as follows: the root of the cluster $C$, the edges that belong to
    cluster $C$ with consecutive black endpoints $y_1,\dots,y_l$, the edge
    that does not belong to the cluster $C$ with the black endpoint equal to
    the anchor $\Root_C$ and the remaining edges that do not belong to the
    cluster~$C$.} \label{fig:rib}
\end{figure}

\begin{proof}
We will go through some iteration of the external loop of the rib treatment for
some specific value of the variable $C$ (i.e., $C$ is a non-spine cluster) and we
will verify that all operations performed in this iteration are well-defined.

\medskip

\emph{Consider the case when the anchor $\Root_C$ is a non-spine vertex.} By
$C_1$ we denote the cluster that is the parent of the cluster $C$ in the tree
$T_1$. Therefore, the black vertex $\Root_C$ also belongs to the cluster $C_1$.
During some previous iteration of the main loop either during the spine
treatment or during the rib treatment (more specifically, this was the
iteration when the variable $C$ took the value $C_1$) we performed an operation
$\mathbb{T}$ with $\mathbb{T}=\Bend_{\Root_{C_1},\Root_{C}}$ or
$\mathbb{T}=\Jump_{\Root_{C_1},\Root_{C}}$. From the above \cref{lm:lem2} it
follows that until the operation $\mathbb{T}$ was performed, the branch defined
by the black vertex $\Root_C$ was unchanged, hence the cluster $C$ had the form
depicted on \cref{subfig:rib1}.

In the case when $\mathbb{T}=\Bend_{\Root_{C_1},\Root_C}$ is a bend operation,
after this operation $\mathbb{T}$ is applied the cluster $C$ 
either has the form depicted on \cref{subfig:rib2} (this happens if $C$ is leftist)
or it still has the form depicted on \cref{subfig:rib1} (otherwise).

In the case when $\mathbb{T}=\Jump_{\Root_{C_1},\Root_C}$ is a jump operation,
after this operation $\mathbb{T}$ is applied the cluster $C$ 
either has the form depicted on \cref{subfig:rib3} (this happens if $C$ is leftist),
or the form depicted on \cref{subfig:rib2} (this happens if $C$ 
is the second leftmost child)
or it still has the form depicted on \cref{subfig:rib1} (otherwise).

For each of the aforementioned three cases depicted on \cref{fig:rib} one can
go through the internal loop (in a manner similar to that from the proof of
\emph{the inductive step} on the pages
\pageref{page:inductive-starts}--\pageref{page:inductive-ends}, but simpler)
and to verify that the assumptions required by the bend/jump operations are
indeed fulfilled, as required.

\bigskip

\emph{Now, we assume that $\Root_C$ is a spine vertex.} In this case we have
two possible situations. The black vertex $\Root_C$ belongs to either one or
two spine clusters in the tree $T_1$.

\smallskip

Consider the case when the anchor $\Root_C$ belongs to two spine clusters in
the tree $T_1$, denoted by $C_1,C_2$. During some previous iteration of the
main loop of the spine treatment part (more specifically, this was the
iteration when the variable $C$ took the value $C_i$ for $i=1,2$) we performed
an operation $\mathbb{T}$ with
\begin{enumerate}[label=(\roman*)] 
\item 
\label{case:A}
$\mathbb{T}=\Bend_{\Root_{C_i},\Root_{C}}$ if $\Root_{C_i}<\Root_{C}$,
\item 
\label{case:B}
$\mathbb{T}=\Bend_{\Root_{C_i},\cdot}$ or $\mathbb{T}=\Jump_{\Root_{C_i},\cdot}$
if $\Root_{C_i}=\Root_{C}$.
\end{enumerate}

In the case \ref{case:A},  after this operation
$\mathbb{T}=\Bend_{\Root_{C_i},\Root_C}$ is applied, the cluster $C$ either has
the form depicted on \cref{subfig:rib2} (this happens if $C$ is leftist) or it
still has the form depicted on \cref{subfig:rib1} (otherwise). By checking
these two cases separately (the reasoning is fully analogous to the one
presented above) we see that during this external loop iteration for this
specific value of $C$ all the assumptions for the bend/jump operations are
fulfilled, as required.

In the case \ref{case:B},
after the operation $\mathbb{T}$ is applied the cluster $C$ 
still has the form depicted on \cref{subfig:rib1}.
Again, one can easily check that that during this loop iteration all the
assumptions for the bend/jump operations are fulfilled, as required.

\smallskip

In the case when the black vertex $\Root_C$ belongs to only one spine cluster in
the tree $T_1$ the reasoning is analogous to the one above and we skip the
details.
\end{proof}

As an extra bonus, the above proof shows that any non-spine cluster $C$
at a later stage of the algorithm (as long as it remains untouched) either has the form 
\subref{subfig:rib1}, \subref{subfig:rib2}, or \subref{subfig:rib3} on \cref{fig:rib}.

\section{The output of the bijection is a Stanley tree}
\label{sec:proof}

The following results are the first step towards the proof of \cref{thm:thm1}.

\begin{lem}
    \label{lem:what-is-j}     
Let $\Jump_{x,y}$ be one of the jump operations performed during the execution
of the algorithm $\CT$, and let $j$ be the corresponding black vertex that was
defined in the assumption \ref{AJ2} from \cref{sec:assumptions-j}; see
\cref{fig:jump,fig:scjump}. Then the label of the vertex $j$ is smaller than
the label of the vertex~$y$.
\end{lem}
\begin{proof}
Our strategy is to explicitly pinpoint the vertex $j=j(x,y)$ 
in the original tree $T_1$.

\bigskip

\emph{We start with the case when the jump operation $\Jump_{x,y}$ was
    performed during the spine treatment} (this case holds true if and only if the
white vertex between $x$ and $y$ is a spine vertex). Let $C$ be the value of
the external loop variable at the time when $\Jump_{x,y}$ was performed. We
have shown in  \cref{lem:invariants-ok} that the anchor of the
cluster $C$ is the endpoint of the root of the cluster $C$ in each step of the
internal loop. Therefore $j=x$; see the case described on \cref{fig:scjump}.

\bigskip

\emph{Consider now the case when the jump operation $\Jump_{x,y}$ was performed
    during the rib treatment;} again let $C$ be the value of the external loop
variable at the time when $\Jump_{x,y}$ was performed; in particular $C$ is a
non-spine cluster.

If the cluster $C$ in the tree $T_1$ is non-leftist, from
\cref{lem:pomoc1,lem:pomoc2} we infer that no bend/jump operation separated the
anchor of the cluster $C$ with the root of this cluster. It follows therefore
that $j=x$; see the case described on \cref{fig:scjump}.

\smallskip

Consider now the case when the cluster $C$ in the tree $T_1$ is leftist. From
\cref{lem:pomoc1,lem:pomoc2} it follows that in order to check whether the
anchor of the cluster $C$ is separated from the root, we need to consider one
of the previous iterations of the main loop (in the spine treatment or the rib
treatment), namely the one for the cluster $C'$ that is the parent of~$C$. In
the case when $x<\Root_{C'}$, this previous iteration contained the bend
operation $\Bend_{\Root_{C'},x}$, which does not separate the anchor from the
root, so again $j=x$.

\medskip

The only challenging case is the one when $C$ is a leftist cluster and
$x>\Root_{C'}$ so that this previous iteration contains the jump operation
$\Jump_{\Root_{C'},x}$, which separated the anchor and the root in the cluster
$C$. Let us have a closer look on this previous jump operation
$\Jump_{\Root_{C'},x}$; it is depicted on \cref{fig:jump,fig:scjump} with the
blue cluster $E_2=C$ and the black cluster $E_1=C'$. Our desired value of
$j(x,y)$ is the black endpoint of the root of the cluster $C$; on
\cref{subfig:jumpB,subfig:scjumpB} this endpoint carries the label $j$. On the
other hand, the value of the variable $j=j(\Root_{C'},x)$ for the previous jump
operation $\Jump_{\Root_{C'},x}$ in the parent cluster is the black endpoint of
the root of the cluster $C'=E_1$; on \cref{subfig:jumpA,subfig:scjumpA} this
vertex carries the same label $j$. In this way we proved that our desired value
of
\[ j = j(x,y) = j(\alpha_{C'},x) \]
coincides with the value of $j$ for the jump operation $\Jump_{\Root_{C'},x}$
in the parent cluster.
It follows that the value of $j$
can be found by the following recursive algorithm.

\begin{figure}    
    \centering
    {\includegraphics[clip,trim=0cm 1.5cm 5cm 0cm,width=0.3\textwidth]{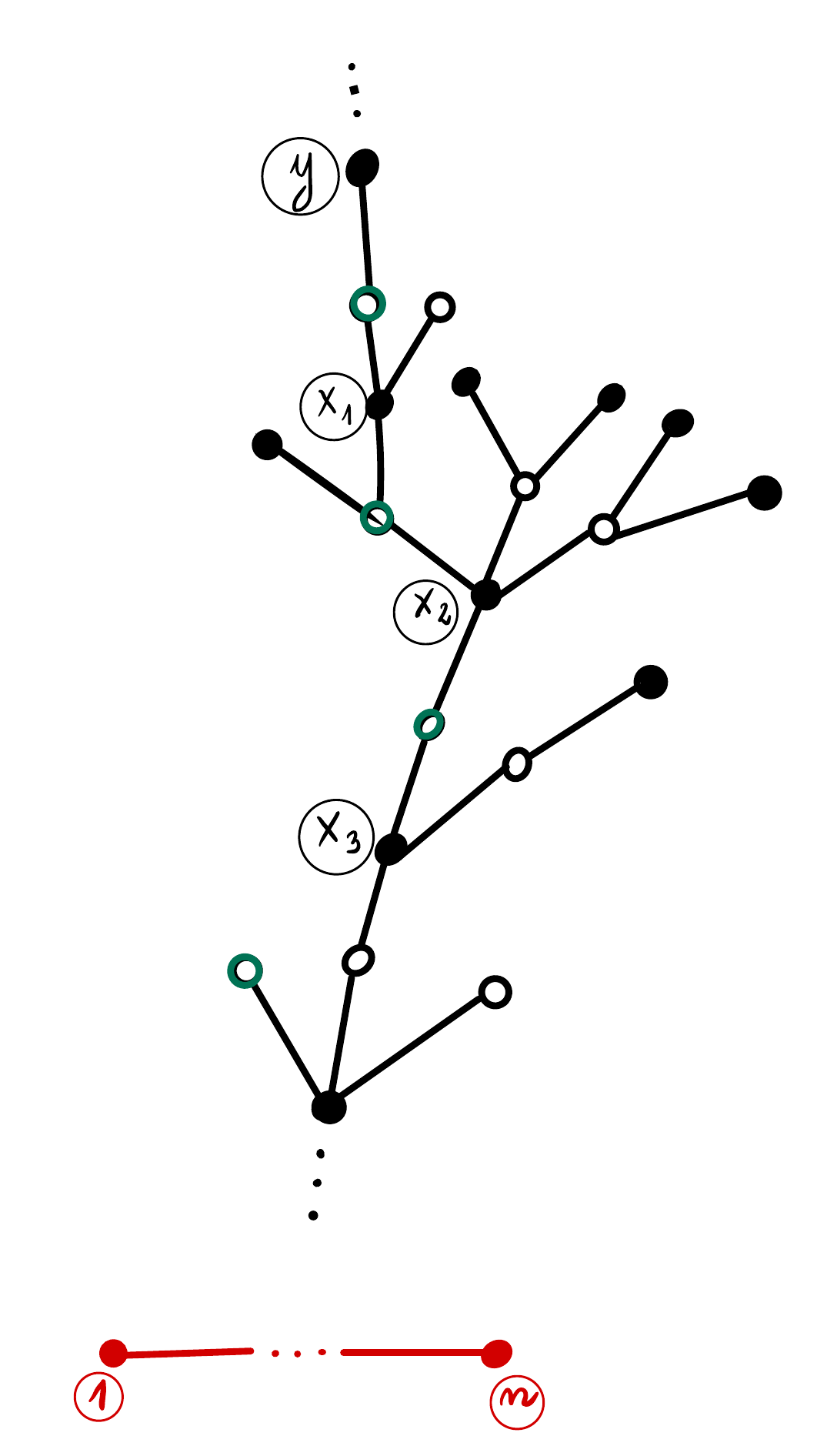}}
    \caption{The tree $T_1$ with black vertices $y,x_1,x_2,x_3$ 
    where a white leftist vertices are marked green.}
    \label{fig:lemat10}  
\end{figure}

\medskip

We traverse the tree $T_1$, starting from the black vertex $y$, %
always going towards the spine, as long as the following two conditions are
satisfied:
\begin{enumerate}[label=(C\arabic*)]
\item 
\label{entry:leftist}
we are allowed to enter a white vertex only if it is a leftist vertex;

\item 
\label{entry:smaller}
we are allowed to enter a black vertex only if its label is smaller than 
the label of the last visited black vertex so far.
\end{enumerate}
Additionally, if we just entered a black spine vertex, the algorithm
terminates. If we entered a white spine vertex $w$ it is not clear what it
means \emph{to move towards the spine}; we declare that we should move now
to the the root $\alpha_w$ of the cluster defined by $w$. We denote by
$z=(z_1,\dots,z_l)$ the labels of the visited black vertices. For the example
on \cref{fig:lemat10} if $x_3<x_2<x_1<y$ we have $z=(y,x_1,x_2,x_3)$. The label
$z_l$ of the last visited black vertex is the label of our searched black
vertex $j$. It is obvious now that $j< y$.
\end{proof}

\begin{propos}
    \label{lm:lem1}
    The output of the algorithm $\CT$ from \cref{sec:algorithm} 
    belongs to $\mathcal{T}_{b_1,\dots ,b_n}$.
\end{propos}
\begin{proof}

Recall that the integers $a_1,\dots,a_n$ are related to $b_1,\dots,b_n$ via
\eqref{eq:a-and-b}.

By construction, the initial tree $T_1$ has $n$ black vertices labeled with the
numbers $1,\dots,n$, and $k$ white vertices, and $\sum_i a_i=n+k-1$ edges
labeled with the numbers $1,\dots,k$ in such a way that each edge label is used
at least once, cf.~\eqref{eq:sum-of-a}. By counting the number of edges we
observe that
\[ \sum_{c=1}^k |B_c|=n+k-1.\]
Furthermore, if $c\in \{1,\dots,k\}\setminus (\mathfrak{C} \cup \Sigma)$, so
that the corresponding cluster is a leaf, then $|B_c|=1$. It follows that the
total number of \emph{jump}/\emph{bend} operations during the execution of our
algorithm $\CT$ is equal to
\begin{multline}\label{eq:operations}
\sum_{c\in \mathfrak{C}} (|B_c|-1)+\sum_{c\in \Sigma} (|B_c|-1)= \\
=\sum_{c\in \mathfrak{C}} (|B_c|-1)+\sum_{c\in \Sigma} (|B_c|-1)+
 \sum_{c\in \{1,\dots,k\}\setminus \{\mathfrak{C} \cup \Sigma\}} (|B_c|-1)=\\
=\left( \sum_{c=1}^k |B_c|\right) -k=n-1.
\end{multline}
After performing each bend/jump operation, the number of white vertices as well
as the number of edges decreases by $1$. Furthermore, the edge that disappears
has a repeated label, in this way the set of the edge labels remains unchanged
in each step.

It follows that the output $T_2$ is a bicolored plane tree with
$(n+k-1)-(n-1)=k$ edges labeled by numbers $1,\dots,k$ (so each label is used
exactly once) and with $k-(n-1)=\sum_{i=1}^{n}b_i$ white vertices,
cf.~\eqref{eq:sum-of-b}. To complete the proof we have to show that the tree
$T_2$ is a Stanley tree of type $(b_1,\dots,b_n)$.

\medskip

We will consider two types of edges: \emph{solid} and \emph{dashed}. At the
beginning in the plane tree $T_1$ each edge is declared to be \emph{solid}.
After performing the operation $\Bend_{x,y}$ the two edges that form the path
between $x$ and $y$ become \emph{dashed}; see \cref{subfig:exB}. After
performing the operation $\Jump_{x,y}$ the two edges that form the path between
$j$ and $y$ become \emph{dashed}; see \cref{subfig:jumpB,subfig:scjumpB}.

\label{page:belongs}

We say that a white white vertex $v$ \emph{is attracted} to a black vertex $i$
if one of the following two conditions holds true:
\begin{itemize}
\item the vertices $i$ and  $v$ are connected by a \emph{solid} edge,
\item the vertices $i$ and  $v$ are connected by a \emph{dashed} edge and 
\[ \max\big\{y:\emph{$y$ and  $v$ are connected by a dashed edge}\big\}=i. \]
\end{itemize} 
If $v$ is attracted to $i$, we will also say that \emph{the edge $e$ between
    $v$ and $i$ is attracted to $i$} or that \emph{$e$ is an attraction edge}.

For each $i\in\{1,\dots,n\}$ we define the variable $\Numberb_i$; in each step
of the algorithm this variable counts the number of white white vertices that
are attracted to the black vertex~$i$.

\medskip

By considering any bend/jump operation that is performed during the algorithm
$\CT$ it is easy to see that each of the variables
$\Numberb_1,\dots,\Numberb_n$ weakly decreases over time. The only difficulty
in the proof is to verify that during the jump operation the newly created
vertex~$w$ (with the notations from \cref{fig:jump,fig:scjump}) is not
attracted to the vertex $j$; this fact is a consequence of
\cref{lem:what-is-j}. Below we will show that for $i\in\{2,\dots,n-1\}$ the
variable~$\Numberb_i$ during the algorithm $\CT$ decreases at least by $2$,
while each of the variables $\Numberb_1$ and $\Numberb_n$ decreases at least by
$1$.

\begin{figure}
    \centering
    \subfloat[]{\label{subfig:exA1}
                {\includegraphics[clip,trim=1.5cm 0cm 1.9cm 3.5cm,angle=-90,width=0.4\textwidth]{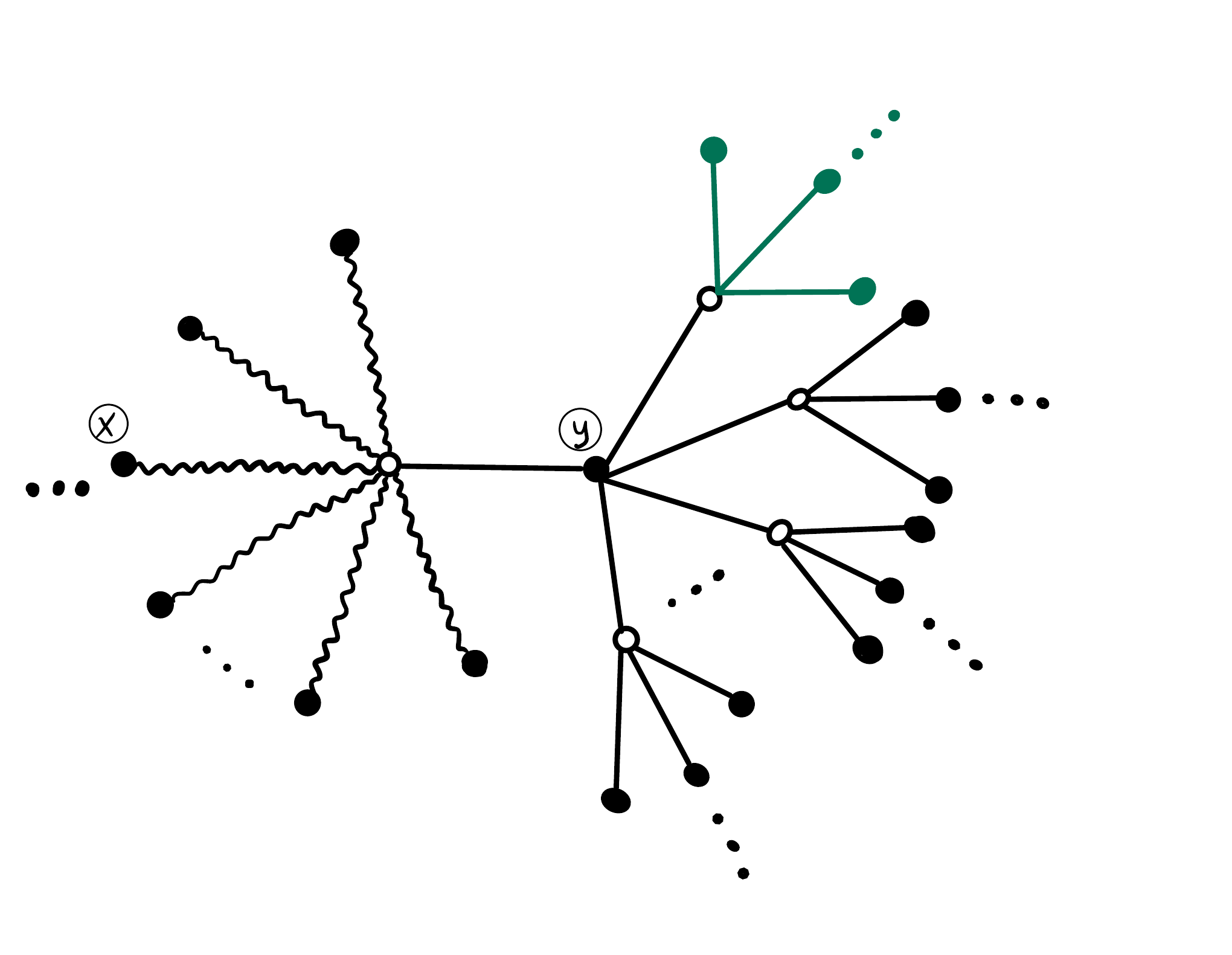}}  
    }
    \quad
    \subfloat[]{\label{subfig:exA2}
                {\includegraphics[clip,trim=3cm 0cm 0cm 3.5cm,angle=-90, width=0.4\textwidth]{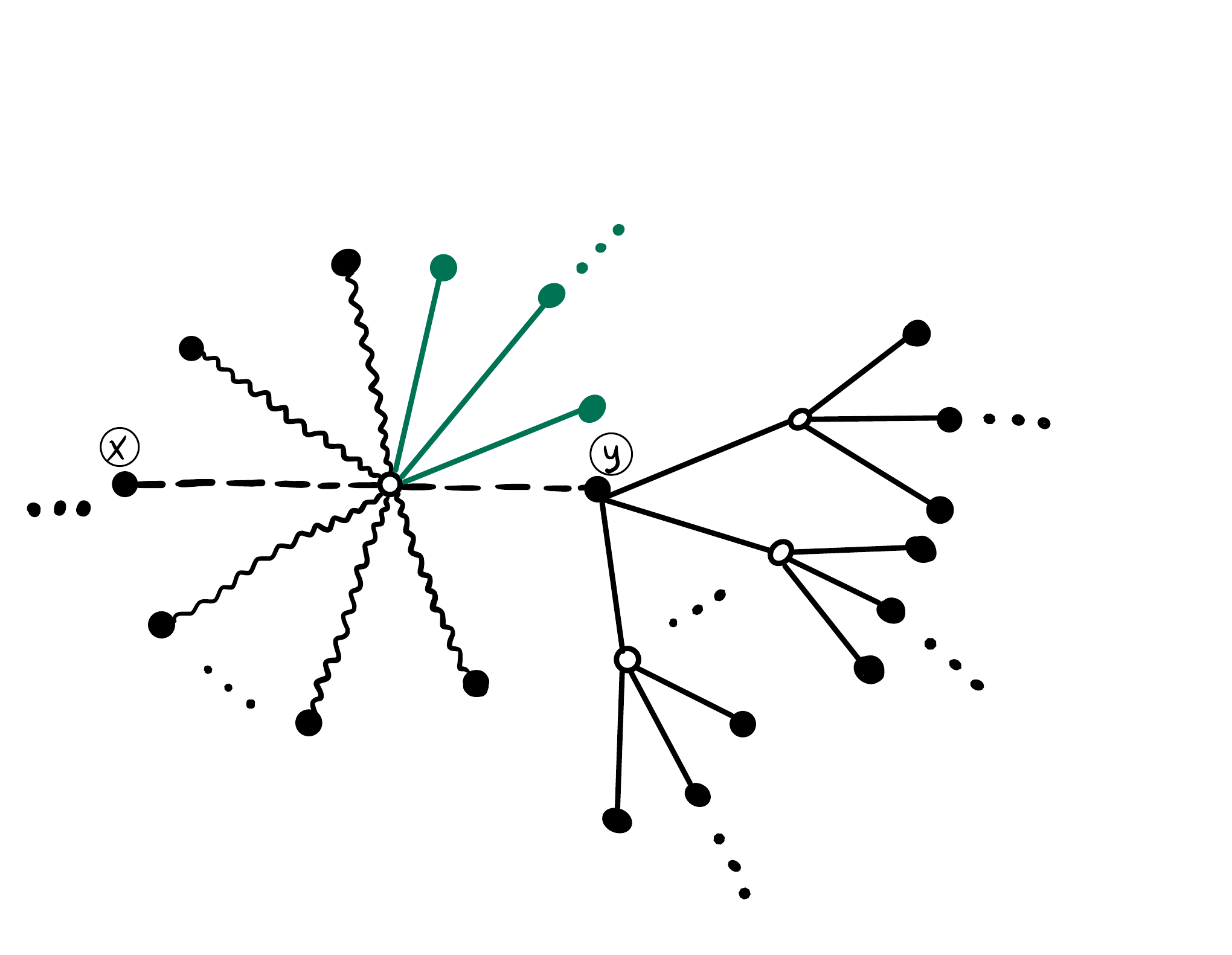}}
    }
    \caption{\protect\subref{subfig:exA1} The initial configuration of the tree. 
        Wavy edge means that this edge can be \emph{solid} or \emph{dashed}.
        \newline
        \protect\subref{subfig:exA2} The structure of the tree \protect\subref{subfig:exA1}
    after performing the operation $\Bend_{x,y}$ with $x>y$.}
    \label{fig:exA}

    \centering
    \subfloat[]{\label{subfig:jumpA1}
        {\includegraphics[angle=-90,width=0.4\textwidth]{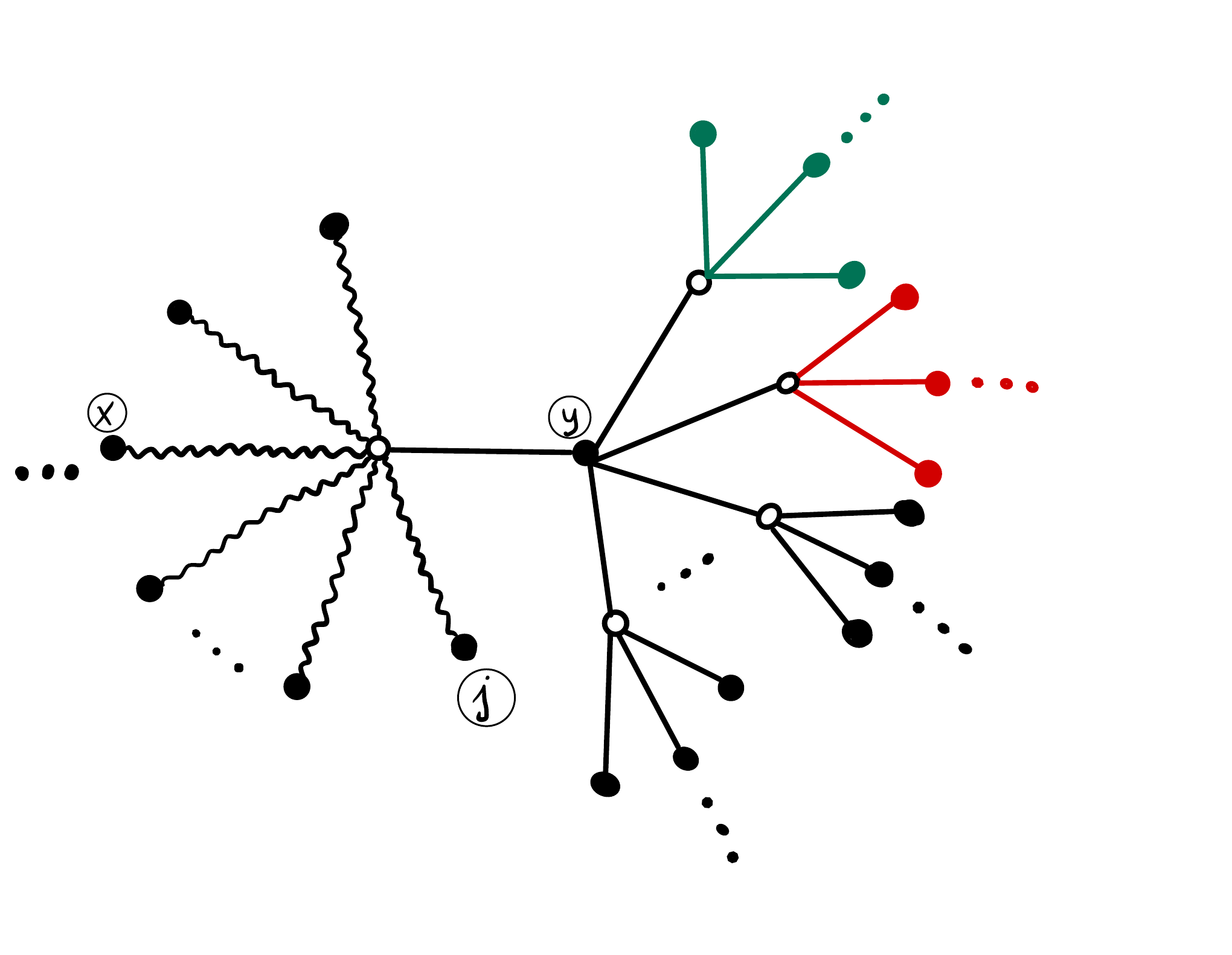}}  }
    \quad
    \subfloat[]{\label{subfig:jumpA2}
        {\includegraphics[angle=-90,width=0.4\textwidth]{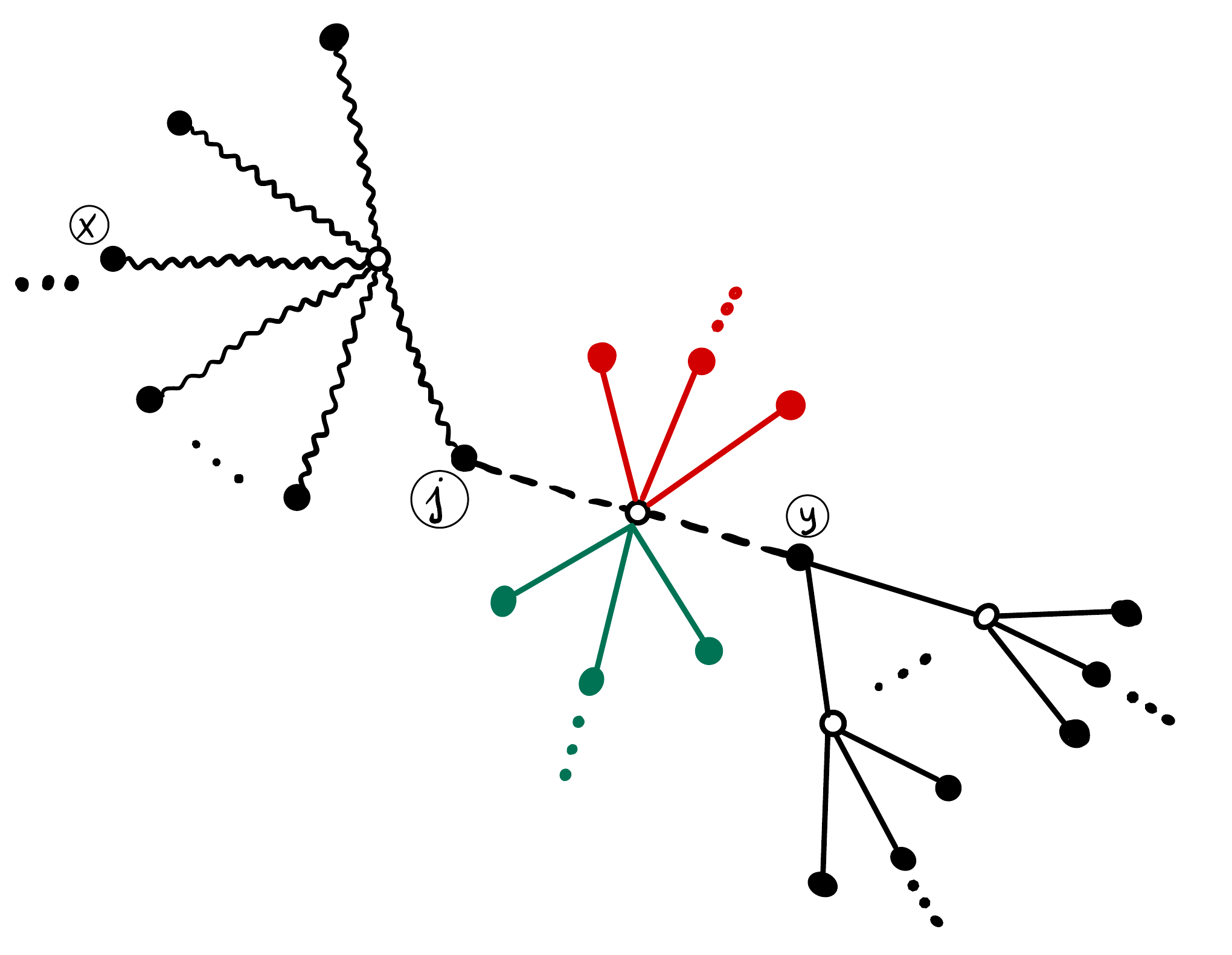}}}
    \caption{\protect\subref{subfig:jumpA1} The initial configuration of the tree.
         Wavy edge means that this edge can be \emph{solid} or \emph{dashed}. 
         We recall that $j< y$ by \cref{lem:what-is-j}.
         \newline
    \protect\subref{subfig:jumpA2} The structure of the tree \protect\subref{subfig:jumpA1}
    after performing the operation  $\Jump_{x,y}$ with $x<y$.}
    \label{fig:jumpA}
\end{figure}

\smallskip

\emph{Case 1: $i$ is a non-spine vertex.} For convenience we will denote the
non-spine black vertex~$i$ by the symbol $y$. There exists exactly one cluster $c\in
\{1,\dots,k\}$ in the direction of the spine in the plane tree $T_1$ such that
$y\in B_c$. In the iteration of the main loop (in the spine treatment or in the
rib treatment) corresponding to $C=c$ one of the following operations is applied:
either $\Bend_{\Root_c,y}$ (if $y<\Root_c$; see \cref{fig:exA}) or
$\Jump_{\Root_c,y}$ (otherwise; see \cref{fig:jumpA}). By \cref{lm:lem2} it
follows that at the time one of these operations is applied, the edges in the
branch $y$ are solid; on \cref{subfig:exA1,subfig:jumpA1} this branch is located
on the right-hand side and drawn with non-wavy lines. By looking on
\cref{fig:exA,fig:jumpA} we see that for either of these two operations 
$\Numberb_y^{\text{after}}=\Numberb_y^{\text{before}}-2$, as required.

\smallskip

\begin{figure}
    \centering
    \subfloat[]{\label{subfig:proofspineA}
                {\includegraphics[clip,trim=5.5cm 0cm 0cm 3.5cm,angle=-90,width=0.4\textwidth]{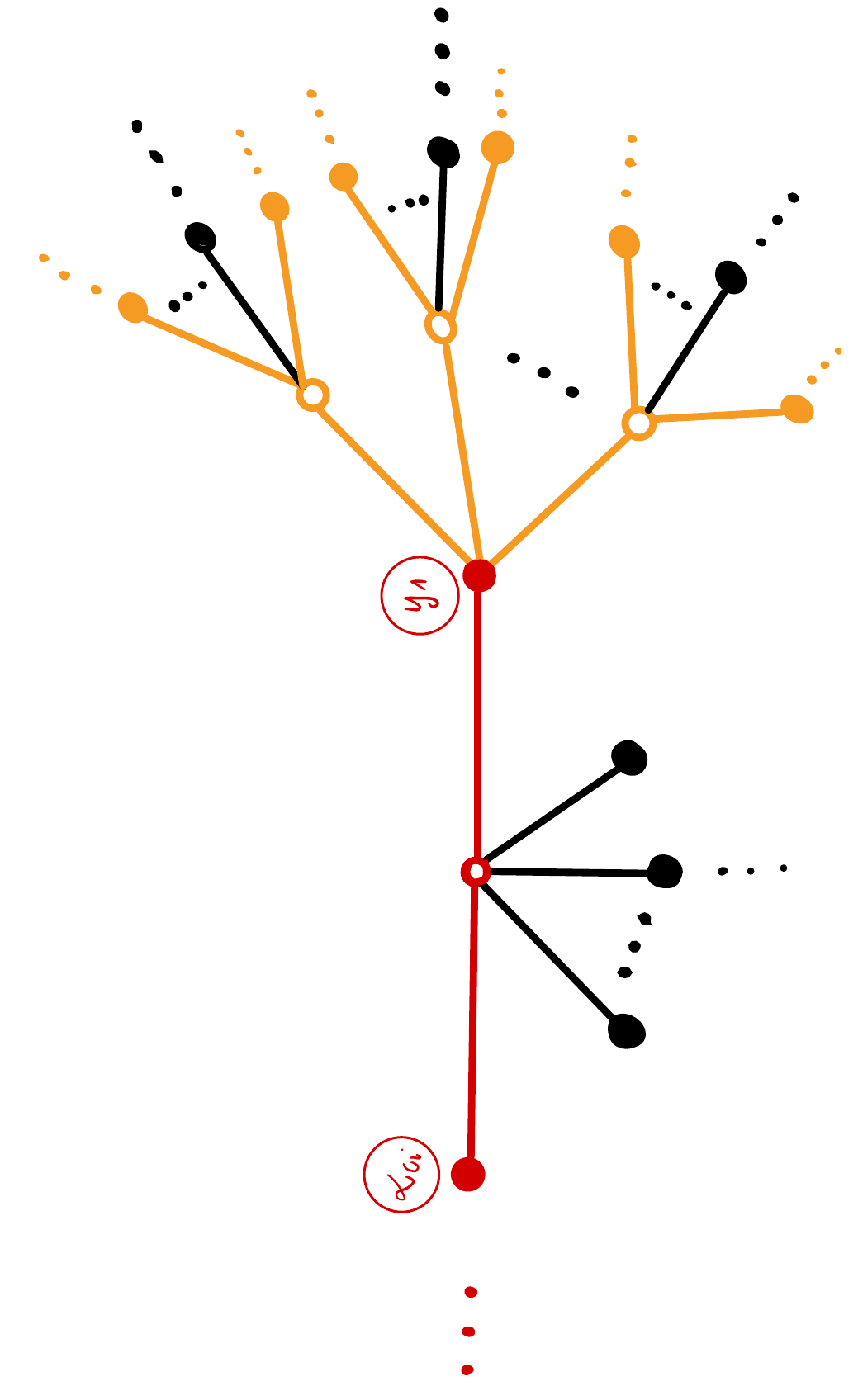}}  
    }
    \quad
    \subfloat[]{\label{subfig:proofspineB}
                {\includegraphics[clip,trim=6cm 0cm 0cm 3.5cm,angle=-90, width=0.4\textwidth]{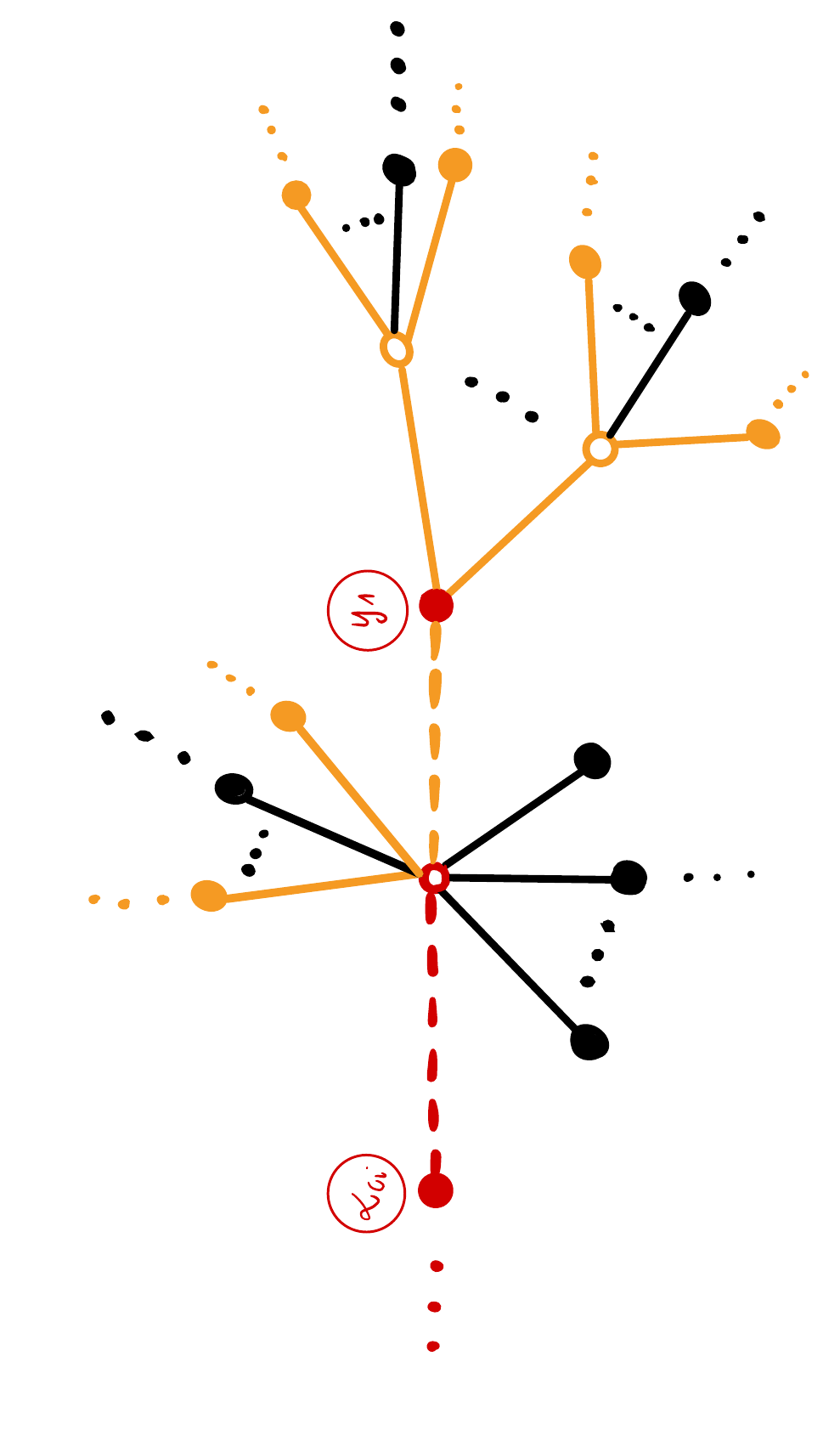}}
    }
    \caption{\protect\subref{subfig:proofspineA} The initial configuration of the tree $T_1$. 
    The black vertices $\Root_{c_i},y_1$ are spine and such that $\Root_{c_i}<y_1$.
    The orange edges mark one of the possible pathways of the spine.
    \newline
    \protect\subref{subfig:proofspineB} The structure of the tree \protect\subref{subfig:proofspineA}
    after performing the operation $\Bend_{\Root_{c_i},y_1}$.}
    \label{fig:proofspine}
\end{figure}

\emph{Case 2: $i$ is a spine vertex.} 
Note that in the spine treatment any two
bend operations commute. Furthermore, any jump operation performed in a given
spine cluster does not affect the remaining spine clusters. Therefore, the final
output does not depend on the order in which we choose the clusters from
$\mathfrak{C}$. For this reason we may start the analysis of the time evolution
of any spine cluster $C$ from the input tree $T_1$ where all edges are solid
(i.e., we may assume that the external loop is executed first for the cluster
$C$).

If $i\in \mathcal{B}\setminus \{1,n\}$ is not one of the endpoints of the spine, 
it belongs to exactly
two spine clusters $c_1,c_2$ in the plane tree $T_1$. Let us fix some
$j\in\{1,2\}$; with the notations of \cref{subfig:untouched1}, $i\in\{
\alpha_{c_j}, y_1 \}$ is one of the two spine vertices in the cluster $c_j$ and
$\Bend_{\alpha_{c_j},y_1}$ is one of the operations that are performed in the
iteration of the external loop in the spine treatment for $C=c_j$. From
\cref{fig:proofspine} it follows that for each spine vertex $v\in\{ \alpha_{c_j},
y_1 \}$ in this cluster $\Numberb_v^{\text{after}}=\Numberb_v^{\text{before}}-1$.
Since there are two choices for $j\in\{1,2\}$ it follows that the variable
$\Numberb_i$ decreases (at least) twice during the whole algorithm, as required.

An analogous argument shows that if $i\in\{1,n\}$ is one of the endpoints of the
spine then the variable $\Numberb_i$ decreases (at least) once during the whole
algorithm, as required.

\medskip

At the
beginning of the algorithm  $\Numberb_i^{\text{initial}}=a_i$ for each
$i\in\{1,\dots,n\}$.
The above discussion shows therefore that in the output tree $T_2$
\begin{equation}\label{eq:inequalities}
    \Numberb^{\text{final}}_1\le a_1-1,\quad \Numberb^{\text{final}}_2\le a_2-2,
    \quad \dots,\quad \Numberb^{\text{final}}_{n-1}\le a_{n-1}-2,
\quad \Numberb^{\text{final}}_n\le a_n-1.
\end{equation}
It follows  that
\[
\sum_{i=1}^{n} \Numberb^{\text{final}}_i \leq
\sum_{i=1}^{n} (a_i-1)-(n-2) =(k-1)-(n-2)=k-(n-1),
\]
where we used the relationship \eqref{eq:sum-of-a}. %
In each step of the algorithm each white vertex is attracted to at least one
black neighbor; it follows that the left-hand side of the above inequality is an
upper bound for the number of white vertices in the output tree $T_2$.

On the other hand, from the first part of the proof we know that the output tree
$T_2$ has exactly $k-(n-1)$ white vertices and the above inequality is saturated:
\begin{equation}\label{eq:magic-sum}
    \sum_{i=1}^{n} \Numberb^{\text{final}}_i=k-(n-1).
\end{equation}
It follows that each white vertex is attracted to exactly one black neighbor;
also the tree $T_2$ is a Stanley tree of type
$( \Numberb^{\text{final}}_1, \dots, \Numberb^{\text{final}}_n )$.
If at least one of the inequalities \eqref{eq:inequalities} 
was strict, this would contradict \eqref{eq:magic-sum}.
We proved in this way that $\Numberb^{\text{final}}_i=b_i$, which is given by
\eqref{eq:a-and-b},
which completes the proof.
\end{proof}

\begin{remark}
The algorithm has one random component in it, namely when we choose the
ordering for $\Sigma$ on Page~\pageref{text:definiton-of-Sigma}. In fact, this
choice of the ordering does not impact the outcome of the algorithm. Probably
the best way to prove it is to use the alternative description of the algorithm
which we provide in \cref{sec:alternative} which is fully deterministic and
does not refer to such choices.
\end{remark}

\section{Alternative description of the bijection}
\label{sec:alternative}

In the current section we will provide an alternative description of the
bijection $\CT$. This alternative viewpoint on $\CT$ will be essential for the
construction of the inverse map $\CT^{-1}$ in \cref{sec:inverse}.

Recall that the \emph{spine} in the output tree $T_2$ is the path connecting
the two black vertices with the labels $1$ and $n$. We orient the non-spine
edges of the tree $T_2$ in the direction of the spine. In this way we can view
the plane tree $T_2$ as a path (=the spine) to which there are attached plane
trees. With this perspective in mind, the output tree $T_2$ will turn out to be
uniquely determined by: the local information about the structure of white
non-spine vertices, the local information about the structure of black
non-spine vertices, and the information about the spine and its small
neighborhood. The aforementioned local information for a black non-spine vertex
$y$ is just the list of the labels of the edges connecting $y$ to its children.
For a white non-spine vertex the local information contains more data; see 
\cref{sec:direct-neighborhood-white}. This local information will be provided
separately for white non-spine vertices
(\crefrange{sec:organic-vs-artificial}{sec:children-artificial}) and for
certain black non-spine vertices (\cref{sec:children-of-black-neverspine}).
Finally, in \cref{sec:anatomy-spine,sec:spine-neighborhood} we will describe
the spine vertices of $T_2$ as well as its neighborhood.

\subsection{White vertices: organic versus artificial}
\label{sec:organic-vs-artificial}

Recall that in the description of the jump operation in \cref{sec:output-of-jump}
the newly created white vertex $w$ (see \cref{subfig:jumpB,subfig:scjumpB}) was
declared to be \emph{artificial}. Each white vertex that was \emph{not} created
in this way by some jump operation will be referred to as \emph{organic}.
More precisely, we declare that all white vertices in the initial tree $T_1$
are organic. 
The bend operation can be viewed as merging two white vertices
(with the notations of \cref{subfig:exA} these are the vertices $v_1$ and $v_2$);
it turns out that for each bend operation that is performed during the execution
of the algorithm the vertex $v_2$ is organic. We declare that the vertex created
by merging $v_1$ and $v_2$ is organic (respectively, artificial) if and only if
$v_1$ was organic (respectively, artificial).

\subsection{Direct neighborhood of a white non-spine vertex}
\label{sec:direct-neighborhood-white}

For any non-spine white vertex~$w$ of $T_2$ we will describe  \emph{the direct
    neighborhood of $w$} that is defined as:
\begin{itemize}
    \item the labels of the children of $w$, together with the labels of the
corresponding edges, and
    \item the label of the \emph{parental edge of $w$}, defined as the edge
that connects $w$ with its parent (however, we are not interested in the
label of this parent).
\end{itemize}
We will describe such direct neighborhoods separately for white non-spine
organic vertices and for white non-spine artificial vertices.

\subsection{Direct neighborhood of white non-spine organic vertices in the output tree $T_2$}
\label{sec:children-of-organic}

Our starting point is the tree $T_1$.  Our goal in this section is to find all
non-spine organic white vertices in $T_2$ and then for each such a vertex to
describe its direct neighborhood. For this reason we disregard now all
artificial white vertices as well as the information about the children of
the black vertices.

\subsubsection{The first pruning} 
\label{sec:first-pruning} 

With this quite narrow perspective in mind, each jump operation (with the
notations from \cref{fig:jump,fig:scjump}) can be seen as removal of the edge
$E_1$ between $y$ and $v_1$ as well as removal of the white vertices $v_2$ and
$v_3$. The remaining children of $y$ (i.e.,~the white vertices $v_4,v_5,\dots$)
remain intact, but we are not interested in keeping track of the parents of
white vertices (we are interested in keeping track of just the label of the
parental edge for a white vertex). Still keeping our narrow perspective in
mind, it follows that the application of all jump operations in the rib
treatment part of the algorithm (\cref{sec:rib}) is equivalent to the following
\emph{first pruning procedure}:
\begin{quotation}
\emph{for each non-spine white vertex $v_1$ and its child $y$ that carries a bigger label
    than its grandparent (i.e.,~$y>\Root_{v_1}$) 
we remove:}
\begin{itemize}
    \item \emph{the vertex $y$ together with the edge that points towards its
    parent,}
    
    \item \emph{the two leftmost children of $y$ (with the notations of
\cref{fig:jump,fig:scjump} they correspond to $v_2$ and $v_3$) together their
children, and the edges that connect them.}
\end{itemize}  
\end{quotation}
The side effect of the removal of the vertex $y$ is that each of the edges
$E_4,E_5,\dots$ (which formerly connected $y$ to its non-two-leftmost children)
has only one endpoint, namely the white one. 

\begin{remark}
    \label{rem:niceorder} Later on we will use the following observation: our
initial choice of the cyclic order of the edges around white vertices
(\cref{sec:tt0}) implies that after this pruning, for each  white non-spine
vertex $w$ the labels of its remaining children are arranged in an
increasing way from right to left; furthermore, each of these labels is
smaller than the label of the parent of $w$, i.e., $\Root_w$.
\end{remark}

\subsubsection{The second pruning}
\label{sec:second-pruning}
If an edge between a non-spine white vertex $w$ and its child $b$ still remains
after the above pruning procedure, this means that during the execution of the
algorithm $\CT$ a bend operation $\Bend_{\Root_w,b}$ was performed.
With the notations of \cref{subfig:exA}
this operation does not affect the white vertices $v_3,v_4,\dots$ (i.e., the
children of the vertex $y=b$ except for the leftmost child). With our narrow
perspective in mind, we are not interested in keeping track of the parent of
these vertices, which motivates the following \emph{second pruning procedure}: 
\begin{quotation} 
    \emph{we disconnect each black non-spine vertex $b$ from the edges connecting $b$
    with each of its children, with the exception of the leftmost child.}
\end{quotation}
Again, as an outcome of this disconnection there are some edges that have only
one endpoint, the white one.

Note that since in the initial tree
$T_0$ the degree of each black vertex was at least $2$, in the outcome of the
second pruning each black non-spine vertex has degree equal exactly to~$2$.
Our analysis of the impact of the bend operation $\Bend_{\Root_w,b}$ is not 
complete yet and will be continued in a moment.

\subsubsection{Folding}
\label{sec:folding}

After performing the above two pruning operations, the tree splits
into a number of connected components.
Let $T$ be one of these connected components that are disjoint with the spine;
it is an oriented tree, which has a white vertex $v$ as the
root. Additionally, this root $v$ has a special edge (\emph{the parental edge})
that is pointing in the former direction of the spine; this edge has only one
endpoint.

Consider some black vertex $b\in T$. Its degree is equal to~$2$ and the
aforementioned bend operation $\Bend_{\Root_w,b}$ can be seen as a
counterclockwise rotation of the edge that connects $b$ with its only child
towards the edge that connects $b$ with its parent so that these two edges are
merged into a single edge (see \cref{fig:bending2}). The label of the merged
edge is declared to be the label of the edge that formerly connected $b$ with
its child, with the notations of \cref{fig:bending2} this label is equal to
$e$. As a result of the edge merging, a pair of white vertices -- the child of
$b$, and the parent of $b$ -- is merged into a single white vertex and the
vertex $b$ becomes a leaf. After the above \emph{folding procedure} is applied
iteratively to all black vertices in $T$,  the connected component $T$ becomes
a single white vertex connected to a number of black vertices, and together
with the parental edge of the root. For an example see \cref{fig:bending}.

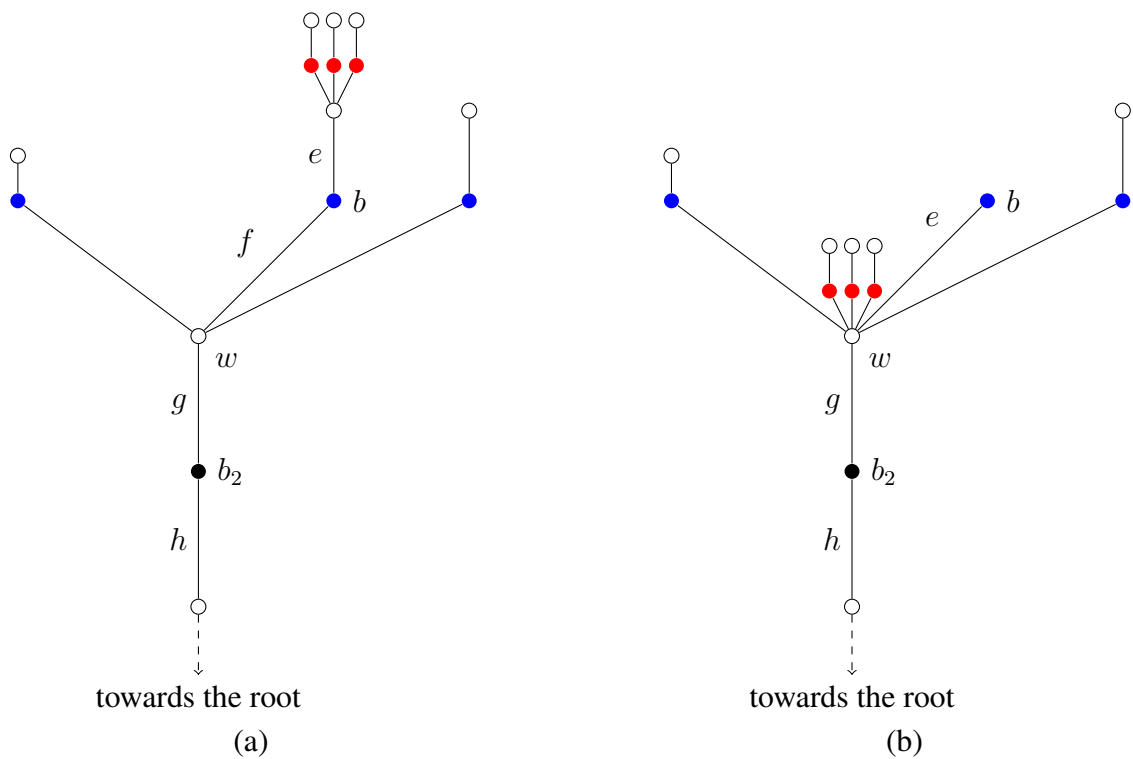
\begin{figure}
    \centering
    \subfloat[]{\begin{tikzpicture}[scale=0.6]
            \node (root) at (0,-4) {towards the root};
            \node[circle, draw=black, inner sep=2pt] (white1) at (0, -2) {}; 
            
            \node[circle, fill=black, inner sep=2pt,label={[label distance=0ex]0:$b_2$}] (black1) at (0, 1) {};
            
            \node[circle, draw=black, inner sep=2pt,label={[label distance=0ex]-45:$w$}] (white2) at (0, 4) {};

            \node[circle, fill=blue, inner sep=2pt]  (branchA1) at (-4, 7) {};
            \node[circle, draw=black, inner sep=2pt] (branchA2) at (-4, 8) {};          
            
            \node[circle, fill=blue, inner sep=2pt,label={[label distance=0ex]0:$b$}]  (branchB1) at (3, 7) {};
            \node[circle, draw=black, inner sep=2pt] (branchB2) at (3, 9) {};          
            
            \node[circle, fill=blue, inner sep=2pt]  (branchC1) at (6, 7) {};
            \node[circle, draw=black, inner sep=2pt] (branchC2) at (6, 9) {};

            \draw (branchB2) 
            ++(-0.5,1) node[circle, fill=red, inner sep=2pt] (red1) {}
            ++(0,1) node[circle, draw=black, inner sep=2pt] (white-away1) {};
            \draw (branchB2) -- (red1) -- (white-away1);
            
            \draw (branchB2) 
            ++(0,1) node[circle, fill=red, inner sep=2pt] (red1) {}
            ++(0,1) node[circle, draw=black, inner sep=2pt] (white-away1) {};
            \draw (branchB2) -- (red1) -- (white-away1);
            
            \draw (branchB2) 
            ++(0.5,1) node[circle, fill=red, inner sep=2pt] (red1) {}
            ++(0,1) node[circle, draw=black, inner sep=2pt] (white-away1) {};
            \draw (branchB2) -- (red1) -- (white-away1);
            
            \draw (white1) edge["$h$"] (black1)
            (black1) edge["$g$"] (white2)
            (white2) -- (branchA1) -- (branchA2);   
            
            \draw (white2) edge["$f$"] (branchB1)
            (branchB1) edge["$e$"]  (branchB2);
            \draw (white2) -- (branchC1) -- (branchC2);
            
            \draw[->, dashed] (white1) -- (root);
            
        \end{tikzpicture}
        \label{subfig:bending2in}
    }
    \hfill
    \subfloat[]{\begin{tikzpicture}[scale=0.6]
            \node (root) at (0,-4) {towards the root};
            \node[circle, draw=black, inner sep=2pt] (white1) at (0, -2) {}; 
            
            \node[circle, fill=black, inner sep=2pt,label={[label distance=0ex]0:$b_2$}] (black1) at (0, 1) {};
            
            \node[circle, draw=black, inner sep=2pt,label={[label distance=0ex]-45:$w$}] (white2) at (0, 4) {};

            \node[circle, fill=blue, inner sep=2pt]  (branchA1) at (-4, 7) {};
            \node[circle, draw=black, inner sep=2pt] (branchA2) at (-4, 8) {};          
            
            \node[circle, fill=blue, inner sep=2pt,label={[label distance=0ex]0:$b$}]  (branchB1) at (3, 7) {};

            \node[circle, fill=blue, inner sep=2pt]  (branchC1) at (6, 7) {};
            \node[circle, draw=black, inner sep=2pt] (branchC2) at (6, 9) {};

            \draw (white2) 
            ++(-0.5,1) node[circle, fill=red, inner sep=2pt] (red1) {}
            ++(0,1) node[circle, draw=black, inner sep=2pt] (white-away1) {};
            \draw (white2) -- (red1) -- (white-away1);
            
            \draw (white2) 
            ++(0,1) node[circle, fill=red, inner sep=2pt] (red1) {}
            ++(0,1) node[circle, draw=black, inner sep=2pt] (white-away1) {};
            \draw (white2) -- (red1) -- (white-away1);
            
            \draw (white2) 
            ++(0.5,1) node[circle, fill=red, inner sep=2pt] (red1) {}
            ++(0,1) node[circle, draw=black, inner sep=2pt] (white-away1) {};
            \draw (white2) -- (red1) -- (white-away1);
            
            \draw (white1) edge["$h$"] (black1)
            (black1) edge["$g$"] (white2)
            (white2) -- (branchA1) -- (branchA2);   
            
            \draw (white2) edge[near end,"$e$"] (branchB1);
            \draw (white2) -- (branchC1) -- (branchC2);  
            
            \draw[->, dashed] (white1) -- (root);
                     
        \end{tikzpicture}
        \label{subfig:bending2out}
    }

    \caption{\protect\subref{subfig:bending2in} The initial configuration of a
    part of the tree~$T$.  Additional labels on this figure will be explained
    and used in \cref{sec:labels}. 
    \newline 
    \protect\subref{subfig:bending2out} The outcome of folding at the
    vertex $b$.} \label{fig:bending2}
\end{figure}

\begin{figure}
    \centering \subfloat[]{\label{subfig:bendingin} {\includegraphics[clip,
        trim=0.5cm 5cm 1cm 2cm,width=0.4\textwidth]{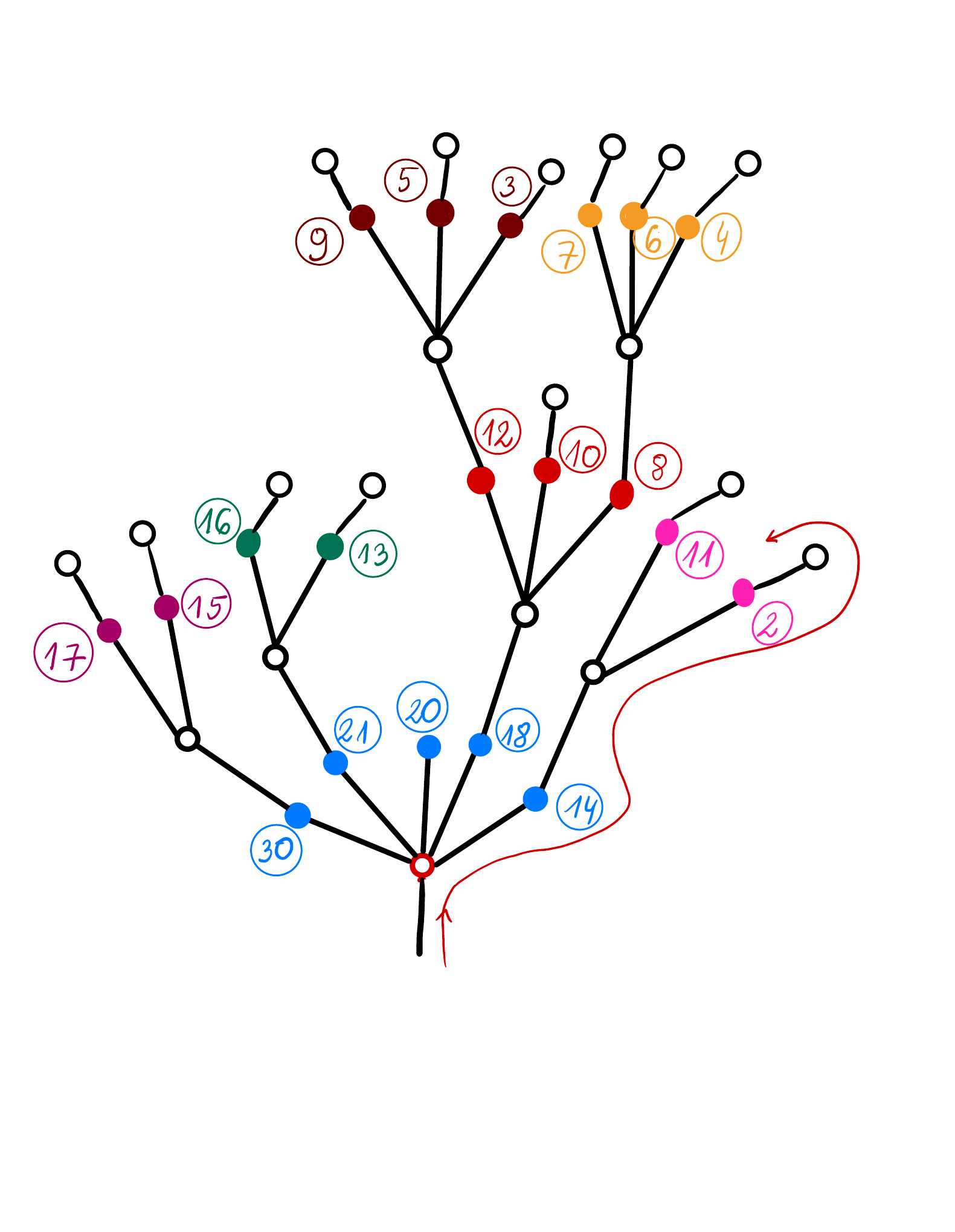}}} \quad
\subfloat[]{\label{subfig:bendingout} {\includegraphics[clip, trim=1.5cm 6cm
        2cm 4cm,width=0.4\textwidth]{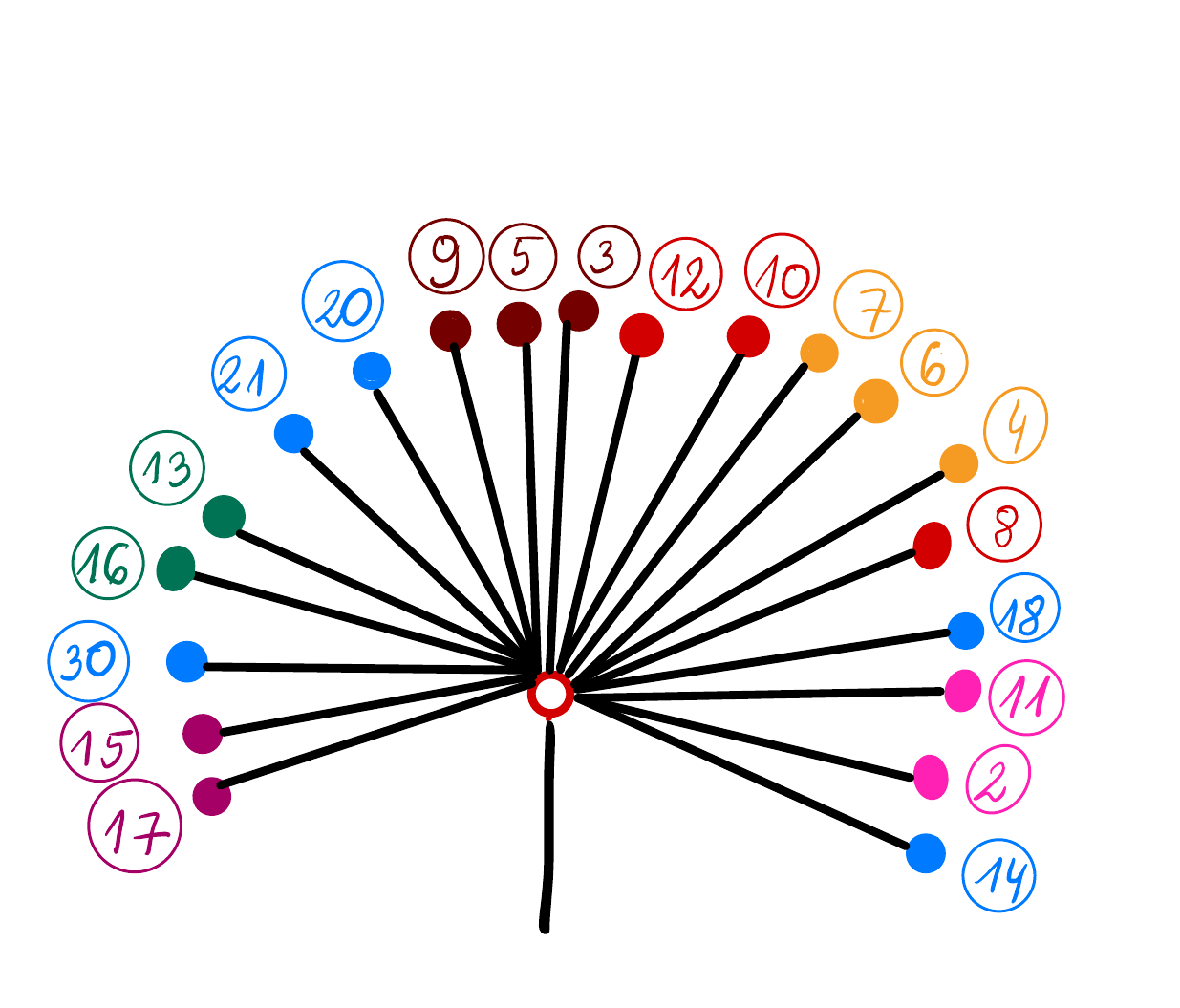}}}
    
\caption{\protect\subref{subfig:bendingin} Example of a connected component $T$
    that has a white root of the output of the two pruning procedures. For simplicity
    the edge labels are not shown. The black vertices having the same parent were
    drawn with the same color. The thin red line with the arrows indicates the
    beginning of the depth-first search traversing of the tree; the black vertices
    are visited in the order $14, 2, 11, 18, 8, 4, \dots$ (only the first visit in
    a given vertex is listed). \newline \protect\subref{subfig:bendingout} The
    output of the folding procedure applied to the tree $T$ from
    \protect\subref{subfig:bendingin}.} \label{fig:bending}
\end{figure}

The above folding procedure can be described alternatively as follows: we
traverse the tree $T$ by the depth-first search, starting at the root and
touching the edmges by the left hand. For example, the beginning of such a
depth-first search is indicated on \cref{subfig:bendingin} by the red line with
an arrow. We order the black vertices according to the time of the first visit
in a given vertex. This order coincides with the children of the root $v$
(listed in the counterclockwise order) in the output of folding applied to $T$.

\subsubsection{Conclusion}

To summarize the current subsection: there is a bijective correspondence
between (i) the white non-spine organic vertices in the output tree $T_2$, and
(ii) the connected non-spine components of the outcome of the two pruning
procedures. The children of such white organic non-spine vertices in $T_2$ can
be found by the above depth-first traversing of the connected component. This
observation will critical for the construction of the inverse map~$\CT^{-1}$.

\subsection{Direct neighborhood of white non-spine artificial vertices in the output tree $T_2$}
\label{sec:children-artificial}

Our goal in this section is to find all non-spine artificial white vertices in
$T_2$ and then for each such a vertex to find its direct neighborhood.

Each white artificial vertex $w$ is created during the execution of some jump
operation $\Jump_{x,y}$ so there is a bijective correspondence between the
artificial white vertices and the collection of certain pairs $(x,y)$, where
$x$ and $y$ are black vertices such that $x$ is a grandparent of $y$ and $x<y$,
cf.~\cref{fig:jump,fig:scjump}. Let us fix the values of $w$, $x$, and $y$; we
shall describe now the children of $w$.  For this reason we will disregard
tracking of the children of the black vertices as well as we will disregard these 
white vertices that will not be merged with $w$ by some bend operations.

After creation, the vertex $w$ may gain children only by the application
of some bend operations. It follows that we may restrict our attention only to
the descendants of the vertex~$w$ and disregard the remaining part of the tree.
This observation motivates the following procedure.

\begin{quotation} 
\emph{We apply the jump operation $\Jump_{x,y}$ to the initial tree $T_1$. The
    resulting tree has a unique artificial white vertex, which is denoted by $w$,
    as above. Then we keep only the vertex $w$, the edge adjacent to $w$ that is in
    the direction of the spine, and the descendants of $w$. We remove the remaining
    part of the tree.}
\end{quotation}

\medskip

A discussion analogous to the one from \cref{sec:children-of-organic} motivates
the following first pruning procedure: 
\begin{quotation}
    \emph{for each remaining white vertex $w'$ and each child $b$ of $w'$ that
    carries a bigger label than its grandparent (i.e.,~$b>\Root_{w'}$) we
    remove the vertex $b$ and all of its descendants, }
\end{quotation}
as well as the following second pruning procedure:
\begin{quotation} 
    \emph{for each remaining black vertex $b$ and each child $w'$ of $b$ that
    is not left-most, we remove the edge between $b$ and $w'$ as well as the
    vertex $w'$, all of its descendants, and all edges between them.}
\end{quotation}

The outcome is a tree, which has the white artificial vertex $w$ as the root.
Furthermore, each black vertex has degree $2$. Exactly as in
\cref{sec:folding}, we apply folding to this tree; as a result we obtain a tree, which
consists of a single white vertex $w$ as the root, which is connected to some
number of black vertices. Additionally, the root $w$ is connected to the parental
edge, which has only one endpoint. This tree is equal to the direct neighborhood
of the artificial white vertex $w$ in the output tree $T_2$.

\medskip

\begin{figure}
    \centering

\def\MarkLt{6pt}
\def\MarkSep{3pt}

\tikzset{
    MyMarks/.style={
        postaction={decorate,
            decoration={
                markings,
                mark=at position 0.5 with
                {
                    \begin{scope}[#1,thick,xslant=0.0]
                        \draw[-] (-0.5*\MarkSep,-\MarkLt) -- (-0.5*\MarkSep,\MarkLt) ;
                        \draw[-] (0.5*\MarkSep,-\MarkLt) -- (0.5*\MarkSep,\MarkLt) ;
                    \end{scope}
                }
            }
        }
    },
    BlueMarks/.default={blue},
}    

\begin{tikzpicture}[scale=1.5]
    
    \node[circle, draw=black, inner sep=2pt] (white2) at (0, 0) {$w$}; 
    
    \node[circle, fill=blue, inner sep=2pt,label={[label distance=0ex]180:$c$}] (c) at (180:2) {};
    \node[circle, fill=blue, inner sep=2pt] (junk1) at (-157.5:1) {};
    \node[circle, fill=blue, inner sep=2pt,label={[label distance=0ex]-135:$b$}] (b) at (-135:2) {};
    
    \node[circle, fill=blue, inner sep=2pt] (junk2) at (-105:1) {};
    \node[circle, fill=blue, inner sep=2pt] (junk3) at (-120:1) {};
    \node[circle, fill=blue, inner sep=2pt] (junk4) at (-60:1) {};
    \node[circle, fill=blue, inner sep=2pt] (junk5) at (-75:1) {};
    
    \node[circle, fill=blue, inner sep=2pt,label={[label distance=0ex]-90:$a$}] (a) at (-90:2) {};
    
    \node[circle, fill=black, inner sep=2pt,label={[label distance=0ex]-45:$y$}] (y) at (-45:2) {};
    
    \draw[blue] (white2) -- (a) (white2) -- (b) (white2) -- (c) (white2)--(junk1) (white2)--(junk2) (white2)--(junk3) (white2)--(junk4) (white2)--(junk5);
    
    \draw[MyMarks=red] (white2) -- (y);

    \node[circle, fill=red, inner sep=2pt,label={[label distance=0ex]0:$f$}] (c) at (0:2) {};
    \node[circle, fill=red, inner sep=2pt] (junk1) at (180-157.5:1) {};
    \node[circle, fill=red, inner sep=2pt,label={[label distance=0ex]180-135:$e$}] (b) at (180-135:2) {};
    
    \node[circle, fill=red, inner sep=2pt] (junk2) at (180-105:1) {};
    \node[circle, fill=red, inner sep=2pt] (junk3) at (180-120:1) {};
    \node[circle, fill=red, inner sep=2pt] (junk4) at (180-60:1) {};
    \node[circle, fill=red, inner sep=2pt] (junk5) at (180-75:1) {};
    
    \node[circle, fill=red, inner sep=2pt,label={[label distance=0ex]180-90:$d$}] (a) at (180-90:2) {};
    
    \node[circle, fill=black, inner sep=2pt,label={[label distance=0ex]45:$j$}] (y) at (180-45:2) {};
    \node (root) at (180-45:3) {spine};
    \draw[dashed,->] (y) -- (root);

    \draw[red] (white2) -- (a) (white2) -- (b) (white2) -- (c) (white2)--(junk1) (white2)--(junk2) (white2)--(junk3) (white2)--(junk4) (white2)--(junk5);
    
    \draw[MyMarks=blue] (white2) -- (y);

\end{tikzpicture}

    \caption{The neighbors of an artificial white vertex $w$ in the output
    tree~$T_2$; see \cref{subfig:jumpB,subfig:scjumpB}. The vertex $j$ is the
    parent of $w$. The blue vertices $a,b,c$, and the other adjacent blue
    vertices are the outcome of the pruning and folding of the blue tree on
    \cref{subfig:jumpA,subfig:scjumpA} that starts with the edge~$E_2$. The red
    vertices $d,e,f$, and the other adjacent red vertices are the outcome of
    the pruning and folding of the red tree on
    \cref{subfig:jumpA,subfig:scjumpA} that starts with the edge~$E_3$. By
    \cref{rem:niceorder} and \cref{lem:what-is-j} the vertex $y$ carries the
    biggest label among all neighbors of $w$. } \label{fig:artificial-anatomy}
\end{figure}
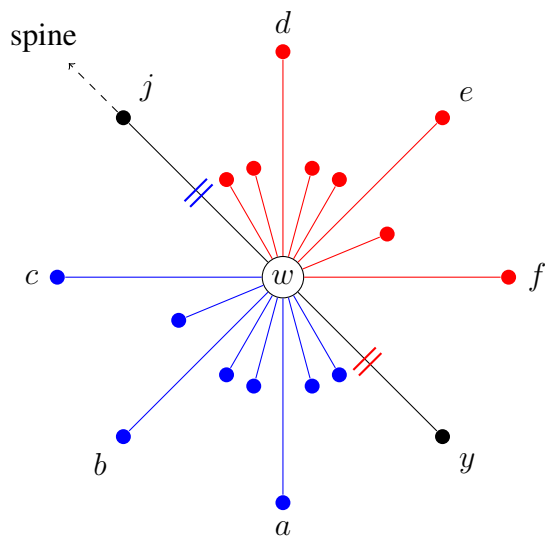

To summarize: the set of \emph{all} neighbors of an artificial vertex $w$
together with the information about their cyclic order looks like on
\cref{fig:artificial-anatomy}. A consequence \cref{rem:niceorder} and
\cref{lem:what-is-j} is that the vertex $y$ carries the biggest label among all
neighbors of $w$ in the output tree $T_2$.

\subsection{Children of black vertices: how to name the white vertices?}
\label{sec:naming-of-white}

For a given black non-spine vertex we intend to find the list of its children
in the output tree $T_2$. Each such child is a white non-spine vertex, and in
\cref{sec:children-of-organic,sec:children-artificial} we already found the
collection of such vertices. Now we need some naming convention that would
allow to match each child to some white vertex from
\cref{sec:children-of-organic,sec:children-artificial}. In this way a person
who has the access to the tree $T_1$ but has no access to the tree $T_2$ would
still be able to distinguish the white non-spine vertices of $T_2$. Our naming
convention will be defined separately for the white organic and for the white
artificial vertices.

In \cref{sec:children-of-organic} we showed that each white organic vertex $w$
in $T_2$ is an outcome of folding applied to a certain tree $T$;  in order to
name $w$ we will identify it with the root of the tree $T$ (which is just a
white vertex in $T_1$).

To each white artificial vertex $w$ of the tree $T_2$ we associate the largest
label of its neighbors; see \cref{fig:artificial-anatomy}. This largest label
$y$ is an element of the set $\{1,\dots,n\}$ and can be identified with a black
vertex in the tree $T_1$. The discussion from \cref{sec:children-artificial}
shows that this map is injective and allows us to pinpoint uniquely each
artificial white vertex. Furthermore, the label $y$ has the following natural
interpretation: the artificial vertex $w$ was created by some jump operation of
the form $\Jump_{\cdot,y}$.

With these conventions we will give the children of a non-spine black vertex in
$T_2$ in the form of a list of (black and white) vertices of the tree $T_1$.

\subsection{Children of black never-spine vertices}
\label{sec:children-of-black-neverspine}

Consider a black vertex $y$ that in the input tree $T_1$ was \emph{not} a
spine vertex. It turns out (we will prove it later in
\crefrange{sec:anatomy-spine}{sec:spine-neighborhood}) that the vertex $y$ in
the output tree $T_2$ is still not a spine vertex; in the following we will
describe the children of $y$ in $T_2$ as well as the labels on the edges
connecting $y$ with its children.

Recall that we oriented all non-spine edges in the tree $T_1$ towards the
spine. Now, we additionally orient some spine edges in $T_1$ in such a way that
for each white spine vertex~$w$ its parent is equal to the root $\Root_w$ of
the corresponding cluster. With this convention, we denote by $x$ the
grandparent of $y$. It follows that during the action of the algorithm $\CT$
exactly one of the following two operations was performed: either the bend
operation $\Bend_{x,y}$ (in the case when $x>y$) or the jump operation
$\Jump_{x,y}$ (in the case when $x<y$).

\subsection{Children of a black never-spine vertex, the case when $x>y$} 
\label{sec:never-spine-x>y}

\subsubsection{The tree $T$}

In this case the bend operation $\Bend_{x,y}$ was performed. As a result, the
edge connecting $y$ to its leftmost child in $T_1$ is \emph{not} among the
edges that connect $y$ to its children in $T_2$. The other  (i.e., the
non-leftmost) edges connecting $y$ to its children in $T_1$ will remain in
$T_2$; these edges will connect $y$ to its organic children. On the other hand,
the vertex $y$ may gain some new, artificial children as an outcome of some
jump operations. We will discuss them in the following.

\medskip

Recall that with the notations of \cref{fig:jump,fig:scjump} we denote by
$j=j(x',y')$ the black vertex that as a result of a jump operation
$\Jump_{x',y'}$ gains a new child, which is a white artificial vertex. This
vertex $j$ is given in terms of the input tree $T_1$ by the algorithm contained
in the proof of \cref{lem:what-is-j}. The problem that we currently encounter
is roughly the opposite: for a given black vertex $y$ of $T_1$ we should find
all its artificial children in the output tree $T_2$ or, equivalently, all jump
operations $\Jump_{x',y'}$ performed during the algorithm for which
$y=j(x',y')$. Our strategy is to construct a certain tree $T$ that is a
subtree of $T_1$. This plane tree will have the vertex $y$ as the root, and
it will have the following two properties:
\begin{itemize} 
    \item whenever $x',y'$ are black   %
vertices in $T_1$ such that $\Jump_{x',y'}$ is one of the operations performed
during the algorithm $\CT$ then
\[ j(x',y')= y  \quad \iff \quad  x',y' \in T, \]
    \item if $x',y'$ are black vertices in $T$ such that $x'$ is the grandparent of $y'$ then
    $\Jump_{x',y'}$ is one of the operations performed during the algorithm $\CT$
    (or, equivalently, $x'<y'$).
\end{itemize}
In this way there will be a bijection between the black non-root vertices of
the tree $T$ and the artificial children of the vertex $y$ in $T_2$, which maps
the black vertex $y'$ to the artificial white vertex created during the operation of
the form $\Jump_{\cdot,y'}$.

\subsubsection{The prunings}
\label{sec:black-prunings}

As the first step towards the construction of the tree $T$ we perform the
following \emph{zeroth pruning}:
\begin{quotation}
    \emph{in the tree $T_1$ we keep only the vertex $y$ and its offspring.}
\end{quotation}
Recall that the bend operation $\Bend_{x,y}$ was performed and the edge
connecting $y$ to its leftmost child in $T_1$ is \emph{not} among the edges
that connect $y$ to its children in $T_2$. This motivates the following
operation.
\begin{quotation}
    \emph{Additionally, we remove the label from the edge connecting $y$ to its
    leftmost child.}
\end{quotation}

\medskip

The condition \ref{entry:smaller} from the algorithm contained in the proof of
\cref{lem:what-is-j} motivates the following \emph{first pruning procedure}:
\begin{quotation}
    \emph{for each remaining white vertex $v_1$ and its child $y'$ that
    carries a \emph{smaller} label than its grandparent we remove the edge connecting
    $v_1$ with $y'$, as well as the vertex $y'$ and its offspring, as well as
    the edges connecting them.}
\end{quotation}
Compare this to \cref{sec:first-pruning} where we removed each vertex that is
\emph{bigger} than its grandparent. As a result, the following analogue of
\cref{rem:niceorder} holds true.
\begin{remark}
    \label{rem:niceorder2} 
Later on we will use the following observation: our
initial choice of the cyclic order of the edges around white vertices
(\cref{sec:tt0}) implies that after this pruning, for each  white non-spine
vertex $w$ the labels of its remaining children are arranged in a
\emph{decreasing way in the clockwise order}; furthermore, each of these labels
is \emph{bigger} than the label of the parent of $w$, i.e., $\Root_w$.
\end{remark}

\medskip

The other condition \ref{entry:leftist} from the algorithm contained in the proof of
\cref{lem:what-is-j} motivates the following \emph{second pruning procedure}:
\begin{quotation}
    \emph{for each remaining black vertex $y'$ such that $y'\neq y$ is not the
    root vertex we remove all non-leftist children of $y'$ as well as their
    offspring and all adjacent edges.}
\end{quotation}
We denote by $T$ the outcome of these prunings. This tree has the property that
the degree of each black vertex (with the exception of the root) is equal to
$2$.

\subsubsection{Jump operation as a replacement}
\label{sec:jump-as-replacement}

Let $w$ be a child of the root vertex $y$. 
We start with the case when $w$ is not the leftmost child of $y$. Let the black
vertices $V_1,\dots,V_d$ be the children of $w$, listed in the clockwise order.
After applying the jump operations that correspond to the cluster $w$ the whole
subtree of $T$ that has $w$ as the root is replaced by the following collection
of trees attached to the root (we list these trees in the clockwise order):
\begin{multline}
    \label{eq:list-of-verices}
 (\text{the single organic vertex $w$}), \\ 
\shoveleft{(\text{an artificial white vertex that carries the name $V_1$}} \\
\shoveright{\text{to which there is an attached tree $T_{V_1}$ that was formerly attached to $V_1$}),} \\ \dots, \\ 
\shoveleft{(\text{an artificial white vertex that carries the name $V_d$}} \\
{\text{to which there is an attached tree $T_{V_d}$ that was formerly attached to $V_d$});} 
\end{multline}
above we used the naming convention for white vertices from 
\cref{sec:naming-of-white}.

In the case when $w$ is the left-most child, the above discussion remains
valid, however one should remove the first entry of the list
\eqref{eq:list-of-verices}, i.e., the single organic vertex $w$.

\medskip

Each tree $T_{V_i}$ has a white vertex as the root, let the black vertices
$W_1,\dots,W_e$ be is children. Now, performing all jump operations in the
cluster defined by this root corresponds to replacing the following item from
the aforementioned list:
\begin{multline*}
   (\text{the artificial white vertex that carries the name $V_i$} \\
   {\text{to which there is attached tree the $T_{V_i}$})} 
\end{multline*}
by the following collection of trees attached to the root $y$:
\begin{multline*}
    (\text{the artificial white vertex that carries the name $V_i$}), \\ 
    \shoveleft{(\text{an artificial white vertex that carries the name $W_1$}} \\
    \shoveright{\text{to which there is an attached tree that was formerly attached to $W_1$}),} \\ \dots, \\ 
    \shoveleft{(\text{an artificial white vertex that carries the name $W_e$}} \\
    {\text{to which there is an attached tree that was formerly attached to $W_e$}).} 
\end{multline*}

\medskip

We iteratively apply these replacements until each child of the root $y$ is a
leaf. When the algorithm terminates, the ordered list of the  children of the
root $y$ in the output coincides with the ordered list of the children of the
vertex $y$ in the tree $T_2$ that we are looking for. With the convention from
\cref{sec:naming-of-white}, the children of $y$ (listed in the clockwise order)
can be identified with a list of (black and white) vertices of the input tree
$T_1$. In the following we will provide a more direct way of finding this list
based only on the tree $T$. In fact, the elements of the list that we construct
will provide more information; each entry of the list will be a pair of the
form
\begin{equation}
    \label{eq:pair}
     \big( (\text{name of a white vertex $w$}), \quad 
       (\text{the label of the edge connecting $v$ and $w$}) \big).
\end{equation}

\subsubsection{The depth-first search}
\label{sec:black-dfs}

The recursive algorithm from \cref{sec:jump-as-replacement} is just a
complicated way of performing the depth-first search on the tree $T$. We
traverse the tree $T$ by keeping it with the \emph{right} hand (note that
in the analogous depth-first search in \cref{sec:folding} we touched the tree
with the \emph{left} hand). When we enter a non-root vertex~$v$ for the
first time, we proceed as follows:
\begin{itemize}
    \item if $v$ is a white vertex that is a child of the root $y$ but not the
left-most child of the root (each such a vertex corresponds to an white
organic child of $y$ in $T_2$), we append the list with the pair
\[ \big( v,  \quad (\text{the label of the edge connecting $v$ and $y$})   \big), \]
[note that during the zeroth pruning we removed the label from the leftmost
edge of the root so that it is not listed here];

    \item if $v$ is a black vertex (each such a vertex corresponds to a white
artificial child of $y$ in $T_2$) we append the list with the pair
\[ \big( v, \quad (\text{the label of the edge connecting $v$ with its unique child in $T$}) \big). \]
\end{itemize}

This procedure can be regarded as an analogue of folding from
\cref{sec:folding}. This completes the description of the \emph{local
    information} about the children of a black never-spine vertex $y$ in the case
when $x>y$.

\subsection{Children of a black never-spine vertex, the case when $x<y$}
\label{sec:never-spine-x<y}

We plan to proceed in an analogous way as in \cref{sec:never-spine-x>y}. In
this case the jump operation $\Jump_{x,y}$ was performed. 
We  with the conventions of
\cref{fig:jump,fig:scjump})

As a result, the
leftmost child $v_2$ of $y$  and the whole subtree attached to $v_2$ will not
contribute to white artificial children of the vertex $y$. Indeed, if $x',y'$
are black vertices, and are among the offspring of $v_2$ then the algorithm
from the end of the proof of \cref{lem:what-is-j} for calculating $j(x',y')$
does not terminate in the vertex $y$ and hence $j(x',y')\neq y$.
For this reason our first step towards the construction of the tree $T$ is more radical:
\begin{quotation}
    \emph{in the tree $T_1$ we keep only the vertex $y$ and its offspring;
    additionally we remove the left-most child of $y$, its offspring and the
    adjacent edges.}
\end{quotation}

Another consequence of the aforementioned jump operation $\Jump_{x,y}$ is that
the second-leftmost child of the vertex $y$ in the tree $T_1$ (with the
notations of \cref{fig:jump,fig:scjump} it is the vertex $v_2$) will not give
rise to an organic child of $y$ in the output tree $T_2$. On the other hand,
the subtree rooted in this second-leftmost child $v_2$ may give rise to some
children of $y$ that are artificial white vertices. For this reason we may
proceed now with the first and the second pruning in the same way as in
\cref{sec:black-prunings}. Finally, the local information about the children of
the black vertex $y$ in the tree $T_2$ is provided by the depth-first search
algorithm from \cref{sec:black-dfs}.

\subsection{Detailed anatomy of the spine}
\label{sec:anatomy-spine}

\subsubsection{The backbone}

Recall that the spine in the input tree $T_1$ is the path connecting the two
black vertices with the labels $1$ and $n$. We define the \emph{backbone}
$\mathcal{B}$ as the graph that consists of the black spine vertices in $T_1$.
We declare that a pair of backbone vertices $v_1,v_2\in \mathcal{B}$ (with
$v_1\neq v_2$) is connected by an oriented edge $(v_1,v_2)$ pointing from $v_1$
towards $v_2$ if and only if (a) the distance between $v_1$ and $v_2$ in the
tree $T_1$ is equal to $2$ (in other words, $v_1$ and $v_2$ have a common white
neighbor), and (b) the labels of the vertices $v_1,v_2\in\{1,\dots,n\}$ fulfill
$v_1<v_2$. As a result, the backbone is a path graph with $1$ and $n$ as the
endpoints, together with the information about the orientation of the edges.

\subsubsection{Bald and hairy edges in the backbone}

Consider some oriented edge $(v_1,v_2)$ in the backbone, and assume that
$v_2\neq n$ does not carry the maximal label. In this case the vertex $v_2$ is
adjacent to another backbone vertex $v_3$ (with $v_3\neq v_1$). In the tree
$T_1$ the path between the vertices $v_1$ and $v_3$ is of the form
\[ (v_1,w_1,v_2,w_2,v_3) \]
for some white vertices $w_1$, $w_2$. Assume that in the tree $T_1$, going
clockwise around the black vertex $v_2$, the direct successor of the edge
connecting $v_2$ to $w_2$ is the edge connecting $v_2$ to $w_1$; see
\cref{subfig:bald-neighbors}. In this case we will say that the edge
$(v_1,v_2)$ in the backbone is \emph{bald}; see \cref{subfig:bald-backbone}.
The edges in the backbone that are not bald will be called \emph{hairy}.

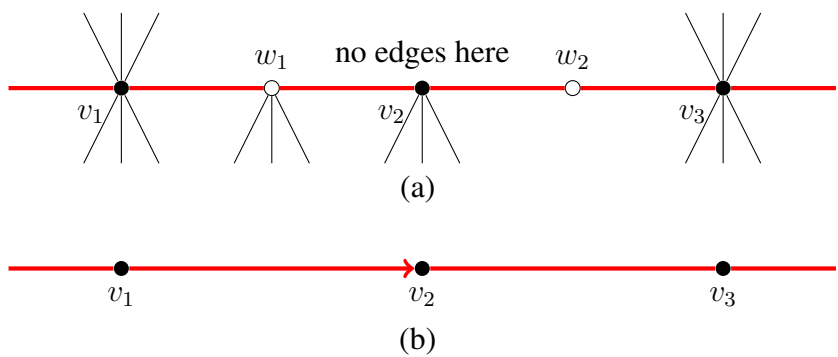
\begin{figure}
    \subfloat[]{\label{subfig:bald-neighbors} \begin{tikzpicture}
            \node[circle, fill=black, inner sep=2pt,label={[label distance=0ex]225:$v_1$}] (v1) at (0, 0) {};        
            \node[circle, fill=black, inner sep=2pt,label={[label distance=0ex]225:$v_2$},label={90:{no edges here}}] (v2) at (4, 0) {};        
            \node[circle, fill=black, inner sep=2pt,label={[label distance=0ex]225:$v_3$}] (v3) at (8, 0) {};

            \node[circle, draw, inner sep=2pt,label={90:$w_1$}] (w1) at (2, 0) {};
            \node[circle, draw, inner sep=2pt,label={90:$w_2$}] (w2) at (6, 0) {};

            \draw[red,ultra thick] (-1.5,0) -- (v1);
            \draw[red,ultra thick] (v1) -- (w1);
            \draw[red,ultra thick] (w1) -- (v2);
            \draw[red,ultra thick] (v2) -- (w2);
            \draw[red,ultra thick] (w2) -- (v3);
            \draw[red,ultra thick] (v3) -- +(1.5,0);

            \draw (v1) -- +(0.5,1);
            \draw (v1) -- +(0,1);
            \draw (v1) -- +(-0.5,1);
            \draw (v1) -- +(0.5,-1);
            \draw (v1) -- +(0,-1);
            \draw (v1) -- +(-0.5,-1);
            
            \draw (w1) -- +(0.5,-1);
            \draw (w1) -- +(0,-1);
            \draw (w1) -- +(-0.5,-1);
            
            \draw (v2) -- +(0.5,-1);
            \draw (v2) -- +(0,-1);
            \draw (v2) -- +(-0.5,-1);
            
            \draw (v3) -- +(0.5,1);
            \draw (v3) -- +(0,1);
            \draw (v3) -- +(-0.5,1);
            \draw (v3) -- +(0.5,-1);
            \draw (v3) -- +(0,-1);
            \draw (v3) -- +(-0.5,-1);

    \end{tikzpicture}}

\bigskip

    \centering \subfloat[]{\label{subfig:bald-backbone}  \begin{tikzpicture}
            \node[circle, fill=black, inner sep=2pt,label={[label distance=0ex]270:$v_1$}] (v1) at (0, 0) {};        
            \node[circle, fill=black, inner sep=2pt,label={[label distance=0ex]270:$v_2$}] (v2) at (4, 0) {};        
            \node[circle, fill=black, inner sep=2pt,label={[label distance=0ex]270:$v_3$}] (v3) at (8, 0) {};

            \draw[red,ultra thick] (-1.5,0) -- (v1);
            \draw[->,red,ultra thick] (v1) -- (v2);
            \draw[red,ultra thick] (v2) -- (v3);
            \draw[red,ultra thick] (v3) -- +(1.5,0);
    \end{tikzpicture}}

    \caption{\protect\subref{subfig:bald-neighbors} A part of the initial tree
    $T_1$. The horizontal thick red edges are the spine, we assume that the
    labels of the black vertices fulfill $v_1<v_2$. We also assume that going
    clockwise around the black vertex $v_2$, the direct successor of the edge
    connecting $v_2$ to $w_2$ is the edge connecting $v_2$ to $w_1$. The neighbors of the white
    vertex $w_2$ are not shown.
    \newline
        \protect\subref{subfig:bald-backbone} The corresponding part of the backbone. The
oriented edge $(v_1,v_2)$ is bald.} 
\label{fig:bald}
\end{figure}

\subsubsection{Maximal backbone segments}
\label{sec:segments}

Let $A_r,A_{r-1},\dots,A_1,B_1,\dots,B_s\in\mathcal{B}$ (with $r,s\geq 1$) 
be backbone vertices such that:
\begin{itemize}
 \item $(A_{r},A_{r-1}), \ (A_{r-1},A_{r-2}),\ \dots, (A_2,A_1)$ are bald
edges in the backbone, \item $(A_1,B_1)$ is a hairy edge in the backbone,
\item  $(B_{s},B_{s-1}), \ (B_{s-1},B_{s-2}),\ \dots, (B_2,B_1)$ are bald
edges in the backbone,
\end{itemize}
see \cref{subfig:segment}. We will say that the above collection of edges,
which form a path connecting $A_r$ and $B_s$, constitutes a \emph{backbone
    segment}. In other words, a backbone segment consists of a single hairy edge
$(A_1,B_1)$ surrounded by two (possibly empty) oriented paths with the
endpoints $A_1$ and $B_1$ that consist of only bald edges.

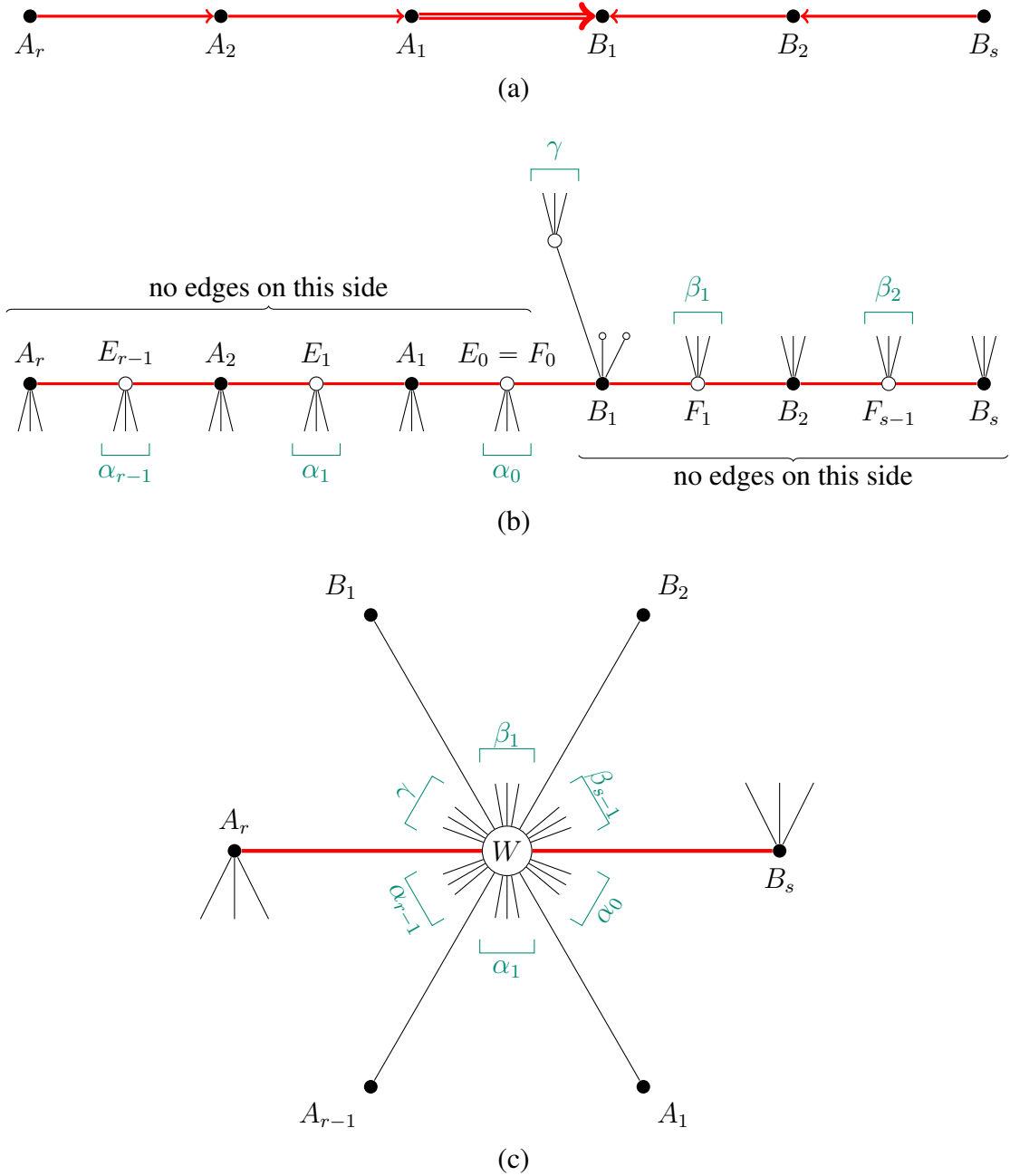
\begin{figure}
    \subfloat[]{\begin{tikzpicture}[scale=0.7]
            \node[circle, fill=black, inner sep=2pt,label={[label distance=0ex]270:$A_r$}] (a3) at (0, 0) {};        
            \node[circle, fill=black, inner sep=2pt,label={[label distance=0ex]270:$A_2$}] (a2) at (4, 0) {};        
            \node[circle, fill=black, inner sep=2pt,label={[label distance=0ex]270:$A_1$}] (a1) at (8, 0) {};
            
            \node[circle, fill=black, inner sep=2pt,label={[label distance=0ex]270:$B_1$}] (b1) at (12, 0) {};
            \node[circle, fill=black, inner sep=2pt,label={[label distance=0ex]270:$B_2$}] (b2) at (16, 0) {};
            \node[circle, fill=black, inner sep=2pt,label={[label distance=0ex]270:$B_s$}] (b3) at (20, 0) {};

            \draw[->,red,very thick] (a3) -- (a2);
            \draw[->,red,very thick] (a2) -- (a1);
            \draw[double,->,red,very thick] (a1) -- (b1);

            \draw[->,red,very thick] (b3) -- (b2);
            \draw[->,red,very thick] (b2) -- (b1);   
        \end{tikzpicture}
    \label{subfig:segment}
}
    
    \subfloat[]{\begin{tikzpicture}[scale=0.7]
            \node[circle, fill=black, inner sep=2pt,label={[label distance=0ex]90:$A_r$}] (a3) at (0, 0) {};        
            \node[circle, fill=black, inner sep=2pt,label={[label distance=0ex]90:$A_2$}] (a2) at (4, 0) {};        
            \node[circle, fill=black, inner sep=2pt,label={[label distance=0ex]90:$A_1$}] (a1) at (8, 0) {};
            
            \node[circle, fill=black, inner sep=2pt,label={[label distance=0ex]-90:$B_1$}] (b1) at (12, 0) {};
            \node[circle, fill=black, inner sep=2pt,label={[label distance=0ex]-90:$B_2$}] (b2) at (16, 0) {};
            \node[circle, fill=black, inner sep=2pt,label={[label distance=0ex]-90:$B_s$}] (b3) at (20, 0) {};

            \node[circle, draw=black, inner sep=2pt,label={[label distance=0ex]90:$E_{r-1}$}] (e2) at (2, 0) {};        
            \node[circle, draw=black, inner sep=2pt,label={[label distance=0ex]90:$E_1$}] (e1) at (6, 0) {};        
            \node[circle, draw=black, inner sep=2pt,label={[label distance=0ex]90:$E_0=F_0$}] (e0) at (10, 0) {};        
            
            \node[circle, draw=black, inner sep=2pt,label={[label distance=0ex]-90:$F_1$}] (f1) at (14, 0) {};        
            \node[circle, draw=black, inner sep=2pt,label={[label distance=0ex]-90:$F_{s-1}$}] (f2) at (18, 0) {};

            \draw (a3) -- +(-0.25,-1); 
            \draw (a3) -- +(0,-1); 
            \draw (a3) -- +(0.25,-1); 
            
            \draw (a2) -- +(-0.25,-1); 
            \draw (a2) -- +(0,-1); 
            \draw (a2) -- +(0.25,-1); 
            
            \draw (a1) -- +(-0.25,-1); 
            \draw (a1) -- +(0,-1); 
            \draw (a1) -- +(0.25,-1);

            \draw (e2) -- +(-0.25,-1); 
            \draw (e2) -- +(0,-1); 
            \draw (e2) -- +(0.25,-1); 
            
            \draw (e1) -- +(-0.25,-1); 
            \draw (e1) -- +(0,-1); 
            \draw (e1) -- +(0.25,-1); 
            
            \draw (e0) -- +(-0.25,-1); 
            \draw (e0) -- +(0,-1); 
            \draw (e0) -- +(0.25,-1);

            \node[circle, draw=black, inner sep=1pt] (X1) at (12,1) {};
            \node[circle, draw=black, inner sep=1pt] (X2) at (12.5,1) {};
            \node[circle, draw=black, inner sep=2pt] (X0) at (11,3) {};
            
            \draw (b1) -- (X0);
            \draw (b1) -- (X1);
            \draw (b1) -- (X2);
            
            \draw (X0) -- +(-0.25,1);
            \draw (X0) -- +(0,1);
            \draw (X0) -- +(0.25,1);
            
            \draw[PineGreen] (X0) +(-0.5,1.25) -- +(-0.5, 1.5) --  +(0.5, 1.5)  -- +(0.5, 1.25);
            \draw (X0) +(0,1.5) node[above] {\textcolor{PineGreen}{$\gamma$}};

            \draw (f1) -- +(-0.25,1); 
            \draw (f1) -- +(0,1); 
            \draw (f1) -- +(0.25,1); 
            
            \draw (f2) -- +(-0.25,1); 
            \draw (f2) -- +(0,1); 
            \draw (f2) -- +(0.25,1); 
            
            \draw (b2) -- +(-0.25,1); 
            \draw (b2) -- +(0,1); 
            \draw (b2) -- +(0.25,1); 
            
            \draw (b3) -- +(-0.25,1); 
            \draw (b3) -- +(0,1); 
            \draw (b3) -- +(0.25,1);

            \draw[red, very thick] (a3) -- (e2) -- (a2) -- (e1) -- (a1) -- (e0) -- (b1) -- (f1) -- (b2) -- (f2) -- (b3);
            
            \draw[PineGreen] (e0) +(-0.5,-1.25) -- +(-0.5, -1.5) --  +(0.5, -1.5)  -- +(0.5, -1.25);
            \draw (e0) +(0,-1.5) node[below] {\textcolor{PineGreen}{$\alpha_0$}};
            
            \draw[PineGreen] (e1) +(-0.5,-1.25) -- +(-0.5, -1.5) --  +(0.5, -1.5)  -- +(0.5, -1.25);
            \draw (e1) +(0,-1.5) node[below] {\textcolor{PineGreen}{$\alpha_1$}};
            
            \draw[PineGreen] (e2) +(-0.5,-1.25) -- +(-0.5, -1.5) --  +(0.5, -1.5)  -- +(0.5, -1.25);
            \draw (e2) +(0,-1.5) node[below] {\textcolor{PineGreen}{$\alpha_{r-1}$}};

            \draw[PineGreen] (f1) +(-0.5,1.25) -- +(-0.5, 1.5) --  +(0.5, 1.5)  -- +(0.5, 1.25);
            \draw (f1) +(0,1.5) node[above] {\textcolor{PineGreen}{$\beta_1$}};
            
            \draw[PineGreen] (f2) +(-0.5,1.25) -- +(-0.5, 1.5) --  +(0.5, 1.5)  -- +(0.5, 1.25);
            \draw (f2) +(0,1.5) node[above] {\textcolor{PineGreen}{$\beta_2$}};

            \draw [decorate, decoration = {brace}] (-0.5,1.5) --  node[above]{no edges on this side} (10.5,1.5);
            \draw [decorate, decoration = {brace}] (20.5,-1.5) --  node[below]{no edges on this side} (11.5,-1.5);

        \end{tikzpicture}
    \label{subfig:segment-neighbor}
}

\subfloat[]{\begin{tikzpicture}
        \node[circle, draw=black, inner sep=2pt] (w) at (0, 0) {$W$};

        \node[circle, fill=black, inner sep=2pt, label={90:$A_r$}] (Ar) at (-4, 0) {};  
        \draw (Ar) -- +(-0.5,-1);
        \draw (Ar) -- +(0,-1);
        \draw (Ar) -- +(+0.5,-1);
        
        \node[circle, fill=black, inner sep=2pt, label={-90:$B_s$}] (Bs) at (4, 0) {};  
        \draw (Bs) -- +(-0.5,1);
        \draw (Bs) -- +(0,1);
        \draw (Bs) -- +(+0.5,1);

        \node[circle, fill=black, inner sep=2pt, label={60:$B_2$}] (B2) at (60:4) {};  
        \node[circle, fill=black, inner sep=2pt, label={120:$B_1$}] (B1) at (120:4) {};  
        
        \node[circle, fill=black, inner sep=2pt, label={-60:$A_1$}] (A1) at (-60:4) {};  
        \node[circle, fill=black, inner sep=2pt, label={-120:$A_{r-1}$}] (A2) at (-120:4) {};

        \draw[ultra thick, red] (Ar) -- (w) -- (Bs);
        
        \draw (w) -- (B1);
        \draw (w) -- (B2);
        \draw (w) -- (A1);
        \draw (w) -- (A2);
        
        \draw (w) -- (20:1);
        \draw (w) -- (30:1);
        \draw (w) -- (40:1);
        
        \draw (w) -- (80:1);
        \draw (w) -- (90:1);
        \draw (w) -- (100:1);
        
        \draw (w) -- (140:1);
        \draw (w) -- (150:1);
        \draw (w) -- (160:1);
        
        \draw (w) -- (-20:1);
        \draw (w) -- (-30:1);
        \draw (w) -- (-40:1);
        
        \draw (w) -- (-80:1);
        \draw (w) -- (-90:1);
        \draw (w) -- (-100:1);
        
        \draw (w) -- (-140:1);
        \draw (w) -- (-150:1);
        \draw (w) -- (-160:1);

        \begin{scope}[shift={(30:1.5)}, scale=0.4, rotate=-60 ]
            \draw[PineGreen] (-1,-0.5) -- (-1,0) -- (1,0) -- (1,-0.5); 
            \draw (0,0.5) node[rotate=-60] {\textcolor{PineGreen}{$\beta_{s-1}$}};
        \end{scope}    
        
        \begin{scope}[shift={(90:1.5)}, scale=0.4, rotate=0 ]
            \draw[PineGreen] (-1,-0.5) -- (-1,0) -- (1,0) -- (1,-0.5); 
            \draw (0,0.5) node[rotate=0] {\textcolor{PineGreen}{$\beta_{1}$}};
        \end{scope}    
        
        \begin{scope}[shift={(150:1.5)}, scale=0.4, rotate=60 ]
            \draw[PineGreen] (-1,-0.5) -- (-1,0) -- (1,0) -- (1,-0.5); 
            \draw (0,0.5) node[rotate=60] {\textcolor{PineGreen}{$\gamma$}};
        \end{scope}    
        
        \begin{scope}[shift={(-30:1.5)}, scale=0.4, rotate=-120 ]
            \draw[PineGreen] (-1,-0.5) -- (-1,0) -- (1,0) -- (1,-0.5); 
            \draw (0,0.6) node[rotate= 60] {\textcolor{PineGreen}{$\alpha_0$}};
        \end{scope}    
        
        \begin{scope}[shift={(-90:1.5)}, scale=0.4, rotate=-180 ]
            \draw[PineGreen] (-1,-0.5) -- (-1,0) -- (1,0) -- (1,-0.5); 
            \draw (0,0.6) node[rotate=0] {\textcolor{PineGreen}{$\alpha_1$}};
        \end{scope}    
        
        \begin{scope}[shift={(-150:1.5)}, scale=0.4, rotate=-240 ]
            \draw[PineGreen] (-1,-0.5) -- (-1,0) -- (1,0) -- (1,-0.5); 
            \draw (0,0.6) node[rotate=-60] {\textcolor{PineGreen}{$\alpha_{r-1}$}};
        \end{scope}

    \end{tikzpicture}
    \label{subfig:segment-fold}
}

    \caption{\protect\subref{subfig:segment} A maximal backbone segment in the
    backbone. The oriented edge $(A_1,B_1)$ is the only hairy edge on this
    picture, for this reason it was decorated by a double line. The labels of
    the vertices fulfill $A_r < \cdots < A_2 < A_1 < B_1 > B_2 > \cdots > B_s$.
    \protect\subref{subfig:segment-neighbor} The corresponding tree segment;
    for simplicity only the direct neighborhood of the spine was shown. Note
    that in the exceptional case $r=1$ there is no restriction on the edges
    surrounding the vertex $A_1=A_r$, and in the exceptional case $s=1$  there
    is no restriction on the edges surrounding the vertex $B_1=B_s$.
    \protect\subref{subfig:segment-fold} The outcome of the algorithm $\CT$
    applied to this tree segment. Only the neighbors of the white spine vertex
    $W$ are shown.}
 
\end{figure}

\medskip

We say that a backbone segment is \emph{maximal} if cannot be extended by adding an
additional bald edge at either of its endpoints. It is easy to check that any
two maximal backbone segments are disjoint. 

\medskip

We claim that each oriented edge $e=(v_1,v_2)$ of the backbone belongs to some
maximal backbone segment. Indeed, if $e$ is hairy, it forms a very short
backbone segment with $r=s=1$; by extending this backbone segment we end up
with a maximal backbone segment. On the other hand, if $e$ is bald we can
follow the orientations of the edges and traverse the backbone as long as we
visit only bald edges $(v_2,v_3), \ (v_3,v_4),\dots, (v_{l-1},v_l)$. In a
finite number of steps our walk will terminate; there is a number of cases
that give a specific reason why the walk terminated.

Firstly, we could have encountered one of the endpoints of the backbone,
i.e., $v_l\in\{1,n\}$. The case $v_l=1$ is not possible because $v_l=1$ carries
the minimal label, so $v_l< v_{l-1}$, which contradicts the assumption that
$(v_{l-1}, v_l)$ is one of the oriented edges of the backbone. The other case
$v_l=n$ is also not possible because, by definition, the oriented edge
$(v_{l-1}, n)$ is not bald.

The second possibility is that the other edge attached to the vertex $v_l$ (it
is an oriented edge of the form $f=(v_l,v_{l+1})$ or $f=(v_{l+1}, v_l)$ for
some $v_{l+1}\neq v_{l-1})$) is hairy. In this case the maximal backbone segment
containing the hairy edge $f$ contains the edge $e$, as required.

The third case is that the other edge attached to the vertex $v_l$ is a bald
edge $f=(v_{l+1},v_l)$ with $v_{l+1}\neq v_l$ with the \emph{wrong} orientation
of the edge, which prevents traversing $f$. This would imply that the backbone
vertex $v_l$ has two incoming bald edges, and therefore that the black vertex
$v_l$ (regarded as a vertex in the original tree $T_1$) has degree $2$. On the
other hand, the degree of the vertex $v_l$ is equal to $a_{v_l}\geq 3$ by
\eqref{eq:a-and-b}, which leads to a contradiction and completes the proof.

\medskip

It follows that the maximal backbone segments provide a partition of the set of
backbone edges.

\subsubsection{Tree segments} 

The collection of all endpoints of the maximal backbone segments (for the
maximal backbone segment depicted on  \cref{subfig:segment} these endpoints are
denoted by $A_r$ and $B_s$) can be used to split the initial tree $T_1$ into a
number of connected components that will be called \emph{tree segments}. Each
such an endpoint (with the exception of the vertices $1$ and $n$) belongs to
two such tree segments; in order to split the neighbors of such boundary
vertices between the two tree segments we use the convention shown on
\cref{subfig:segment-neighbor}: each tree segment contains these non-spine edges
attached at the endpoint that are on \emph{``the left-hand side''} of the
spine (from the viewpoint of a person who looks along the spine, in the
direction of the respective endpoint).

The spine naturally splits each tree segment into two parts. We refer to them
as \emph{the upper part} and \emph{the lower part} according to the convention
from \cref{subfig:segment-neighbor}.

\subsection{Action of the algorithm on a tree segment}
\label{sec:spine-neighborhood}

In the following we will
investigate the action of the algorithm on some tree segment.
We will use the notations from \cref{subfig:segment-neighbor}.

We denote by $E_0=F_0$ the white vertex between $A_1$ and $B_1$, and for
$i\in\{1,\dots,r-1\}$ we denote by $E_i$ the white vertex between $A_i$ and
$A_{i+1}$. The assumption about the orientations of the edges in the backbone
segment implies that the vertex $A_{i+1}$ is the anchor of the cluster $E_i$;
this implies that the bend operation $\Bend_{A_{i+1},A_i}$ is one of the
operations performed in the iteration of the main loop for the cluster $C=E_i$.
As a result, the white spine vertices $E_i$ and $E_{i-1}$ are merged into a
single white spine vertex. On the other hand, after this bend operation is
performed, the black spine vertex $A_i$ no longer belongs to the spine.

For $i\in\{1,\dots,s-1\}$ we denote by $F_i$ the white vertex between $B_i$ and
$B_{i+1}$. Similarly as above, the bend operation $\Bend_{B_{i+1},B_i}$ is one
of the operations performed during the spine treatment part of the algorithm.
As a result the white spine vertices $F_i$ and $F_{i-1}$ are merged into a
single white spine vertex, and the black spine vertex $B_i$ no longer belongs
to the spine.

As a result of the above operations, all white spine vertices
$E_0,E_1,\dots,E_{r-1},F_1,\dots,F_{s-1}$ in the considered tree segment are
merged into a single white spine vertex that will be denoted by $W$; see
\cref{subfig:segment-fold}. In the following we will describe the neighbors of
$W$ in the output tree $T_2$.

\medskip

In the output tree $T_2$, the vertex $W$ is connected to two black spine
vertices: $A_r$ and $B_s$, which were the endpoints of the backbone segment
that we consider. With the notations of \cref{subfig:segment-fold} we will
first concentrate on the \emph{`bottom'} part of $W$. More specifically, going
counterclockwise around the vertex $W$, after the vertex $A_r$ and before the
vertex $B_s$, we encounter (among other black vertices) the vertices $A_{r-1},
A_{r-2}, \dots, A_2, A_1$, arranged in this exact order; see
\cref{subfig:segment-fold}. In order to find the remaining black vertices that
are interlaced in between we may revisit \cref{sec:children-of-organic}; a
large part of that section is applicable also in the current context.

More specifically, for $i\in\{0,\dots,r-1\}$ we denote by $\alpha_i$ the part
of the original tree $T_1$ that is attached to the spine at the white vertex
$E_i$; see \cref{subfig:segment-neighbor}. Just like in
\cref{sec:first-pruning} we perform the first pruning; for the purposes of this
first pruning we declare that the parent of the white vertex $E_{i}$ (which is
the root of the tree $\alpha_i$) is equal to $A_{i+1}$. Then, just like in
\cref{sec:second-pruning}, we perform the second pruning and then the folding
(see \cref{sec:folding}). The outcome of this procedure is a sequence of black
vertices: these are exactly the neighbors of the white vertex $W$ that appear
(in the clockwise cyclic order) after the vertex $A_{i+1}$ and before $A_i$; we
use the convention that $A_0=B_s$; see \cref{subfig:segment-fold}.
This completes the description of the \emph{bottom} part of the vertex $W$.

\medskip

The \emph{upper} part of the vertex $W$ is slightly more complicated. We denote
by $\gamma$ the part of the tree $T_1$ that has the leftmost non-spine child
of $B_1$ as the root; see \cref{subfig:segment-neighbor}. The key difference
between the bottom and the upper part lies in the fact that one of the
operations performed in the iteration of the main loop in the spine treatment
part of the algorithm for the cluster $C=E_0$ is the bend operation
$\Bend_{A_1,B_1}$. This operation plants the tree $\gamma$ in the vertex $W$,
in the counterclockwise order after the vertex $B_1$ and before the vertex
$A_r$; see \cref{subfig:segment-fold}. Then the usual pruning and folding
procedures are applied to $\gamma$. The remaining black vertices in the upper
part of $W$ correspond to the trees $\beta_1,\dots,\beta_{s-1}$; see
\cref{subfig:segment-fold}.

\medskip

As we can see, any operation performed for the clusters that form the given
tree segment does not affect the other segments. For this reason it is possible
to study the action of the algorithm $\CT$ on each tree segment separately.

\section{The inverse bijection}
\label{sec:inverse}

We will complete the proof of \cref{thm:thm1} by constructing explicitly the
inverse map $\CT^{-1}$. The starting point of the algorithm is a Stanley tree
$T_2$ of type $(b_1,\dots,b_n)$, cf.~\cref{sec:trees}. Our goal is to turn this
tree into a minimal factorization $(\sigma_1,\dots,\sigma_n)\in
\mathcal{C}_{a_1,\dots ,a_n}$, where the numbers $a_1,\dots,a_n$ are specified
in \cref{thm:thm1}.

\medskip

Similarly as in \cref{sec:algorithm}, the path in $T_2$ between the two black
vertices with the labels $1$ and $n$ will be called \emph{the spine}.

\subsection{Which white vertices are artificial?}
\label{subsec:artificial}

Recall that the notion of \emph{artificial} vertices was introduced in the
context of the jump operation at the beginning of \cref{sec:output-of-jump}. In
the first step of our algorithm $\CT^{-1}$, for a given a Stanley tree $T_2$ we
will guess which white vertices are artificial, i.e., which vertices are the
result of the jump operation. The remaining part of the current section is
devoted to the details of this issue.

\begin{lem}
    \label{lem:artificial-or-organic} 
    Let $T_2$ be a Stanley tree that is an output of the algorithm $\CT$
for some input data and let $v$ be one of its white non-spine vertices. The
vertex $v$ is organic if and only if its neighbor with the maximal label is
equal to the parent of $v$.
\end{lem}

With the terminology from page \pageref{page:belongs}, since in the output tree
$T_2$ all edges are dashed, each white vertex is attracted to exactly one neighbor,
and the above result can be rephrased as follows: $v$ is organic if and only if
$v$ is attracted to its parent.

\begin{proof}
Since artificial vertices arise only as a result of a jump operation, the
vertex $v$ of the tree $T_2$ is artificial if and only if there is a black
vertex $y$ in the tree $T_2$ that is a child of $v$ and such that during the
execution of the algorithm $\CT$ there was a jump operation of the form
$\Jump_{\cdot,y}$.

\medskip

Consider the case when $v$ is artificial; immediately after this jump operation
is performed, the newly created vertex $v$ (on \cref{subfig:jumpA2} it is the
vertex between $j$ and $y$) is attracted to $y$ because $j< y$ by
\cref{lem:what-is-j}. Moreover, from the proof of the \cref{lm:lem1} it follows
that after this jump operation $\Jump_{\cdot,y}$ is performed, the set of
vertices that are attracted to $y$ does not change until the very end of the
algorithm $\CT$. In this way $v$ is attracted to one of its children, hence not
to its parent, as required.

\medskip

Consider the opposite case when $v$ is organic; let $y$ be some child of $v$.
It follows that during the execution of the  algorithm $\CT$ there was a bend
operation $\Bend_{\cdot,y}$. By examining \cref{fig:exA} it follows that
immediately that after this bend operation was performed, the vertex $v$ is not
attracted to $y$. Again, from the proof of the \cref{lm:lem1} it follows that
after this bend operation $\Bend_{\cdot,y}$ is performed, the set of vertices
that are attracted to $y$ does not change until the very end of the algorithm
$\CT$. In this way we proved that $v$ is not attracted to any of its children,
hence it is attracted to its parent, as required.
\end{proof}

\subsection{Recovering the cycles away from the spine} 
\label{sec:away-begins}

\newcommand{\BB}{{B}}  
\newcommand{\WW}{{W}}
\newcommand{\JJ}{{J}}
\newcommand{\maks}{{Y}}

Each black vertex $\BB\in\{1,\dots,n\}$ in the initial tree $T_1$ corresponds
to the cycle $\sigma_\BB$. There is a canonical bijective correspondence
between the set of black vertices in the tree $T_1$ and the set of black
vertices in the output tree~$T_2$, therefore $\BB$ can be identified with a
black vertex in the output tree $T_2$. In order to find the inverse map
$\CT^{-1}$ we need to recover the cycle $\sigma_\BB$ based on the information
contained in the tree $T_2$.

We will start with the assumption that the black vertex $\BB$ of $T_2$ is away
at least by two edges from the spine; this assumption will be used also in
\crefrange{sec:away-begins}{sec:away-ends} below.

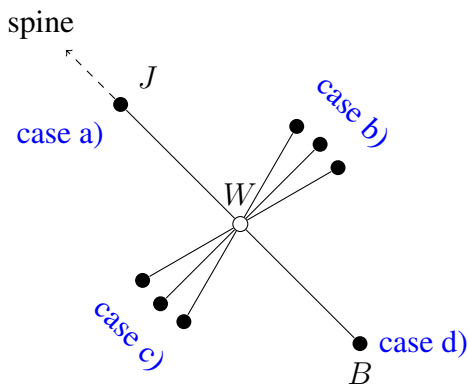
\begin{figure}
    \centering \begin{tikzpicture}[scale=1.5]
        \node[circle, draw=black, inner sep=2pt, label={92:$W$}] (w) at (0,0) {};
        
        \node[circle, fill=black, inner sep=2pt] (b1) at (45:1) {};
        \node[circle, fill=black, inner sep=2pt] (b2) at (30:1) {};
        \node[circle, fill=black, inner sep=2pt] (b3) at (60:1) {};
        
        \node[circle, fill=black, inner sep=2pt] (b4) at (180+45:1) {};
        \node[circle, fill=black, inner sep=2pt] (b5) at (180+30:1) {};
        \node[circle, fill=black, inner sep=2pt] (b6) at (180+60:1) {};
        
        \node[circle, fill=black, inner sep=2pt,label=45:$J$, label=225:{\textcolor{blue}{case a)}}] (b7) at (90+45:1.5) {};
        \node (b7prim) at (90+45:2.5) {spine};
        \draw[dashed,->] (b7) -- (b7prim);

        \node[circle, fill=black, inner sep=2pt,label=-90:$B$, label=0:{\textcolor{blue}{case d)}}] (b8) at (-45:1.5) {};
        
        \node[rotate=-45] at (45:1.4) {\textcolor{blue}{case b)}};
        \node[rotate=-45] at (180+45:1.4) {\textcolor{blue}{case c)}};
        
        \draw (w) -- (b1);
        \draw (w) -- (b2);
        \draw (w) -- (b3);
        \draw (w) -- (b4);
        \draw (w) -- (b5);
        \draw (w) -- (b6);
        \draw (w) -- (b7);
        \draw (w) -- (b8);
    \end{tikzpicture}
\caption{The four cases necessary for determining the cycle $\sigma_\BB$. The
    black vertex $B$ is assumed to be away from spine by at least two edges. The
    vertex $\WW$ is the parent of $\BB$. The vertex $\JJ$ is the parent of $\WW$. We
    define $\maks$ to be the neighbor of $\WW$ with the biggest label. The case a)
    (see \cref{sec:inverse-A}) corresponds to the case when $\maks=\JJ$. The case b)
    (see \cref{sec:inverse-B}) corresponds to the case when $\maks$ is located
    (going clockwise around the vertex~$\WW$) after the vertex $\BB$ and before
    $\JJ$, while the case c) (see \cref{sec:inverse-C}) corresponds to the case
    when $\maks$ is located after $\JJ$ and before $\BB$.     The case d) (see
    \cref{sec:inverse-D}) corresponds to the case when $\maks=\BB$. }
\label{fig:cases-abcd}
\end{figure}

Let $\WW$ denote the parent of $\BB$, and let $\JJ$ denote the parent of $\WW$.
We define $\maks$ to be the neighbor of $\WW$ with the biggest label. Our
prescription for finding $\sigma_\BB$ will depend on the position of $\maks$
with respect to $\JJ$ and $\BB$; see \cref{fig:cases-abcd}.

\subsection{Case (a): $\maks=\JJ$.}
\label{sec:inverse-A}

By \cref{lem:artificial-or-organic}, in this case $\WW$ is an organic white
vertex.

\subsubsection{The labels $E_2,\dots,E_d$}
\label{sec:labels-e2-ed}

The proof of \cref{lm:lem1} shows that the black vertex~$\BB$ lost exactly two
white neighbors that were attracted to $\BB$ when we performed the
operation~$\Bend_{\cdot,\BB}$ and it never gained new ones (we showed that
$b_\BB=\Numberb_\BB^{\text{final}}=\Numberb_\BB^{\text{initial}}-2=a_\BB-2$).
Using this information and basing on \cref{fig:ex} (with the notations of this
figure we have $y=\BB$, $v_1=\WW$, $x=\JJ$) we conclude that in the output tree
$T_2$ the labels of the edges attracted to $\BB$ (with the notations of
\cref{subfig:exB} these are the edges $E_3,\dots,E_d$) all belong to the cycle
$\sigma_\BB$. The label of the edge that is in the direction of the spine (this
edge is not attracted to $\BB$; with the notations of \cref{subfig:exB} this is
the edge $E_2$) also belongs to the cycle $\sigma_\BB$; see \cref{subfig:exB}.
In this way we identified so far all elements of the cycle
$\sigma_\BB=(E_1,\dots,E_d)$, with the exception of the label~$E_1$.

\subsubsection{How to find the missing label $E_1$?}
\label{sec:how-to-find} 

The information about the missing label $E_1$ is contained in one of the
connected components of the two pruning procedures considered in
\cref{sec:children-of-organic} (more precisely, in the tree $T$ that contains
the black vertex $\BB$). Clearly, $E_1$ is the label on the edge that connects
the vertex $\BB$ to its parent in $T$. Regretfully, we do not have (yet) the
access to this tree $T$.  On the bright side, we do have the access to the
outcome of the folding procedure (\cref{sec:folding}) applied to this tree~$T$.
The discussion from \cref{sec:children-of-organic} shows that this outcome is
just $\WW$ (which is a white organic vertex) in the output tree $T_2$ as well
as its all children, together with its parental edge. In the following we will
describe how to reverse the folding procedure and, in this way, to recover the
tree $T$ (in the example from \cref{sec:folding} it is the tree
\cref{subfig:bendingin}) from the outcome of folding
(\cref{subfig:bendingout}).

\subsubsection{The greedy algorithm} 

\label{sec:greedy} 

Our first step is to identify the children $V_1,\dots,V_m$ (listed in the
counterclockwise order) of the root of $T$;
on the example from \cref{subfig:bendingin} these  vertices \begin{equation}
    \label{eq:example-V}
    V=(V_1,\dots,V_5)=(14,
18, 20, 21, 30)
\end{equation} 
are drawn in blue. In the following we will treat the aforementioned outcome of the folding as
the list $L=(L_1,\dots,L_\ell)$ of the children of $\WW$ in $T_2$,
listed in the counterclockwise order. In the example from
\cref{subfig:bendingout} we have
\begin{equation}
    \label{eq:example-L}
    L=(\textbf{14},2,11,\textbf{18},8,4,6,7,10,12,3,5,9,\textbf{20},\textbf{21},13,16,\textbf{30},15,17).
\end{equation}
With boldface we indicated the entries of $V$; on \cref{subfig:bendingout}
these boldface vertices are also drawn in blue.

Since the list $L$ can be seen as an outcome
of the depth-first search, it has the following structure:
\begin{multline}
    \label{eq:recursive-structure}
     L=\big(  V_1, \text{(the black descendants of $V_1$)}, 
     \\ V_2, \text{(the black descendants of $V_2$)},\quad \dots,\\
     V_m, \text{(the black descendants of $V_m$)} \big). 
\end{multline}
By \cref{rem:niceorder} each black descendant of $V_i$ carries
a label that is smaller than $V_i$; furthermore $V=(V_1,\dots,V_m)$ is an
increasing sequence. It follows that $V=(L_{i_1},\dots,L_{i_m})$ is an increasing
subsequence of $L$, which can be found by the following \emph{greedy algorithm}:
\begin{quotation}
    \emph{We set $i_1=1$ so that $V_1:=L_1$.}

   \smallskip

    \emph{If $i_k$ was already calculated, we define $i_{k+1}$ as the smallest
    number $l\in\{i_k+1,\dots,\ell\}$ with the property that $L_l >
    V_k=L_{i_k}$; we set $V_{k+1}:=L_{i_{k+1}}$. If such a number $l$ with this
    property does not exist, the algorithm terminates.}
\end{quotation}
In other words, we start with the empty list $V=\emptyset$ and read the list $L$.
If the just read entry of $L$ is greater than the last entry of $V$ (or if $V$ is
empty), we append it to $V$. We leave it as an exercise to the Reader to verify
that this greedy algorithm applied to $L$ from \eqref{eq:example-L} indeed gives
\eqref{eq:example-V}.

\subsubsection{Recovering the black vertices in the tree $T$}
\label{sec:recovering}

The above greedy algorithm identifies just the children of the root of $T$. In
order to recover the full structure of the tree $T$ we use recursion, as follows.
Equation \eqref{eq:recursive-structure} shows that the elements of the sequence
$V$ act like separators between the list of the black descendants of $V_1$, the
list of the black descendants of $V_2$, \dots; in particular we are able to
find explicitly such a list of the black descendants of any vertex $V_i$. In
the example from \eqref{eq:example-L} we have
\begin{align*}
    \text{(the descendants of $V_1=14$)} &=    (2,11), \\
    \text{(the descendants of $V_2=18$)} &=     (8,4,6,7,10,12,3,5,9), \\
    \text{(the descendants of $V_3=20$)} &=     \emptyset,\\ 
       \text{(the descendants of $V_4=21$)} &=  (13,16),\\
        \text{(the descendants of $V_5=30$)} &= (15,17).
\end{align*}
Each such a list of the descendants of $V_i$ is again the outcome of the
depth-first-search, this time restricted to the black descendants of the vertex
$V_i$.

By applying the above greedy algorithm to the list of the descendants of the
vertex $V_i$ we identify the grandchildren of $V_i$, as well as the lists of
the descendants of each of these grandchildren. For example, the greedy
algorithm applied to the descendants of $V_2$
\[ (\textbf{8},4,6,7,\textbf{10},\textbf{12},3,5,9) \]
shows that the vertex $V_2=18$ has three grandchildren: $8$, $10$, and $12$, as well
as gives the list of the descendants for each of them.

By applying the above procedure recursively, we recover the structure of the
tree $T$ (with the edge labels removed).

\subsubsection{Labels come back}
\label{sec:labels}
\label{sec:how-to-find-end}

Our goal is to use the information about the outcome of the folding in order to
recover the edge labels in the original tree $T$.

Essentially, the above reconstruction of the tree $T$ implies that we know
which elementary folding operations from \cref{fig:bending2} were performed. We
will use the notations from this figure. The only missing component in order to
fully revert such an elementary folding operation at a black vertex $b$ is the
label of the edge $f$ that connected $b$ to its parent $w$ in the tree $T$. We
will denote by $v$ the parent of $b$ in the output tree $T_2$; the vertex $v$
corresponds to the root of the folded version of $T$. Note that both $w$ as
well as $v$ denote the parent of $b$; the difference lies in a different tree
being considered. In fact, in order to solve the problem from
\cref{sec:how-to-find} of finding the missing label $E_1$, we are interested in
the special case when $b=B$ and $w=W$; with these notations $f=E_1$ and $e=E_2$.
Fortunately, before the folding was applied, the edge $f$ and the edge $g$ that
connects $w$ to its parent belonged to the same cluster, so they carried the
same label. The problem is therefore reduced to finding explicitly the location
of the edge $g$ in the outcome of the folding.

\medskip

Firstly, consider the generic case (shown on \cref{fig:bending2}) when the
vertex $w$ is not the root of the tree $T$. We denote by $b_2$ the black vertex
that is the parent of $w$ in $T$. 
In this case, after the folding at the vertex $b_2$ is
performed, the label of the edge $g$ will be stored in the edge connecting the
vertex $b_2$ with its parent. This property will not be changed by further
foldings of the tree~$T$.

To summarize: in order to find the label $f$ of the edge in the initial tree
$T_1$ that connected the black vertex $b$ to its parent $w$ one should apply
the following procedure. We apply recursively the greedy algorithm from
\cref{sec:recovering} to the children (in the output tree $T_2$) of the white
vertex $v$ that is the parent of $b$; this vertex $v$  and its children
correspond to the folded version of $T$. In this way we find the vertex $b_2$
that in the input tree $T_1$ is the grandparent of $b$. The desired label $f$
is carried by the edge connecting $b_2$ with its parent $v$ in the output tree
$T_2$. 

\medskip

Consider now the exceptional case when the white vertex $w$ is the root of the
tree $T$. In fact, based on the information contained in the tree $T_2$ it is
easy to check whether this exceptional case holds true by applying the greedy
algorithm from \cref{sec:greedy} to the children of $v$ and checking if the
vertex $b$ belongs to the list $(V_1,\dots,V_m)$.

In this case the edge $h$ as well as its endpoints should be removed from
\cref{fig:bending2} since they do not belong to the tree~$T$. The vertex $w$ is
attached to its parent by the parental edge~$g$. This edge will not be modified
by further steps of the algorithm; as a result our desired label $f$ is stored
in the folded version of $T$ in the parental edge of the root, and hence in the
output tree $T_2$ it is still stored in the edge $g$ connecting $v=w$ with its
parent.

\subsection{Cases (b), (c), (d): $\maks\neq \JJ$}

By \cref{lem:artificial-or-organic}, in these three cases $\WW$ is an
artificial white vertex thus the discussion from \cref{sec:children-artificial}
and \cref{fig:artificial-anatomy} in particular are applicable here. We will
study each of the cases (b), (c) and (d) in more detail in the following.

\subsection{Case (b): going counterclockwise around $\WW$, the vertex $\maks$
    is after $\BB$ and before~$\JJ$} 

\label{sec:inverse-B}

With the notations of \cref{fig:artificial-anatomy}, our vertex $B$ is one of
the \emph{blue neighbors} of $w$, located after $j$ and before $y$ (going
counterclockwise around $w$); see \cref{fig:artificial-anatomy}.

Since $B\neq Y$ it follows that we \emph{did not} perform a jump operation of
the form $\Jump_{\cdot, B}$, hence we performed a bend operation of the form
$\Bend_{\cdot,B}$. As a consequence, the discussion from
\cref{sec:labels-e2-ed} is applicable also in our context; as a consequence we
have found the labels $E_2,\dots,E_d$ that contribute to the cycle
$\sigma_B=(\sigma_1,\dots,\sigma_d)$. As before, the remaining difficulty is to
find the label $E_1$.

If we keep only the vertex $w$ and the aforementioned blue neighbors of the
vertex $w$, and declare that the edge between $w$ and its parent $j$ is the
parental edge of the vertex $w$, we obtain a tree with $w$ as a root. This tree
$T'$ is the outcome of the folding of the blue tree on
\cref{subfig:jumpA,subfig:scjumpA}. It follows that the algorithm from
\crefrange{sec:how-to-find}{sec:how-to-find-end} applied to $T'$ gives the blue
tree on \cref{subfig:jumpA,subfig:scjumpA} and the desired edge $E_1$ can be
recovered.

\subsection{Case (c): going counterclockwise around $\WW$, the vertex $\maks$
    is after $\JJ$ and before~$\BB$}
\label{sec:inverse-C}

This case is fully analogous to the case b) considered in \cref{sec:inverse-B}.
The only difference is that instead of \emph{blue} one should keep only the
\emph{red} neighbors of $w$, and one should take the edge connecting $w$ with
$y=Y$ as the parental edge of $w$.

\subsection{Case (d): $\BB=\maks$}
\label{sec:inverse-D}
\label{sec:away-ends}

In the case $\BB=\maks$ we deduce that the vertex $\WW$ was created by a jump
operation of the form $\Jump_{\cdot,\BB}$.

\subsubsection{The labels $E_2,\dots,E_d$} 

The following discussion is fully analogous to the one from
\cref{sec:labels-e2-ed}. \cref{lm:lem1} and its proof imply that the black
vertex $\BB$ lost exactly two white neighbors from
$\Numberb_\JJ^{\text{initial}}$ when we performed the operation
$\Jump_{\cdot,\BB}$ and never gained new ones. Using this information and
\cref{fig:jump,fig:scjump} we conclude that in the output tree $T_2$ the
Stanley edge labels of the black vertex $\BB$ all belong to
cycle~$\sigma_\BB=(E_1,\dots,E_d)$ (with the notations of
\cref{subfig:jumpB,subfig:scjumpB} these are the edges $E_3,\dots,E_d$).

The label of the edge $E_2$ is easy to recover: by \cref{fig:jump,fig:scjump}
we conclude that in the output tree $T_2$ the label $E_2$ is carried by the
edge connecting $\WW$ and $\JJ$. The only remaining difficulty is to find the
label $E_1$.

\subsubsection{How to find the missing label $E_1$?}

The following discussion is somewhat analogous to the one from 
\crefrange{sec:how-to-find}{sec:how-to-find-end}.

With the notations from \cref{fig:jump,fig:scjump} the missing label $E_1$ was
also carried by the edge of the tree $T_1$ connecting $v_1$ with its parent
$x$. This edge belongs also to the subtree $T$; regretfully we do not have
(yet) access to this tree.

On the bright side, we do have the access to the outcome of the algorithm from
\cref{sec:black-dfs}. The organic children of $J$ can be treated as separators
that split the artificial children of $J$ into a number of lists. Each such a
list was generated by the depth-first search algorithm applied to the subtree
attached to some organic child of $J$. The discussion from \cref{sec:greedy} is
also applicable in our context with some minor modifications: the role of
\cref{rem:niceorder} is played now by \cref{rem:niceorder2}, we now list the
vertices in the \emph{clockwise} order, and the greedy algorithm aims to find a
\emph{decreasing} subsequence. This, together with the analogue of
\cref{sec:recovering} allows us to recover the structure of the tree $T$
(without the edge labels yet).

In this way, using only the information contained in the tree $T_2$, we can
find the vertex $x$ that in the tree $T$ was the grandparent of $Y$.  There
are the following two cases.

Firstly, if $x=J$ is the root of the tree $T$, the missing label $E_1$ is
carried in the tree $T_2$ by the edge connecting $J$ with its organic child
$W$.

Secondly, if $x\neq J$ is not the root of $T$, the desired label $E_1$ appears
in the pair \eqref{eq:pair} together with the vertex $x$. This means that the
label of $E_1$ is carried  by the first edge on the (very short) path that
connects the vertex $J$ with $x$ in $T_2$.

\subsection{Recovering the cycles near the spine} 

Our goal is to recover the cycle $\sigma_B$ in the special case when $B$ in the
output tree $T_2$ is either a spine vertex or a neighbor of a white spine
vertex. We will keep this assumption also in
\crefrange{sec:cases-begin}{sec:cases-end}.

In order to achieve this goal we need to recover the past of the white
spine vertex (or the two white spine vertices) that is a neighbor of $B$.

\subsubsection{The fake symmetry}

From the discussion in \cref{sec:anatomy-spine} it follows that each white
spine vertex of the output tree $T_2$ corresponds to some tree segment. At the
first sight it might seem that the definition of a tree segment (or a backbone
segment) has a \emph{rotational symmetry} that corresponds to a rotation by
$180^\degree$ of \cref{subfig:segment,subfig:segment-neighbor} and reversing
the roles played by the sequences $A_1,\dots,A_r$ and $B_1,\dots,B_s$. This
false impression may be reinforced by the apparent $180^\degree$ symmetry of
the output of the algorithm depicted on \cref{subfig:segment-fold}. In fact,
this false symmetry of \cref{subfig:segment-fold} is problematic for our
purposes because for a given white spine vertex $W$ of the output tree $T_2$ we
need to know which of its two black spine neighbors plays the role of $A_r$
and which one plays the role of $B_s$.

In order to resolve this ambiguity we split the set of neighbors of the white
§spine vertex $W$ into two \emph{halves}: each half consists of an endpoint of a
spine edge and (going counterclockwise) the subsequent non-spine neighbors of
$W$, up until the other spine neighbor; see \cref{subfig:segment-fold}. With
these notations the set $\{A_1, B_1\}$ is equal to the set of two neighbors of
$W$ with the maximal labels in each of the halves respectively. The requirement
that $A_1<B_1$ gives us the unique way to fit \cref{subfig:segment-fold} into
the the neighborhood of $W$.

\subsubsection{The first application of the greedy algorithm}

By applying the greedy algorithm from \cref{sec:greedy} we can reconstruct a
large part of the information depicted on \cref{subfig:segment-neighbor}:
recover the black vertices $A_1,\dots,A_r$ and $B_1,\dots,B_s$ as well as the
folded versions of the trees $\alpha_1,\dots,\alpha_{r-1},
\beta_1,\dots,\beta_{s-1}, \gamma$.

Our prescription for finding the cycle $\sigma_B$ will depend on the location
of the black vertex on \cref{subfig:segment-neighbor}.

\subsection{Case (i): never-spine vertex}
\label{sec:cases-begin}

Consider the case when $B$ is a non-spine black vertex that is adjacent to a
white spine vertex $W$ and $B\notin\{A_1,\dots,A_r,B_1,\dots,B_s\}$; in other
words $B$ in the input tree $T_1$ was not a spine vertex. This means that $B$
belongs to a folded version of one of the trees $\alpha_0,\dots,\alpha_{r-1},
\beta_1,\dots,\beta_{s-1},\gamma$ (\cref{subfig:segment-neighbor}) and with the
available information we can pinpoint this folded tree. The algorithm presented
in \crefrange{sec:labels-e2-ed}{sec:how-to-find-end} is also applicable in this
context. Note, however that this algorithm takes as an input a white vertex as
a root surrounded by black vertices, \emph{together with the parental edge of
    the root}, so we need to specify in each case this parental edge. For the tree
$\alpha_i$ (with $i\in\{1,\dots,r-1\}$) it is the edge connecting $W$ with
$A_{i+1}$; for the tree $\beta_i$ (with $i\in\{0,\dots,s-1\}$) it is the edge
connecting $W$ with $B_{i+1}$; for the tree $\gamma$ it is the edge connecting
$W$ with $B_{1}$.

\subsection{Case (ii): post-spine vertices}
\label{sec:case-ii}

Consider the case when $B\in\{A_1,\dots,A_{r-1}, \allowbreak B_1,\dots,
\allowbreak B_{s-1} \}$; in other words $B$ was a black spine vertex in the
tree $T_1$ but it is not a spine vertex in the tree~$T_2$. It follows that the
corresponding cycle can be written in the form $\sigma_B=(E_1,\dots,E_d)$,
where $E_1$ and $E_2$ are the spine edges surrounding the vertex $B$; see
\cref{subfig:segment-neighbor}. From the proof of \cref{lm:lem1} (Case 2) it
follows that the vertex $B$ lost two edges that are attracted to $B$, namely
the two spine edges $E_1$ and $E_2$. In this way we recovered the edges
$E_3,\dots,E_d$ (these are the edges that are attracted to $B$ in $T_2$) and
the remaining difficulty is to find $E_1$ and $E_2$.

The key observation is that all edges surrounding the vertex $E_i$ (with
$i\in\{0,\dots,r-1\}$) in the tree $T_1$ carried the same label; after
performing the algorithm $\CT$ this label is stored in the edge connecting $W$
with $A_{i+1}$. Similarly, all edges surrounding the vertex $F_i$ (with
$i\in\{1,\dots,s-1\}$) in the tree $T_1$ carried the same label; after
performing the algorithm $\CT$ this label is stored in the edge connecting $W$
with $B_{i+1}$. In this way we are able to recover the labels of all spine
edges in the segment and, as a result, to recover the labels $E_1$ and $E_2$.

\subsection{Case (iii): spine vertices}
\label{sec:cases-end}

Consider now the case when $B$ is a black spine vertex in the tree $T_2$. In
the generic case when $B\notin\{1,n\}$ is not one of the endpoints of the
spine, the vertex $B$ lies on the interface between two tree segments. The
cycle $\sigma_B$ consists of: the edge labels lying on one side of the spine,
followed by a single spine edge label, after which come the edge labels lying
the other side of the spine, and the other spine edge label. From the proof of
\cref{lm:lem1} (Case 2) it follows that the vertex $B$ lost precisely these two
spine edges; 
fortunately the discussion from
\cref{sec:case-ii} shows how to recover them.

\section{Conclusion. Proof of \cref{thm:thm1}}
\label{sec:conclusion}

In \cref{sec:inverse} we constructed a map 
\[ \CT^{-1}\colon 
  \operatorname{Image}(\CT) \to \mathcal{C}_{a_1,\dots ,a_n} \]
 that has the property that 
\[ \CT^{-1} \circ \CT  = 
\operatorname{id} \colon \mathcal{C}_{a_1,\dots ,a_n} \to \mathcal{C}_{a_1,\dots ,a_n} \]
is the identity map. In particular, it follows that the map
\[ \CT\colon \mathcal{C}_{a_1,\dots ,a_n} \to \mathcal{T}_{b_1,\dots ,b_n} \]
is injective. It remains to show that it is also surjective.

\medskip

A fast way to do this is to compare the cardinalities of the respective sets by
\cref{cor:number-of-stanley-trees} and \eqref{eq:formula-for-c}. This proof has
a minor disadvantage of being not sufficiently bijective.

\medskip

An alternative but more challenging strategy is to notice that the map $
\CT^{-1}$ from \cref{sec:inverse} is well defined on $\mathcal{T}_{b_1,\dots
    ,b_n}$. This time, however, one has to check that the image of $\CT^{-1}$ on
this larger domain is still a subset of $\mathcal{C}_{a_1,\dots ,a_n}$. The next
step is to verify that
\[  \CT \circ \CT^{-1}  = 
\operatorname{id} \colon  \mathcal{T}_{b_1,\dots ,b_n}\to \mathcal{T}_{b_1,\dots ,b_n} \]
is the identity map.
This method of proof does not create real difficulties, but it is somewhat
lengthy. In order to keep this paper not excessively long we decided to omit
this more bijective approach.

\section*{Acknowledgments} 

Research of \trokowska was supported by Grant 2017/26/A/ST1/00189 of Narodowe
Centrum Nauki. \Sniady is grateful to Max Planck Institute for Mathematics in
Bonn for its hospitality and financial support.

\printbibliography

\end{document}